\makeatletter

\documentclass[
		a4paper,	
		oneside,
		article, 	
		english,
		10pt,
		leqno,
		]{memoir}

\newif\ifprenumber 
\prenumbertrue

\ifluatex
	\usepackage{
		fontspec, 		
		babel,			
		lualatex-math,	
		}
\else
	\ifxetex
		\usepackage{fontspec}
		\usepackage[babelshorthands=true]{polyglossia}
		}
		\setmainlanguage{danish}
		\setotherlanguage{english}
	\else
		\usepackage[utf8]{inputenc} 
		\usepackage[T1]{fontenc} 	
		\usepackage{babel}
		\usepackage[
			noDcommand,
			slantedGreeks,
			]{kpfonts}

		\usepackage{newunicodechar}
		
		\def\totheminusone¹{^{-1}}
		
		\newunicodechar{⁻}{\totheminusone}
	\fi
\fi

\title{Homotopy limits in the category of dg-categories in terms of \( \mathup{A}_{\infty} \)-comodules}
\date{\today}
\author{Sergey Arkhipov and Sebastian Ørsted}

\newcommand\parsepdfdata{
	\hypersetup{
		pdfauthor={Sergey Arkhipov and Sebastian Ørsted},
		pdftitle={Homotopy limits in the category dg-categories in terms of A\_∞-comodules},
	}
}

\usepackage{
	graphicx, 	
	mathtools, 	
	etoolbox,	
	textcomp,	
	microtype, 	
}
\usepackage{
	lettrine,
	xspace,
	xparse,
	}

\usepackage[shortlabels]{enumitem}

\usepackage[draft]{fixme}

\usepackage[autostyle,german=guillemets,english=american]{csquotes}
\DeclareQuoteStyle[myguillemotsstyle]{danish}
{\guillemotright}{\guillemotleft}
{\textquoteleft}{\textquoteright} 

\ExecuteQuoteOptions{danish=myguillemotsstyle}

\MakeAutoQuote{“}{”}

\MakeAutoQuote{‘}{’}

\MakeAutoQuote{»}{«}

\usepackage{tikz}

\usetikzlibrary{babel}

\usetikzlibrary{cd}

\tikzcdset{
  arrow style=tikz,
}

\tikzset{
  >/.tip={Stealth[length=2.9pt, width=4.4pt, inset=1.8pt]},
  tikzcd left hook/.tip={Hooks[
	  left,
	  length=2pt,
	  width=5.5pt,
	]},
  iso/.style={
    every to/.append style={
      edge node={
        node [sloped, allow upside down]{
          \raisebox{0.1em}[0pt][0pt]{\ensuremath{\sim}}
        }
      }
    }
  },
  iso'/.style={
    every to/.append style={
      edge node={
        node [sloped, allow upside down]{
          \raisebox{-0.6em}[0pt][0pt]{\ensuremath{\sim}}
        }
      }
    }
  },
  symbol/.style={
      draw=none,
      every to/.append style={
        edge node={node [sloped, allow upside down, auto=false]{$#1$}}}
  },
}

\newcommand\pushout{\arrow[ul,phantom,"\lrcorner",very near start]}
\newcommand\pullback{\arrow[dr,phantom,"\ulcorner",very near start]}

\newcommand\invdot{\mathrlap{.}}
\newcommand\invcomma{\mathrlap{,}}

\usepackage[backend=biber,citestyle=authoryear-ibid,bibstyle=authoryear,
	uniquename=init,giveninits=true,autolang=hyphen]{biblatex}
\nobibintoc 

\usepackage[thmmarks,amsmath]{ntheorem}

\makeatletter
\newtheoremstyle{notitle}
{\item[\hskip\labelsep \theorem@headerfont ##2\theorem@separator]}%
{\item[\hskip\labelsep \theorem@headerfont ##2\theorem@separator\ ##3\theorem@separator]}

\newtheoremstyle{notitlebreak}
{\item[\rlap{\vbox{\hbox{\hskip\labelsep \theorem@headerfont ##2\theorem@separator}\hbox{\strut}}}]}%
{\item[\rlap{\vbox{\hbox{\hskip\labelsep \theorem@headerfont ##2\theorem@separator\ ##3\theorem@separator}\hbox{\strut}}}]}

\newtheoremstyle{prenumber}%
  {\item[\hskip\labelsep \theorem@headerfont ##2\theorem@separator\ ##1\theorem@separator]}%
  {\item[\hskip\labelsep \theorem@headerfont ##2\theorem@separator\ ##3\theorem@separator]}

\newtheoremstyle{prenumberbreak}%
  {\item[\rlap{\vbox{\hbox{\hskip\labelsep \theorem@headerfont 
          ##2\theorem@separator\ ##1\theorem@separator}\hbox{\strut}}}]}%
  {\item[\rlap{\vbox{\hbox{\hskip\labelsep \theorem@headerfont 
          ##2\theorem@separator\ ##3\theorem@separator}\hbox{\strut}}}]}

\newtheoremstyle{customplain}%
  {\item[\hskip\labelsep \theorem@headerfont ##1\ ##2\theorem@separator]}%
  {\item[\hskip\labelsep \theorem@headerfont ##3\ ##2\theorem@separator]}
  
\newtheoremstyle{custombreak}%
  {\item[\rlap{\vbox{\hbox{\hskip\labelsep \theorem@headerfont
          ##1\ ##2\theorem@separator}\hbox{\strut}}}]}%
  {\item[\rlap{\vbox{\hbox{\hskip\labelsep \theorem@headerfont 
          ##3\ ##2\theorem@separator}\hbox{\strut}}}]}

\newtheoremstyle{postnumpara}%
  {\item[\hskip\labelsep \theorem@headerfont ##2\theorem@separator]}%
  {\item[\hskip\labelsep \theorem@headerfont ##3\ ##2\theorem@separator]}
  
\newtheoremstyle{postnumparabreak}%
  {\item[\rlap{\vbox{\hbox{\hskip\labelsep \theorem@headerfont
          ##2\theorem@separator}\hbox{\strut}}}]}%
  {\item[\rlap{\vbox{\hbox{\hskip\labelsep \theorem@headerfont 
          ##3\ ##2\theorem@separator}\hbox{\strut}}}]}

\newtheoremstyle{nonumberplainflex}
	{\item[\theorem@headerfont\hskip\labelsep ##1\theorem@separator]}%
	{\item[\theorem@headerfont\hskip \labelsep ##3\theorem@separator]}

\newtheoremstyle{nonumberbreakflex}%
	{\item[\rlap{\vbox{\hbox{\hskip\labelsep \theorem@headerfont
		##1\theorem@separator}\hbox{\strut}}}]}%
	{\item[\rlap{\vbox{\hbox{\hskip\labelsep \theorem@headerfont 
		##3\theorem@separator}\hbox{\strut}}}]}

\newcommand\danishtheorems{
	\def\@nthm@theorem{Sætning}
	\def\@nthm@proposition{Proposition}
	\def\@nthm@corollary{Korollar}
	\def\@nthm@lemma{Lemma}
	\def\@nthm@claim{Påstand}
	\def\@nthm@definition{Definition}
	\def\@nthm@example{Eksempel}
	\def\@nthm@exercise{Opgave}
	\def\@nthm@remark{Bemærkning}
	\def\@nthm@proof{Bevis}
}

\newcommand\englishtheorems{
	\def\@nthm@theorem{Theorem}
	\def\@nthm@proposition{Proposition}
	\def\@nthm@corollary{Corollary}
	\def\@nthm@lemma{Lemma}
	\def\@nthm@claim{Claim}
	\def\@nthm@definition{Definition}
	\def\@nthm@example{Example}
	\def\@nthm@exercise{Exercise}
	\def\@nthm@remark{Remark}
	\def\@nthm@proof{Proof}
}

\englishtheorems

\theoremseparator{.}

\theorembodyfont{\normalfont}

\ifprenumber
	\theoremstyle{notitle}
\else
	\theoremstyle{postnumpara}
\fi

\newtheorem{para}[equation]{}

\ifprenumber
	\theoremstyle{prenumber}
\else
	\theoremstyle{customplain}
\fi

\theorembodyfont{\itshape}

\newtheorem{theorem}[equation]{\@nthm@theorem}
\newtheorem{corollary}[equation]{\@nthm@corollary}
\newtheorem{proposition}[equation]{\@nthm@proposition}
\newtheorem{lemma}[equation]{\@nthm@lemma}

\theorembodyfont{\normalfont}

\theoremsymbol{\ensuremath{\triangle}}
\newtheorem{remark}[equation]{\@nthm@remark}

\theoremsymbol{\ensuremath{\scriptstyle\bigcirc}}
\newtheorem{example}[equation]{\@nthm@example}

\theoremstyle{notitle}

\theoremstyle{nonumberplainflex}
\theoremsymbol{\ensuremath{\square}}
\theoremheaderfont{\itshape\bfseries}
\newtheorem{proof}{\@nthm@proof}

\theoremstyle{nonumberplainflex}
\theoremheaderfont{\normalfont\bfseries}
\theoremsymbol{}

\theorembodyfont{\normalfont}

\theorembodyfont{\itshape}
\newtheorem{theoremnonumber}{\@nthm@theorem}

\theorembodyfont{\normalfont}

\theoremsymbol{\ensuremath{\triangle}}

\theoremsymbol{\ensuremath{\scriptstyle\bigcirc}}

\theoremsymbol{}
\theoremheaderfont{\normalfont\bfseries}

\ifprenumber
	\theoremstyle{notitlebreak}
\else
	\theoremstyle{postnumparabreak}
\fi

\ifprenumber
	\theoremstyle{prenumberbreak}
\else
	\theoremstyle{custombreak}
\fi

\theorembodyfont{\itshape}
\newtheorem{theorembreak}[equation]{\@nthm@theorem}
\newtheorem{corollarybreak}[equation]{\@nthm@corollary}
\newtheorem{propositionbreak}[equation]{\@nthm@proposition}

\theorembodyfont{\normalfont}

\theoremsymbol{\ensuremath{\triangle}}

\theoremsymbol{\ensuremath{\scriptstyle\bigcirc}}

\theoremsymbol{}

\theoremstyle{notitlebreak}

\theoremstyle{nonumberbreakflex}
\theoremsymbol{\ensuremath{\square}}
\theoremheaderfont{\itshape\bfseries}

\theoremstyle{nonumberbreakflex}
\theoremheaderfont{\normalfont\bfseries}
\theoremsymbol{}
\theorembodyfont{\normalfont}

\theorembodyfont{\itshape}

\theorembodyfont{\normalfont}

\theoremsymbol{\ensuremath{\triangle}}

\theoremsymbol{\ensuremath{\scriptstyle\bigcirc}}

\usepackage{varioref}
	
\usepackage[hidelinks]{hyperref}
	
\usepackage[nameinlink]{cleveref}

	\crefname{para}{}{}
\crefname{theorem}{Theorem}{Theorems}
\crefname{definition}{Definition}{Definitions}
\crefname{example}{Example}{Examples}
\crefname{corollary}{Corollary}{Corollaries}
\crefname{proposition}{Proposition}{Propositions}
\crefname{lemma}{Lemma}{Lemmata}
\crefname{remark}{Remark}{Remarks}
\crefname{claim}{Claim}{Claims}
\crefname{exercise}{Exercise}{Exercises}
\crefname{textexercise}{Exercise}{Ecercises}

\crefalias{parabreak}{para}

\crefalias{theorembreak}{theorem}

\crefalias{definitionbreak}{definition}

\crefalias{examplebreak}{example}

\crefalias{corollarybreak}{corollary}

\crefalias{propositionbreak}{proposition}

\crefalias{lemmabreak}{lemma}

\crefalias{claimbreak}{claim}

\crefalias{remarkbreak}{remark}

\crefalias{textexercisebreak}{textexercise}

\newcommand{\crefrangeconjunction}{--} 

\newlist{paralist}{enumerate}{2}
\setlist[paralist]{
	label=\textup{(\roman*)},
	ref=\thepara\textup{(\roman*)},
	resume,
}
\crefalias{paralisti}{theorem}

\newlist{theoremlist}{enumerate}{2}
\setlist[theoremlist]{
	label=\textup{(\roman*)},
	ref=\thepara\textup{(\roman*)},
	resume,
}
\crefalias{theoremlisti}{theorem}

\newlist{corollarylist}{enumerate}{2}
\setlist[corollarylist]{
	label=\textup{(\roman*)},
	ref=\thepara\textup{(\roman*)},
	resume,
}
\crefalias{corollarylisti}{corollary}

\newlist{propositionlist}{enumerate}{2}
\setlist[propositionlist]{
	label=\textup{(\roman*)},
	ref=\thepara\textup{(\roman*)},
	resume,
}
\crefalias{propositionlisti}{proposition}

\newlist{lemmalist}{enumerate}{2}
\setlist[lemmalist]{
	label=\textup{(\roman*)},
	ref=\thepara\textup{(\roman*)},
	resume,
}
\crefalias{lemmalisti}{lemma}

\newlist{definitionlist}{enumerate}{2}
\setlist[definitionlist]{
	label=\textup{(\roman*)},
	ref=\thepara\textup{(\roman*)},
	resume,
}
\crefalias{definitionlisti}{definition}

\newlist{textexerciselist}{enumerate}{2}
\setlist[textexerciselist]{
	label=\textup{(\roman*)},
	ref=\thepara\textup{(\roman*)},
	resume,
}
\crefalias{textexerciselisti}{textexercise}

\newlist{remarklist}{enumerate}{2}
\setlist[remarklist]{
	label=\textup{(\roman*)},
	ref=\thepara\textup{(\roman*)},
	resume,
}
\crefalias{remarklisti}{remark}

\newlist{examplelist}{enumerate}{2}
\setlist[examplelist]{
	label=\textup{(\roman*)},
	ref=\thepara\textup{(\roman*)},
	resume,
}
\crefalias{examplelisti}{example}

\newlist{exerciselist}{enumerate}{2}
\setlist[exerciselist]{
	label=\textup{(\roman*)},
	ref=\thepara\textup{(\roman*)},
	resume,
}
\crefalias{exerciselisti}{exercise}

\newlist{claimlist}{enumerate}{2}
\setlist[claimlist]{
	label=\textup{(\roman*)},
	ref=\thepara\textup{(\roman*)},
	resume,
}
\crefalias{claimlisti}{claim}

\newcounter{localreftmpcnt} %
\newcommand\alphsubformat[1]{\textup{(\roman{#1})}} 
\newcommand\localref[2][\alphsubformat]{%
	\ifcsname r@#2@cref\endcsname
	\cref@getcounter {#2}{\mylabel}%
	\setcounter{localreftmpcnt}{\mylabel}%
	\hyperref[#2]{%
		\alphsubformat{localreftmpcnt}%
	}%
	\else ?? \fi}   
	
	\parsepdfdata 	

\input{_commands/_commands_1}

\numberwithin{equation}{section}

\setupclass{category}{output=algmap}

\renewcommand\categoryformat[1]{\textup{\textsf{#1}}}
\newcommand\categoryformatmath[1]{\mathsf{#1}}

\newcommand\adjunctionarrow{\rightleftarrows}

\setupclass{var}{
	singlekeys={
		{smash}{return,command=\noexpand\smash},
		{limits}{command=\mathop,return,symbolputright=\limits},
		{bin}{command=\mathbin},
		{operator}{command=\mathop},
		{reduce}{command=\noexpand\overline},
		{op}{upper=\mathup{op}},
		{*res}{upper=*,nopar,output=map},
		{!kan}{lower=!,nopar,output=map},
		{*kan}{lower=*,nopar,output=map},
		{qpull}{pull,nopar,output=algmap},
		{qpush}{push,nopar,output=algmap},
		{q!}{lower=!,nopar,output=algmap},
		{dg}{symbolputleft=\noexpand\categoryformat{dg}},
		{gr}{symbolputleft=\noexpand\categoryformat{gr}},
		{graded}{upper=\noexpand\categoryformat{gr}},
		{star}{return,symbolputleft=\noexpand\categoryformatmath{*}},
		{nonunital}{lower=\noexpand\categoryformat{nu}},
		{nu}{nonunital},
		{noncounital}{lower=\noexpand\categoryformat{ncu}},
		{ncu}{noncounital},
		{homotopycounital}{cupper=\noexpand\categoryformat{hcu}},
		{hcu}{homotopycounital},
		{formal}{cupper=\noexpand\categoryformat{formal}},
		{dg}{i=\noexpand\vphantom{a}\noexpand\smash{\mathup{dg}},output=alg},
		{hc}{i=\mathup{hc},output=algmap},
		{top}{i=\mathup{top},output=classset},
		{bar}{command=\noexpand\widebar},
		{dgbar}{dg,bar},
		{fat}{i=\noexpand\vphantom{a}\noexpand\smash{\mathup{fat}}},
		{rep}{command=\widetilde,ifoutput=true,outputoptions={smash}},
		{z0}{command=\Zdgcat},
		{h0}{command=\Hdgcat},
		{proj}{modelstructure={\modelstructureproj}},
		{inj}{modelstructure={\modelstructureinj}},
		{reedy}{modelstructure={\modelstructurereedy}},
		{reedydirect}{lower=+},
		{reedyinverse}{lower=-},
		{reedy+}{reedydirect},
		{reedy-}{reedyinverse},
		{lgrade*}{lgrade=\noexpand\smallbullet},
		{rgrade*}{rgrade=\noexpand\smallbullet},
		{red}{command=\noexpand\overline},
		{aug}{lower=\categoryformatmath{aug}},
		{coaug}{lower=\categoryformatmath{coaug}},
		{augment}{upper=+},
		{cocomplete}{upper=\noexpand\categoryformatmath{c}},
		{comor}{upper=\#},
		{boundary}{symbolputleft=\partial},
	},
	keyvals={
		{powering}{upper={\keyvalue}},
		{latching}{symbolputleft=L_{\keyvalue}},
		{latch}{latching={\keyvalue}},
		{matching}{symbolputleft=M_{\keyvalue}},
		{match}{matching={\keyvalue}},
		{tens}{upper={\tens\keyvalue}},
		{rmod}{symbolputright={/\keyvalue}},
		{lmod}{symbolputleft={\keyvalue\backslash}},
		{mod}{rmod={\keyvalue}},
		{rrmod}{symbolputright={/\!\!/\keyvalue}},
		{llmod}{symbolputleft={\keyvalue\backslash\!\!\!\backslash}},
		{leftrestrict}{symbolputleft={{}_{\keyvalue}}},
		{rightrestrict}{lower={\keyvalue}},
		{leftres}{leftrestrict={\keyvalue}},
		{rightres}{rightrestrict={\keyvalue}},
		{shift}{return,symbolputright={{}\lbrack\keyvalue\rbrack}},
		{weirdindex}{lower={(\keyvalue)}},
		{iterate}{upper={(\keyvalue)}},
		{over}{symbolputright={/\keyvalue}},
		{res}{return,symbolputright=\vert,lower={\keyvalue}},
		{to}{upper={\keyvalue}},
		{lgrade}{symbolputleft={}_{\keyvalue}},
		{rgrade}{lower={\keyvalue}},
		{setobj}{lower={\keyvalue}},
		{linearize}{symbolputleft={\keyvalue}},
		{modelstructure}{lower={\keyvalue}},
		{commacat}{return,symbolputright={/\keyvalue}},
		{lin}{linearize={\keyvalue}},
		{shift}{return,symbolputright={{}{}{}{}\lbrack\keyvalue\rbrack}},
		{equivar}{upper={\keyvalue}},
		{dom}{symbolputleft={\keyvalue\backslash}},
	},
	output=var,
}

\setupclass{homset}{
	singlekeys={{enrich}{command=\noexpand\underline}},
}

\newvar\valpha{\alpha}
\newvar\vbeta{\beta}
\newvar\vgamma{\gamma}
\newvar\va{a}
\newvar\vb{b}
\newvar\vc{c}
\newvar\vd{d}
\newvar\ve{e}
\newvar\vf{f}
\newvar\vg{g}
\newvar\vh{h}
\newvar\vv{v}
\newvar\vA{A}
\newvar\vB{B}
\newvar\vC{C}
\newvar\vD{D}
\newvar\vE{E}
\newvar\vX{X}
\newvar\vY{Y}
\newvar\vZ{Z}
\newvar\vR{R}
\newvar\vS{S}

\newmap\miota{\iota}
\newmap\diagonalmap{\Delta}
\newmap\mvarphi{\varphi}
\newmap\malpha{\alpha}
\newmap\mbeta{\beta}
\newmap\mtheta{\theta}
\newmap\mpsi{\psi}
\newmap\meta{\eta}
\newmap\mepsilon{\epsilon}
\newmap\ms{s}
\newmap\momega{\omega}
\newmap\mr{r}
\newmap\mW{W}
\newmap\comult{\Delta}
\newmap\counit{\varepsilon}
\newmap\coaug{\eta}
\newmap\aug{\varepsilon}
\newmap\ca{\mathup{ca}}
\newmap\mult{m}
\newmap\mR{R}
\newmap\mS{S}
\newmap\mT{T}
\newmap\mpr{\mathup{pr}}
\newmap\mpi{\pi}

\newclassset\setA{A}
\newclassset\setB{B}
\newclassset\setC{C}
\newclassset\setD{D}
\newclassset\setX{X}
\newclassset\setY{Y}
\newclassset\setZ{Z}
\newclassset\setU{U}
\newclassset\setV{V}
\newclassset\setW{W}
\newclassset\setS{S}
\newclassset\setT{T}
\newclassset\setP{P}

\newalg\algG{G}
\newalg\algH{H}
\newalg\algA{A}
\newalg\algB{B}
\newalg\algC{C}
\newalg\algM{M}
\newalg\algN{N}
\newalg\algP{P}
\newalg\algV{V}
\newalg\algW{W}
\newalg\algU{U}
\newalg\algK{K}

\newalg\fieldk{k}
\newalg\ringk{k}

\setupclass{cohomology}{
	parent=alg,
	singlekeys={
		{hochschield}{symbolputleft={H\!}},
		{hoch}{hochschield},
	},
}

\newmapclass{algmap}[
	parent=alg,output=algmap,
]

\newalgmap\tensoralg{T}[
	singlekeys={
		{red}{command=\noexpand\widebar},
	},
	output=complex,
]
\newalgmap\tensorcoalg{T}[parent=tensoralg,upper=c]

\setupobject\Bar{
	output=alg,
	singlekeys={
		{red}{command=\noexpand\overline},
		{plus}{upper=+},
	},
}

\setupobject\Cob{parent=Bar}

\newmap\difforms{\Omega}

\newalgmap\cone{C}

\newcategory\catqcoh{QCoh}

\newcategory\catcoh{Coh}

\newcategory\catdgqcoh{\textbf{QCoh}}
\newcategory\catdgcoh{\textbf{Coh}}

\newcategory\catdesc{Desc}

\setupobject\catdgalg{
	singlekeys={
		{le0}{upper={\le0}},
		{ge0}{upper={\ge0}},
		{com}{lower=\categoryformatmath{com}},
	},
}

\setupobject\der{
	singlekeys={
		{dg}{lower=\mathup{dg}},
	},
}

\newcategory\cathot{Hot}

\newcategory\catquiv{Quiv}[
	argsinglekeys={
		{E}{\setobjE},
		{F}{\setobjF},
	},
]

\newcategory\catgrquiv{grQuiv}[parent=catquiv]
\newcategory\catdgquiv{dgQuiv}[parent=catquiv]
\newcategory\catstquiv{\cocatstar Quiv}[parent=catquiv]

\newcategory\catsch{Sch}

\newmap\groupoidinv{\iota}

\newcategory\catstalg{\cocatstar Alg}
\newcategory\catgralg{grAlg}
\newcategory\catcoalg{Coalg}
\newcategory\catstcoalg{\cocatstar Coalg}[cocomplete]
\newcategory\catgrcoalg{grCoalg}[cocomplete]
\newcategory\catdgcoalg{dgCoalg}[cocomplete]

\newcategory\catstcoalgnoc{\cocatstar Coalg}

\newcommand\ainfty{\ensuremath{\mathup{A}_{\infty}}}
\newcommand\ainftysf{\ensuremath{\mathsf{A}_{\infty}}}

\newcategory\catcocat{Cocat}[upper=\categoryformatmath{c}]
\newcategory\catdgcocat{dgCocat}[parent=catcocat]
\newcategory\catgrcocat{grCocat}[parent=catcocat]
\newcategory\catstcocat{\cocatstar Cocat}[parent=catcocat]

\newcategory\catcofun{Cofun}
\newcommand\cocatstar{\ensuremath{\mathsf{*}}}

\newcategory\catstcofun{\cocatstar Cofun}
\newcategory\catgrcofun{grCofun}
\newcategory\catdgcofun{dgCofun}

\newcategory\catainftycoalg{\( \categoryformatmath{A_{\infty}} \)Coalg}
\newcategory\catainftycocat{\( \categoryformatmath{A_{\infty}} \)Cocat}
\newcategory\catainftycomod{\( \categoryformatmath{A_{\infty}} \)Comod}
\newcategory\catainftycontra{\( \categoryformatmath{A_{\infty}} \)Contra}
\newcategory\cathom{Hom}
\newmap\coder{\operatorname{Coder}}
\newmap\Der{\operatorname{Der}}

\setupclass{category}{
	singlekeys={
		{infty}{lower=\noexpand\mathsf{\infty}},
	},
}

\newcategory\comcoh{Coh}[output=complex]

\newalgmap\catdelta{\mathbf{\Deltaup}}[
	singlekeys={
		{plus}{lower=+},
	},
]

\newcategory\catsset{SSet}
\newcategory\catvarsset{\catset}[diag=\catdelta]

\newcomplexclass{cosimplicial}[
	parent=var,
	gradingpos=upper,
	keyvals={
			{degreedefault}{sd={\keyvalue},ifoutput=true},
			{sk}{symbolputleft={\noexpand\operatorname{Sk}^{\keyvalue}}},
	},
	singlekeys={
			{deltaplus}{symbolputleft={\miota[*res]}},
	},
	output=classset,
]

\newcomplexclass{simplicial}[
	parent=cosimplicial,
	gradingpos=lower,
	singlekeys={
		{realize}{
			return,leftspar=\lvert,rightspar=\rvert,spar,nowoutput=classset,
		},
		{boundary}{
			symbolputleft=\partial,
		},
	},
	output=classset,
]

\newsimplicial\ssetX{X}
\newsimplicial\ssetY{Y}
\newsimplicial\ssetZ{Z}
\newsimplicial\ssetK{K}
\newsimplicial\ssetL{L}

\newsimplicial\ssetGamma{\Gamma}

\newdelim\geo{\lvert}{\rvert}[
	output=classset,
	singlekeys={
		{dg}{ outputoptions={lower=\noexpand\mathup{dg}} },
	},
]

\newcomplexclass{cosimplicialdgcat}[parent=cosimplicial,output=algop]

\newcosimplicial\simp{\Deltaup}[
	d={\!},iffirstd=true,
	output=simplicial,
]

\newcosimplicial\catiso{Iso}[parent=category]

\newdelim\ordset{[}{]}[
	output=alg,
	singlekeys={
		{plus}{outputoptions={upper=+}},
		{+}{plus},
	},
]

\newcategory\catsimp{Simp}

\setupclass{algop}{
	keyvals={
		{arg}{argwithoutkeyval={\keyvalue}},
	},
}

\newalgop\Prod{\prod}
\newalgop\Coprod{\coprod}
\newalgop\dirsum{\bigoplus}
\newautomathoperator\pushout{\smallcoprod}
\newautomathoperator\varpushout{\mathbin{\textstyle\coprod}}
\newautomathoperator\pullback{\smallprod}
\newautomathoperator\varpullback{\mathbin{\textstyle\prod}}
\newautomathoperator\bigpullback{\prod}
\newautomathoperator\varbigpullback{\bigtimes}

\newalgop\Eq{\operatorname{Eq}}[par]
\newalgop\Coeq{\operatorname{Coeq}}[par]

\newalgop\catob{\operatorname{Ob}}
\newvar\catobpoint{\bullet}
\newvar\objectpoint{*}

\newalgmap\unitmap{\eta}
\newdelim\dgcatdeg{\lvert}{\rvert}[output=alg]

\newalgmap\dgcatA{\mathscr{A}}
\newalgmap\dgcatB{\mathscr{B}}
\newalgmap\dgcatC{\mathscr{C}}
\newalgmap\dgcatD{\mathscr{D}}
\newalgmap\dgcatE{\mathscr{E}}
\newalgmap\dgcatH{\mathscr{H}}
\newalgmap\dgcatunit{\mathbf{1}}
\newalgmap\catunit{\mathbf{1}}
\newvar\finalobject{\ast}

\setupobject\catsset{
	singlekeys={
		{quillen}{modelstructure={\modelstructurequillen}},
	},
}

\newclassset\setobjE{\mathbb{E}}
\newclassset\setobjF{\mathbb{F}}
\newsimplex\basisvector{e}

\newcategory\catkEmodkE{\(\ringk[setobj=\setobjE]\)-mod-\(\ringk[setobj=\setobjE]\)}

\newalg\ringkE{\ringk}[setobj=\setobjE]

\newalgmap\catA{\mathscr{A}}
\newalgmap\catB{\mathscr{B}}
\newalgmap\catC{\mathscr{C}}
\newalgmap\catD{\mathscr{D}}
\newalgmap\catE{\mathscr{E}}
\newalgmap\catV{\mathscr{V}}
\newalgmap\catGamma{\Gamma}

\newtuple\innerhom{[}{,}{]}{\ldots}
\newtuple\powering{\{}{,}{\}}{\ldots}

\setupclass{tuple}{
	singlekeys={
		{left}{outputoptions={lower=l}},
		{right}{outputoptions={lower=r}},
	},
}

\newalgop\Zdgcat{Z}[d=0,nopar]
\newalgop\Hdgcat{H}[d=0,nopar]

\newclassset\setgencof{I}
\newclassset\setgentrivcof{J}

\newclassset\setcof{\operatorname{Cof}}
\newclassset\setfib{\operatorname{Fib}}
\newclassset\setweakeqs{\mathscr{W}}

\newvariableclass{modelstructure}[parent=classset]

\newmodelstructure\modelstructureproj{\operatorname{Proj}}
\newmodelstructure\modelstructureinj{\operatorname{Inj}}
\newmodelstructure\modelstructuretab{\operatorname{Tab}}
\newmodelstructure\modelstructurejoyal{\opereratorname{Joyal}}
\newmodelstructure\modelstructurequillen{\operatorname{Quillen}}
\newmodelstructure\modelstructurereedy{\operatorname{Reedy}}

\newcommand\noloc{%
  \nobreak
  \mspace{6mu plus 1mu}
  {:}
  \nonscript\mkern-\thinmuskip
  \mathpunct{}
  \mspace{2mu}
}

\newmapclass{kanextension}[
	parent=algop,
	output=algmap,
]

\newmap\mprojcof{\mathscr{F}}[output=map]
\newmap\mconst{\mathup{const}}

\newkanextension\mLan{\operatorname{Lan}}
\newkanextension\mRan{\operatorname{Ran}}

\setupobject\catdgfun{
	singlekeys={
		{infty}{lower=\noexpand\boldsymbol{\infty}},
		{o}{upper=\noexpand\boldsymbol{\circ}},
	},
}

\newcategory\catstfun{\cocatstar Fun}[parent=catdgfun]
\newcategory\catgrfun{grFun}[parent=catdgfun]

\newcommand\catxstmod[1]{\category{$#1$-\cocatstar mod}}
\newcommand\catxgrmod[1]{\category{$#1$-grmod}}

\newcommand\catstmodx[1]{\category{\cocatstar mod-$#1$}}

\newcommand\catxstmodx[2]{\category{$#1$-\cocatstar mod-$#2$}}
\newcommand\catxgrmodx[2]{\category{$#1$-grmod-$#2$}}

\setupobject\catainftyfun{
	symbol={\categoryformat{\ainftysf{}Fun}},
	singlekeys={
		{o}{upper=\noexpand\boldsymbol{\circ}},
	},
}

\setupobject\Hom{
	singlekeys={
		{int}{symbol=\noexpand\mathop{\noexpand\mathsf{Hom}}},
		{R}{symbolputleft=\noexpand\mathbb{R}},
	},
}

\setupobject\catdgcat{
	singlekeys={
		{le0}{upper=\le0},
		{ge0}{upper=\ge0},
	},
}

\newcategory\catgrcat{grCat}[parent=catdgcat]
\newcategory\catstcat{\cocatstar Cat}[parent=catdgcat]

\newhomset\RHom{\Hom}[parent=Hom,int,return,R]

\newmapclass{replacementfunctor}[
	parent=algmap,
	singlekeys={
		{deltaplus}{
			lower=\catdelta[plus],
			output=simplicial,
		},
	},
]

\newreplacementfunctor\mcofrep{Q}
\newreplacementfunctor\mfibrep{R}
\newreplacementfunctor\cofrep{Q}[parent=mcofrep]
\newreplacementfunctor\fibrep{R}[parent=mfibrep]

\newmap{\face}{\partial}[
	parent=algmap,
	gradingpos=lower,
	lower={\!},
	keyvals={
		{min}{min,i={\keyvalue}},
		{max}{max,i={\keyvalue}},
	},
	singlekeys={
		{min}{d=0},
		{max}{d=\mathup{max}},
	},
]
\newmap{\degen}{\sigma}[parent=algmap,gradingpos=lower]

\newmap{\coface}{\partial}[
	parent=algmap,gradingpos=upper,lower={\!},
	keyvals={
		{min}{min,spar,upper={\keyvalue}},
		{max}{max,spar,upper={\keyvalue}},
	},
	singlekeys={
		{min}{d=0},
		{max}{d=\mathup{max}},
	},
]
\newmap{\codegen}{\sigma}[parent=algmap,gradingpos=upper]

\newalgop\End{\int}[
	output=alg,
	singlekeys={
		{rder}{symbolputleft={
			\noexpand\mathbb{R}
			\!\!\mathchoice{\!}{}{}{}
		}},
	},
]
\newalgop\Coend{\int}[
	gradingpos=lower,
	output=alg,
	arg={\mathchoice{\!\!\!}{\!}{\!}{\!}},
	iffirstarg=true,
	singlekeys={
		{lder}{symbolputleft={
			\noexpand\mathbb{L}
			\mathchoice{\!\!}{\!}{\!}{\!}
		}},
	},
]

\newsimplicial\simpdgcatB{\dgcatB}[output=algmap,ifoutput=true]
\newcosimplicial\cosimpdgcatA{\dgcatA}[output=algmap,ifoutput=true]

\setupobject\face{upper={\vphantom{*}}}
\setupobject\degen{upper={\vphantom{*}}}

\setupobject\coface{lower={\vphantom{*}}}
\setupobject\codegen{lower={\vphantom{*}}}

\newalgop\Union{\bigcup}[output=classset]
\newalgop\Intersection{\bigcap}[output=classset]

\newalgmap\Map{\operatorname{Map}}

\newalgmap\ssetmap{\operatorname{Map}}[output=simplicial]

\newhomology\tausset{\tau}

\newtuple\barorderedtensor{(}{<}{)}{\cdots}

\newtuple\naturaltensor{\langle}{,}{\rangle}{\ldots}[output=map]

\newclassset\dgschemeX{\mathbf{X}}
\newclassset\dgschemeY{\mathbf{Y}}
\newclassset\dgschemeZ{\mathbf{Z}}
\newclassset\dgschemeU{\mathbf{U}}
\newsheaf\shM{\mathscr{M}}
\newcategory\catdgaff{dgAff}

\usetikzlibrary{patterns}
\usetikzlibrary{decorations.pathreplacing,decorations.markings}

\tikzset{
  on each segment/.style={
    decorate,
    decoration={
      show path construction,
      moveto code={},
      lineto code={
        \path [#1]
        (\tikzinputsegmentfirst) -- (\tikzinputsegmentlast);
      },
      curveto code={
        \path [#1] (\tikzinputsegmentfirst)
        .. controls
        (\tikzinputsegmentsupporta) and (\tikzinputsegmentsupportb)
        ..
        (\tikzinputsegmentlast);
      },
      closepath code={
        \path [#1]
        (\tikzinputsegmentfirst) -- (\tikzinputsegmentlast);
      },
    },
  },
  mid arrow/.style={postaction={decorate,decoration={
        markings,
        mark=at position .5 with {\arrow[#1]{stealth}}
      }}},
}

\makeatother

\begin{document}

\maketitle

%


\begin{abstract}
	\noindent
	In this paper, we apply an explicit construction
	of a simplicial powering in dg-categories,
	due to \textcite{holstein} and \textcite{holstein_note},
	as well as our own results
	on homotopy ends
	\parencite{ends},
	to obtain an explicit model for the homotopy
	limit of a cosimplicial system of dg-categories.
	We apply this to obtain a model
	for homotopy descent in terms of \ainfty-comodules,
	proving a conjecture by~\textcite{explicit}
	in the process.
\end{abstract}

\chapter{Introduction}

This is a preparatory paper covering homotopical details
needed to define the derived category of \( \dgcatH \)-equivariant
\( \dgcatA \)-dg-modules in the case where \( \dgcatA \)~is a dg-algebra and \( \dgcatH \)~is a dg-Hopf-algebra acting on~\( \dgcatA \).
The example of interest is when \( \setX \)~is a regular, affine scheme and \( \algG \) an algebraic group acting on~\( \setX \),
and where \( \dgcatA = \difforms{\setX} \) and~\( \dgcatH = \difforms{\algG} \), both equipped with the \emph{zero} differential (\emph{not} the de Rham differential).
Compare this with the classical situation where \( \algA = \reg{\setX} \)~is an ordinary algebra and \( \algH = \reg{\algG} \)~is an ordinary Hopf algebra given by the functions on some algebraic group.
Then we may define the category of \( \algH \)-equivariant \( \algA \)-modules by the \emph{homotopy limit}
\[\textstyle
	\catxmod{\algA}[spar,equivar=\algH]
	=
	\invholim[\catdelta]{ \catxmod{ (\algH[tens=n]\tens\algA) } }
\]
with respect to the model structure on categories
described in~\textcite{rezk}.
In the case where \( \setX[dom=\algG] \)~exists
in schemes and the map~\( \setX \to \setX[dom=\algG] \)
is fully faithful, descent theory tells us that 
\( \catxmod{\algA}[spar,equivar=\algH,smash] \)~recovers~\(
	\catqcoh{{\setX[dom=\algG]}}
\).
If \( \setX[dom=\algG] \)~exists only as a stack,
it will instead recover quasi-coherent sheaves on that.

More generally, if \( \mf \colon \setX \to \setY \)~is an fpqc morphism of schemes, we may consider its \textdef{descent groupoid},
the internal groupoid in schemes~\( \ssetX{1} \rightrightarrows \ssetX{0} \)
given by~\( \ssetX{0} = \setX \)
and~\( \ssetX{1} = \setX \fibre[\setY] \setX \)
(both maps in the fibre product being~\( \mf \)).
We may then consider its classifying space,
the internal Kan complex in schemes given by
\[
	\ssetX{n}
	=
	\underbrace{
		\ssetX{1}
		\fibre[\ssetX{0}]
		\ssetX{1}
		\fibre[\ssetX{0}]
		\cdots
		\fibre[\ssetX{0}]
		\ssetX{1}
	}_{\text{\(n\)~factors}}
	=
	\underbrace{
		\setX
		\fibre[\setY]
		\setX
		\fibre[\setY]
		\cdots
		\fibre[\setY]
		\setX
	}_{\text{\(n+1\)~factors}}
\]
with the usual simplicial structure, the face maps~\(
	\face[d=i]\colon\ssetX{n}\to\ssetX{n-1}
\)
applying~\( \mf \) at the \( i \)th~step,
and the degeneracy maps~\( \degen[d=i]\colon\ssetX{n}\to\ssetX{n+1} \)
inserting the diagonal map at the \( i \)th~step.
Then \( \setY \)~becomes an augmentation of the
simplicial scheme~\( \ssetX{*} \):
\[\begin{tikzcd}[sep=scriptsize]
	\setY
	\ar[r,<-,dashed]
		&
			\ssetX{0}
			\ar[r,<-,yshift=1.5pt]
			\ar[r,->]
			\ar[r,<-,yshift=-1.5pt]
				& \ssetX{1}
				\ar[r,<-,yshift=3pt]
				\ar[r,->,yshift=1.5pt]
				\ar[r,<-]
				\ar[r,->,yshift=-1.5pt]
				\ar[r,<-,yshift=-3pt]
					& \ssetX{2}
					\ar[r,<-,yshift=4.5pt]
					\ar[r,->,yshift=3pt]
					\ar[r,<-,yshift=1.5pt]
					\ar[r,->]
					\ar[r,<-,yshift=-1.5pt]
					\ar[r,->,yshift=-3pt]
					\ar[r,<-,yshift=-4.5pt]
						& \cdots\invdot
\end{tikzcd}\]
Then descent theory tells us that we
recover quasi-coherent sheaves on~\( \setY \)
by gluing quasi-coherent sheaves on~\( \ssetX{0} \)
via gluing data stored in the categories~\( \catqcoh{\ssetX{i}} \)
for~\( i>0 \). This may be formulated by saying that
\( \catqcoh{\setY} \)~is given by the homotopy limit
(see \cref{res:classical_descent_holim}),
\[\textstyle
	\catqcoh{\setY}
	=
	\invholim[\catdelta]{ \catqcoh{\ssetX{*}} }
	.
\]
The homotopy limit is the derived functor of the limit.
It can be roughly formulated as a homotopy-invariant
version of the limit where the usual squares only hold up to correction (in the case of~\( \catcat \), by isomorphisms).
In other words, up to correction,
we have an augmented cosimplicial system of categories
\[\begin{tikzcd}[sep=scriptsize]
	\catqcoh{\setY}
	\ar[r,->,dashed]
		&
			\catqcoh{\ssetX{0}}
			\ar[r,->,yshift=1.5pt]
			\ar[r,<-]
			\ar[r,->,yshift=-1.5pt]
				& \catqcoh{\ssetX{1}}
				\ar[r,->,yshift=3pt]
				\ar[r,<-,yshift=1.5pt]
				\ar[r,->]
				\ar[r,<-,yshift=-1.5pt]
				\ar[r,->,yshift=-3pt]
					& \catqcoh{\ssetX{2}}
					\ar[r,->,yshift=4.5pt]
					\ar[r,<-,yshift=3pt]
					\ar[r,->,yshift=1.5pt]
					\ar[r,<-]
					\ar[r,->,yshift=-1.5pt]
					\ar[r,<-,yshift=-3pt]
					\ar[r,->,yshift=-4.5pt]
						& \cdots\invcomma
\end{tikzcd}\]
where the cosimplicial maps are given by pullbacks of the simplicial
maps.
We notice that \( \invholim[\catdelta]{ \catqcoh{\ssetX{*}} } \)~makes
sense even if the scheme~\( \setY \) does not exist, as
it depends only on the groupoid~\( \ssetX{1} \rightrightarrows \ssetX{0} \).

%
%
%
%
%
%
%
%

Alternatively, the pullback and pushforward functors
\[
	\mf[shpull]
	\colon
	\catqcoh{\setY}
	\adjunctionarrow
	\catqcoh{\setX}
	\noloc
	\mf[shpush]
\]
yield a comonad~\(
	\mT = \mf[shpull] \mf[shpush]
	\colon
	\catqcoh{\setX}
	\to
	\catqcoh{\setX}
\).
Then the Barr--Beck theorem tells us that we recover~\( \catqcoh{\setY} \)
as
\[
	\catqcoh{\setY}
	\cong
	\catxcomod{\mT}{\catqcoh{\setX}}
	,
\]
where the \anrhs the category of \( \mT \)-comodules in~\( \catqcoh{\setX} \).
In the affine situation,
we may write \( \setX = \Spec{\algA} \)
and~\( \ssetX{1} = \Spec{\algC} \)
and observe that \( \algC \)~becomes
a coalgebra in~\( \catxmodx{\algA}{\algA} \) with comultiplication~\( \comult = \face[d=1,comor] \colon \algC \to \algC \tens[\algA] \algC \).
Then
~\(
	\catxcomod{\mT}{\catqcoh{\setX}}
	=
	\catxcomod{\algC}{\catxmod{\algA}}
\)
is just the category of \( \algC \)-comodules
in~\( \catxmod{\algA} \).
Again, this is definable using only the data of the groupoid~\( \ssetX{1} \rightrightarrows \ssetX{0} \),
even if \( \setY \)~does not exist (in schemes).

The purpose of this paper is to prove a homotopy version of the equivalence
\[\textstyle
	\invholim[\catdelta]{ \catqcoh{\ssetX{*}} }
	\cong
	\catxcomod{\algC}{\catxmod{\algA}}
\]
for affine dg-schemes.
More precisely, we prove
\begin{theoremnonumber}[\ref{res:A_infty_comodules}. Theorem]
	Suppose that
	\( \dgschemeX[1] \rightrightarrows \dgschemeX[0] \)
	is a groupoid in affine dg-schemes,
	and consider the associated classifying space~\( \dgschemeX[ibullet] \) given by
	\[\textstyle
		\dgschemeX[n]
		=
		\dgschemeX[1]
		\fibre[\dgschemeX[0]]
		\dgschemeX[1]
		\fibre[\dgschemeX[0]]
		\cdots
		\fibre[\dgschemeX[0]]
		\dgschemeX[1]
		.
	\]
	Write~\( \cosimpdgcatA{n} = \dgcatA[\dgschemeX[n],smash] \)
	for the associated cosimplicial system of dg-algebras.
	Let~\( \dgcatA = \cosimpdgcatA{0} \)
	and~\( \dgcatC = \cosimpdgcatA{1} \),
	and note that \( \dgcatC \)~is a
	counital
	coalgebra in~\( \catxdgmodx{\dgcatA}{\dgcatA} \)
	via the map~\(
		\comult
		=
		\face[d=1,comor]
		\colon
		\dgcatC
		\to
		\dgcatC\tens[\dgcatA]\dgcatC
	\).
	Then we have an equivalence of dg-categories
	\[\textstyle
		\invholim[\catdelta]{
				\catdgqcoh{ \dgschemeX[ibullet] }
		}
		\cong
		\catxcomod{\dgcatC}[infty,hcu,formal]{\dgcatA}
		,
	\]
	where the \anrhs denotes the dg-category
	of formal, homotopy-counital \ainfty-comodules
	over~\( \dgcatC \)
	in~\( \catxdgmod{\dgcatA} \).
\end{theoremnonumber}

Much of the inspiration comes from~\textcite{explicit}.
In the process, we prove their Conjecture~1 and generalize their results.

In \cref{chap:classical_descent}, we set up
classical descent theory, including the homotopy limit and Barr--Beck formulations.
In \cref{sec:differential-graded-categories},
we recall differential graded (co)algebras and (co)categories
and their \ainfty-analogues.
Finally, in \cref{chap:holim_dgcat}, we present our main results
on homotopy limits of dg-categories, including the above theorem.

\section*{Acknowledgements}

We would like to thank Edouard Balzin for his interest in the project, his guidance, suggestions, and for reading an early draft of the paper. We also thank Daria Poliakova for stimulating discussions.

\begingroup

\chapter{Classical descent theory}\label{chap:classical_descent}

Suppose that \( \mf \colon \setX \to \setY \)~is an fpqc morphism of schemes. Then descent theory tells us that we
recover quasi-coherent sheaves on~\( \setY \) by gluing quasi-coherent sheaves on~\( \setX \).
Concretely, if we denote by~\( \ssetX{1} \rightrightarrows \ssetX{0} \)
the \ancech groupoid (see the introduction),
then the category~\( \catqcoh{\setY} \)~is equivalent
to the category of “descent data” on the
groupoid~\( \ssetX{*} \), defined in the following manner:

Let~\( \ssetX{1} \rightrightarrows \ssetX{0} \)
be any internal groupoid in the category of schemes,
and denote by~\( \ssetX{*} \)
its classifying space, the internal Kan complex
in schemes given by~\(
	\ssetX{n}
	=
	\ssetX{1}
	\fibre[\ssetX{0}]
	\cdots
	\fibre[\ssetX{0}]
	\ssetX{1}
\).
The category~\( \catdesc{\ssetX{*}} \) of \textdef{descent data} on~\( \ssetX{*} \)
has objects the pairs~\( \tup{\algM,\mtheta} \),
where \( \algM \in \catqcoh{\ssetX{0}} \),
and \(
	\mtheta \colon \face[d=1,shpull]{\algM}\to\face[d=0,shpull]{\algM}
\)
a map satisfying
the cocycle and unit conditions
\[
	\face[d=0,shpull]{\mtheta}
	\circ
	\face[d=2,shpull]{\mtheta}
	=
	\face[d=1,shpull]{\mtheta}
	\qquad\text{and}\qquad
	\degen[d=0,shpull]{\mtheta} = \id
.\]
A morphism~\( \malpha\colon\tup{\algM,\mtheta}\to\tup{\algN,\meta} \)
is a morphism~\( \malpha \colon \algM \to \algN \)
in~\( \catqcoh{\ssetX{0}} \) such that~\(
	\meta\circ\face[d=1,shpull]{\malpha} = \face[d=0,shpull]{\malpha}\circ\mtheta
\).


\begin{proposition}\label{res:cocycle_iff_degen=0}
	Assuming the cocycle condition,
	the assumption~\( \degen[d=0,shpull]{\mtheta}=\id \)
	is equivalent to~\( \mtheta \) being an isomorphism.
\end{proposition}

\begin{proof}
	If \( \mtheta \)~is an isomorphism,
	we may apply \(
		\degen[d=0,shpull] \degen[d=0,shpull]
		=
		\degen[d=0,shpull] \degen[d=1,shpull]
	\)
	to both sides of the cocycle condition
	and get~\(
		\degen[d=0,shpull]{ \mtheta }
		\circ
		\degen[d=0,shpull]{ \mtheta }
		=
		\degen[d=0,shpull]{ \mtheta }
	\).
	Now \( \degen[d=0,shpull]{ \mtheta } \)~is the image of an isomorphism and hence an isomorphism, so we obtain~\(
		\degen[d=0,shpull]{ \mtheta } = \id
	\).
	Conversely, if~\( \degen[d=0,shpull]{\mtheta} = \id \),
	then we use the groupoid conditions
	\[
		\face[d=1]
		( \miota\times\id )
		\diagonalmap
		=
		\degen[d=0]
		\face[d=1]
		\qquad\text{and}\qquad
		\face[d=1]
		( \id\times\miota )
		\diagonalmap
		=
		\degen[d=0]
		\face[d=0]	
	\]
	where \( \diagonalmap \)~denotes the diagonal map;
	in the first case, it is the diagonal
	map
	\[
		\diagonalmap
		\colon
		\ssetX{1}
		\to
		\ssetX{1}
		\;
		{}_{\face[d=1]}
		\!
		\fibre[limits,\ssetX{0},return,bin]
		{}_{\face[d=1]}
		\;
		\ssetX{1},
	\]
	the fibre product being taken on both sides with
	respect to the source map~\( \face[d=1] \).
	In the second equation, \( \face[d=0] \)~is used instead.
	We have~\(
		\map{\degen[d=0]\face[d=1]}[spar,shpull]{\mtheta}
		=
		\face[d=1,shpull]{{ \degen[d=0,shpull]{ \mtheta } }}
		= \id,
	\)
	and hence
	\begin{align*}
		\id
		&=
		\map{ ( \miota\times\id ) \diagonalmap }[spar,shpull]{{
			\face[d=1,shpull]{\mtheta}
		}}
		=
		\map{ ( \miota\times\id ) \diagonalmap }[spar,shpull,par]{{
			\face[d=0,shpull]{\mtheta}
			\circ
			\face[d=2,shpull]{\mtheta}
		}}
	\\
		&=
		\map{ \face[d=0] ( \miota \times \id ) \diagonalmap }[spar,qpull]{
			\mtheta
		}
		\circ
		\map{ \face[d=1] ( \miota \times \id ) \diagonalmap }[spar,qpull]{
			\mtheta
		}
		=
		\miota[qpull]{\mtheta}\circ\mtheta
	.
	\end{align*}
	Similarly, the other equation yields \(
		\id = \mtheta\circ\miota[qpull]{\mtheta}
	\).
\end{proof}

\begin{example}
	Suppose that \( \setY \)~is a scheme
	and \( \bigcup \setU[i] \to \setY \)~an fpqc covering,
	and define~\( \setX = \Coprod{\setU[i]} \),
	so that the morphism~\( \mf\colon\setX\to\setY \) is~fpqc.
	Then the descent groupoid is exactly
	the \textdef{\ancech groupoid}
	given by~\( \ssetX{0} = \Coprod{\setU[i]} \)
	and~\( \ssetX{1} = \Coprod{\setU[ij]} \)
	(here, we use the usual convention
	of letting~\(
		\setU[i_0\cdots i_n]
		=
		\setU[i_0]\cap\cdots\cap\setU[i_n]
	\)).
	The source map~\( \face[d=1]\colon\ssetX{1}\to\ssetX{0} \)
	is given by the embeddings~\( \setU[ij]\into\setU[i] \),
	the target map~\( \face[d=0]\colon\ssetX{1}\to\ssetX{0} \)
	is given by the embeddings~\( \setU[ij]\into\setU[j] \),
	the unit~\( \ssetX{0}\to\ssetX{1} \) is given by~\( \setU[i] \xto{=} \setU[ii] \),
	composition~\(
		\face[d=1]
		\colon
		\ssetX{1}
		\fibre[\ssetX{0}]
		\ssetX{1}
		\to
		\ssetX{1}
	\)
	is given by~\( \setU[ijk]\into\setU[ik] \),
	and inversion~\( \groupoidinv\colon\ssetX{1}\to\ssetX{1} \)
	is given by~\( \setU[ij] \xto{=} \setU[ji] \).
	More generally
	\[\textstyle
		\ssetX{n}
		=
		\ssetX{1}
		\fibre[\ssetX{0}]
		\ssetX{1}
		\fibre[\ssetX{0}]
		\cdots
		\fibre[\ssetX{0}]
		\ssetX{1}
		=
		\Coprod{ \setU[i_0 \cdots i_n] }
	.\]
	If \( \mvarphi\colon\ordset{m}\to\ordset{n} \)
	is a map in~\( \catdelta \),
	the map~\( \mvarphi[pull]\colon\ssetX{n}\to\ssetX{m} \)
	is given by the embedding~\(\smash{
		\setU[i_0\cdots i_n]
		\into
		\setU[i_{\mvarphi{0}} \cdots i_{\mvarphi{m}}]
	}\).
	Then a pair \( \tup{\algM,\mtheta}\in\catdesc{\ssetX{*}} \)~consists
	of collections~\( \algM = \tup{\algM[i]} \)
	of quasi-coherent sheaves~\( \algM[i]\in\catqcoh{\setU[i]} \)
	on each element in the covering,
	and collections~\( \mtheta = \tup{\mtheta[ij]} \)
	of gluing morphisms~\(\smash{
		\mtheta[ij]
		\colon
		\algM[i,res=\setU[ij]]
		\to
		\algM[j,res=\setU[ij]]
	}\),
	subject to the cocycle and unit conditions
	\[
		\mtheta[jk]\circ\mtheta[ij]=\mtheta[ik]
		\qquad\text{and}\qquad
		\mtheta[ii] = \id[\algM[i]]
		.
	\]
	As in the general case, this implies
	that all~\( \mtheta[ij] \) are automatically isomorphisms.
	A morphism~\( \malpha\colon\tup{\algM,\mtheta}\to\alg{\algN,\meta} \)
	is a tuple~\( \malpha = \tup{\malpha[i]} \)
	of morphisms~\( \malpha[i]\colon\algM[i]\to\algN[i] \)
	such that~\(\smash{
		\meta[ij]\circ\malpha[i,res=\setU[ij]]
		=
		\malpha[j,res=\setU[ij]]\circ\mtheta[ij]
	}\)
	for all~\( i,j \).
	Then the descent statement from before simply
	translates to the fact that we recover quasi-coherent sheaves on~\( \setY \)
	from these data.
\end{example}

\section{Barr--Beck theorem and comodules}

One classical way of rewriting the descent condition
is via the Barr--Beck theorem.
We state it in the generality we shall need it.
Following \textcite{maclane},
a \textdef{comonad}
on a category~\( \catC \) is an endofunctor~\(
	\mT\colon\catC\to\catC
\)
together with natural transformations~\(
	\comult\colon\mT\to\mT[squared]
\), called \textdef{comultiplication},
and~\(
	\counit\colon\mT\to\id[\catC]
\),
called \textdef{counit},
such that the following diagrams commute:
\[\begin{tikzcd}
	\mT
	\ar[r,"\comult"]
	\ar[d,"\comult"']
		& \mT[squared]
		\ar[d,"\comult\mT"]
			&	& \mT \ar[d,"\comult"] \ar[ld,equal] \ar[rd,equal]
\\
	\mT[squared]
	\ar[r,"\mT\comult"']
		& \mT[to=3]
			&[3em] \mT
				& \mT[squared] \ar[l,"\counit\mT"] \ar[r,"\mT\counit"']
					& \mT
					\invdot
\end{tikzcd}\]
A \textdef{comodule} for a comonad~\( \mT \)
consists of an object~\( \vx \in \catC \)
together with a morphism~\( \ca \colon \vx \to \mT[nopar]{\vx} \),
called \textdef{coaction},
such that the following diagrams commute:
\[\begin{tikzcd}
	\vx \ar[r,"\ca"] \ar[d,"\ca"']
		& \mT[nopar]{\vx} \ar[d,"\comult"]
			&[3em]
			\vx \ar[r,"\ca"] \ar[rd,equal]
				& \mT[nopar]{\vx} \ar[d,"\counit"]
\\
	\mT[nopar]{\vx} \ar[r,"\mT\ca"']
		& \mT[squared,nopar]{\vx}
			&	& \vx \invdot
\end{tikzcd}\]
A map of comodules~\( \mf \colon \tup{\vx,\ca}\to\tup{\vy,\ca} \) is a map commuting with coaction.
We thus obtain a category~\( \catxcomod{\mT} \)
of comodules over~\( \mT \).

\begin{example}
	Any pair of adjoint functors~\(
		\mF \colon \catD\adjunctionarrow\catC\noloc\mG
	\)
	defines a comonad in~\( \catC \)
	by letting~\( \mT = \mF \mG \)
	and defining the comultiplication~\(
		\comult
		\colon
		\mT
		\to
		\mT[squared]
	\)
	by the unit of adjunction~\( \mF\mG\to \mF(\mG\mF)\mG \),
	and the counit by the counit of adjunction.
	We see that any object of the form~\( \mF{\vx} \in\catC \)
	for~\( \vx \in \catD \)
	is a comodule over~\( \mT \)
	via~\( \mF{x} \to \map{\mF (\mG \mF)}{\vx} \).
\end{example}

\begin{theorem}[Barr--Beck theorem]
	Suppose that \( \catC \) and \( \catD \)~are
	Abelian categories and
	\(
		\mF \colon \catD\adjunctionarrow\catC\noloc\mG
	\)~an adjunction with~\( \mF \)~full and exact.
	If~\( \mT = \mF \mG \)~is the associated comonad,
	\( \mF \)~descends to an equivalence
	of Abelian categories~\( \mF \colon \dgcatC \isoto \catxcomod{\mT} \).
\end{theorem}

Suppose now that \( \mf \colon \setX \to \setY \)~is a faithfully
flat morphism of affine schemes \( \setX = \Spec{\algA} \)
and~\( \setY = \Spec{\algB} \), and consider the pullback square
\[\begin{tikzcd}
	\setX \fibre[\setY] \setX
	\ar[r,"{\face[d=1]}"] \ar[d,"{\face[d=0]}"']
		& \setX \ar[d,"\mf"]
\\
	\setX \ar[r,"\mf"']
		& \setY \invdot
\end{tikzcd}\]
Then base change shows that the comonad
\[
	\mT = \mf[shpull] \mf[shpush]
	\colon
	\catqcoh{\setX}
	\longto
	\catqcoh{\setX},
	\qquad
	\algM
	\longmapsto
	\algA
	\tens[\algB]
	\algM,
\]
is equal to
\[
	\face[d=1,shpush] \face[d=0,shpull]
	\colon
	\catqcoh{\setX}
	\longto
	\catqcoh{\setX},
	\qquad
	\algM
	\longmapsto
	( \algA \tens[\algB] \algA ) \tens[\algA] \algM
.\]
Now \(
	\algC
	:=
	\algA\tens[\algB]\algA
	=
	\reg{\setX\fibre[\setY]\setX}
	=
	\reg{\ssetX{1}}
\)
is naturally a coalgebra in the monoidal category~\( \tup{ \catxmodx{\algA}{\algA} , \tens[\algA] } \)
of \( \algA \)-bimodules
via the comorphism~\(
	\face[d=1,comor]
	\colon
	\algC
	\to
	\algC \tens[\algA] \algC
\)
associated to~\(
	\face[d=1]
	\colon
	\ssetX{1} \fibre[\ssetX{0}] \ssetX{1}
	\to
	\ssetX{1}
\).
The category of \( \mT \)-comodules is
the same as the category~\( \catxcomod{\algC}{\catxmod{\algA}} \)
of \( \algC \)-comodules
in the category~\( \catxmod{\algA} \),
so Barr--Beck becomes the statement that
	
\begin{proposition}
	If \( \mf \colon \setX = \Spec{\algA} \to \setY = \Spec{\algB} \)~is
	faithfully flat, consider the coalgebra~\(
		\algC = \reg{\setX\fibre[\setY]\setX} = \algA \tens[\algB] \algA
	\).
	Then
	\[
		\mf[shpull]
		\colon
		\catqcoh{\setY} = \catxmod{\algB}
		\longisoto
		\catxcomod{\algC}{\catxmod{\algA}}
	\]
	is an equivalence of Abelian categories.
\end{proposition}

Even if the descent groupoid~\( \ssetX{*} \)
does not come from the covering of a scheme~\( \setY \),
the category of descend data still becomes equivalent to
comodules over the coalgebra~\( \algC = \reg{\ssetX{1}} \).
Suppose that \( \ssetX{1} \rightrightarrows \ssetX{0} \)~is
an internal groupoid in the category of
\emph{affine} schemes.
By
adjunction, the gluing map~\(
	\mtheta
	\colon
	\face[d=1,shpull]{\algM}
	\to
	\face[d=0,shpull]{\algM}
\)
is equivalent to a map~\(
	\ca
	\colon
	\algM
	\to
	\face[d=1,shpush]{{ \face[d=0,shpull]{ \algM } }}
	=
	\algC \tens[\algA] \algM
\).
We claim that this operation makes \( \algM \)~into a
comodule over~\( \algC \) in the category~\( \catxmod{\algA} \).

\begin{proposition}\label{res:descent=comodules}
	If \( \ssetX{1} \rightrightarrows \ssetX{0} \)
	is an internal groupoid in \emph{affine} schemes,
	then we have an equivalence of categories~\( \catdesc{\ssetX{*}} \cong \catxcomod{\algC}{\catxmod{\algA}} \),
	where the \anrhs denotes the category of \( \algC \)-comodules
	in~\( \catxmod{\algA} \).
\end{proposition}

For this, we need some technical lemmas that will come in handy later.
For a simplicial object~\( \ssetX{*} \) and any~\( n \), we shall use the notation~\( \face[max] \)
for the morphism~\( \face[d=n]\colon\ssetX{n}\to\ssetX{n-1} \).
This allows us to consider powers of these, e.g.
\( \face[max,upper=l] = \face[d=1] \cdots \face[d=n-1] \face[d=n] \).
Similarly, we write \( \face[min=l] = \face[min] \face[min] \cdots \face[min]\).
We use a similar convention in the cosimplicial
case and hence write
e.g.~\( \coface[max=l] = \coface[d=n] \coface[d=n-1]\cdots\coface[d=1] \)
and~\( \coface[min=l] = \coface[d=0] \coface[d=0] \cdots \coface[d=0] \).
Applying
base change to the pullback diagram
\[\begin{tikzcd}
	\ssetX{i+j} \ar[r,"{\face[max=j]}"] \ar[d,"{\face[min=i]}"']
		& \ssetX{i} \ar[d,"{\face[min=i]}"]
\\
	\ssetX{j} \ar[r,"{\face[max=j]}"']
		& \ssetX{0}
\end{tikzcd}\]
we obtain
\(
	\face[max=j,shpush] \face[min=i,shpull]
	\cong
	\face[min=i,shpull] \face[max=j,shpush]
\).
This implies

\begin{lemma}\label{res:base_change_face_maps}
	Let~\( \ssetX{1} \rightrightarrows \ssetX{0} \)
	be a monoid in the category of schemes,
	with~\( \ssetX{*} \) the associated classifying space.
	Suppose we are given objects~\( \algM , \algN , \algP \in \catqcoh{\ssetX{0}} \),
	along with two maps~\(
		\mtheta
		\colon
		\face[max=i,shpull] { \algM }
		\to
		\face[min=i,shpull] { \algN }
	\)
	in~\( \catqcoh{\ssetX{i}} \)
	and~\(\smash{
		\meta
		\colon
		\face[max=j,shpull] { \algN }
		\to
		\face[min=j,shpull] { \algP }
	}\)
	in~\( \catqcoh{\ssetX{j}} \), where~\( i , j \ge 0 \).
	Then the composition
	\begin{align*}
	\MoveEqLeft
	\begin{tikzcd}[sep=scriptsize,ampersand replacement=\&]
		\algM \ar[r]
			\& \face[max=i,shpush] {{ \face[max=i,shpull] { \algM } }}
			\ar[r,"{\mtheta}"]
				\& \face[max=i,shpush] {{ \face[min=i,shpull] { \algN } }}
				\ar[r]
					\& \face[max=i,shpush] {{
						\face[min=i,shpull] {{
							\face[max=j,shpush] {{
								\face[max=j,shpull] {{
									\algN
								}}
							}}
						}}
					}}
	\end{tikzcd}
	\\[-.7em]
	&\begin{tikzcd}[sep=scriptsize,ampersand replacement=\&]
		{} \ar[r,"{\meta}"']
			\& \face[max=i,shpush] {{
				\face[min=i,shpull] {{
					\face[max=j,shpush] {{
						\face[min=j,shpull] {{
							\algP
						}}
					}}
				}}
			}}
			\ar[r,symbol={\cong}]
				\&[-1.5em] \face[max=i,shpush] {{
				 \face[max=j,shpush] {{
						\face[min=i,shpull] {{
							\face[min=j,shpull] {{
								\algP
							}}
						}}
					}}
				}}
				\ar[r,symbol={=}]
					\&[-1.5em] \face[max=i+j,shpush] {{	
						\face[min=(i+j),shpull] {{ \algP }}
					}}
	\invdot
	\end{tikzcd}\end{align*}
	is the same as the composition
	\begin{align*}
	\MoveEqLeft\begin{tikzcd}[sep=scriptsize,ampersand replacement=\&]
		\algM \ar[r]
			\& \face[max=i+j,shpush] {{
				\face[max=(i+j),shpull] {{ \algM }}
			}}
			\ar[r,"{\face[max=j,shpull]{ \mtheta }}"]
				\&\face[max=i+j,shpush] {{
					\face[max=j,shpull] {{
						\face[ min=i, shpull ] { \algN }
					}}
				}}
	\end{tikzcd}
	\\[-.7em]
	&\begin{tikzcd}[sep=scriptsize,ampersand replacement=\&]
		\&{}= \face[max=i+j,shpush] {{
			\face[min=i,shpull] {{
				\face[ max=j, shpull ] { \algN }
			}}
		}}
		\ar[r,"{\face[min=i,shpull] { \meta }}"']
			\&\face[max=i+j,shpush] {{
				\face[min=i,shpull] {{
					\face[ min=j, shpull ] { \algP }
				}}
			}}
			\ar[r,symbol={=}]
				\&[-1.5em] \face[max=i+j,shpush] {{
					\face[min=\smash{(i+j)},shpull] {{ \algP }}
				}}
				.
	\end{tikzcd}\end{align*}
	In other words, the map is
	the unit of adjunction~\(
		\algM
		\to
		\face[max=i+j,shpush] {{
			\face[max=(i+j),shpull] { \algM }
		}}
	\)
	composed with the
	map~\(\smash{
		\face[min=i,shpull] { \meta }
		\circ
		\face[max=j,shpull] { \mtheta }
	}\).
\end{lemma}

\begin{proof}
	The statement is local, so assume
	that \( \ssetX{0} = \Spec{\algA} \)
	and~\( \ssetX{1} = \Spec{\algC} \)
	are affine.
	The two maps~\(
		\face[d=0],\face[d=1]
		\colon
		\ssetX{1}
		\rightrightarrows
		\ssetX{0}
	\) correspond to maps~\(
		\face[d=0,comor],\face[d=1,comor]
		\colon
		\algA
		\rightrightarrows
		\algC
	\),
	which turn~\( \algC \) into an \( \algA \)-bimodule, say, with~\( \face[d=1,comor] \) providing the left action and~\( \face[d=0,comor] \) the right one.
	Then \( \mtheta \) and~\( \meta \)
	become maps
	\begin{flalign*}
	&&
		\mtheta
		\colon
		\algM
		\tens[\algA]
		\algC[tens=i]
		&\longto
		\algC[tens=i]
		\tens[\algA]
		\algN
		&&\text{in~\( \catxmod{\algC[tens=i]} \)}
			&&
	\\
	&\text{and}&
		\meta
		\colon
		\algN
		\tens[\algA]
		\algC[tens=j]
		&\longto
		\algC[tens=j]
		\tens[\algA]
		\algP
		&&\text{in~\(
			\catxmod{\algC[tens=j]}
		\).}
	\end{flalign*}
	The unit of adjunction~\(
		\algM
		\to
		\face[max=i,shpush]
		\face[max=i,shpull]
		\algM
		=
		\algM
		\tens[\algA]
		\algC[tens=i]
	\)
	is then given by~\( m\mapsto m\tens1\tens\cdots\tens1 \), and
	similarly for the other units of adjunction.
	Then the statement we want simply
	says that
	\[\begin{tikzcd}[sep=small]
		\algM
		\ar[r]
		&
		\algM
		\tens[limits,\algA,return,bin]
		\algC[tens=i]
		\ar[r,"\mtheta"]
		&
		\algC[tens=i]
		\tens[limits,\algA,return,bin]
		\algN
		\ar[r]
		&
		\algC[tens=i]
		\tens[limits,\algA,return,bin]
		(
			\algN
			\tens[limits,\algA,return,bin]
			\algC[tens=j]
		)
		\ar[r,"\meta"]
		&
		\algC[tens=i]
		\tens[limits,\algA,return,bin]
		(
			\algC[tens=j]
			\tens[limits,\algA,return,bin]
			\algP
		)		
	\end{tikzcd}\]
	is equal to
	\[\begin{tikzcd}[sep=small]
		\algM
		\ar[r]
		&
		\algM
		\tens[limits,\algA,return,bin]
		\algC[tens=(i+j)]
		\ar[r,"\mtheta"]
		&
		\algC[tens=i]
		\tens[limits,\algA,return,bin]
		\algN
		\tens[limits,\algA,return,bin]
		\algC[tens=j]
		\ar[r,"\meta"]
		&
		\algC[tens=(i+j)]
		\tens[limits,\algA,return,bin]
		\algP,
	\end{tikzcd}\]
	which boils down to associativity of the tensor product.
\end{proof}

\begin{lemma}\label{res:add_comultiplication_map}
	Let the assumptions be as in~\cref{res:base_change_face_maps}.
	Suppose that we are given
	objects~\( \algM , \algN \in \catqcoh{\ssetX{0}} \)
	along with a
	map~\(
		\mtheta
		\colon
		\face[max=n,shpull] { \algM }
		\to
		\face[min=n,shpull] { \algM }
	\).
	Let~\( 0 < i < n \).
	Then we have a commutative diagram
	\[\begin{tikzcd}[sep=scriptsize]
		\algM \ar[r] \ar[d,dashed]
			&[-1.5em] \face[max=n,shpush] {{ \face[max=n,shpull] { \algM } }}
			\ar[r,"{\mtheta}"]
			\ar[d,dashed]
				& \face[max=n,shpush] {{ \face[min=n,shpull] { \algN } }}
				\ar[d]
	\\
		\face[max=n+1,shpush] {{ \face[max=(n+1),shpull] { \algM } }}
		\ar[r,symbol={=}]
			& \face[max=n,shpush] {{
				\face[d=i,shpush] {{
					\face[d=i,shpull] {{
						\face[max=n,shpull] { \algM }
					}}
				}}
			}}
			\ar[r,dashed,"{\face[d=i,shpull]{\mtheta}}"]
				& \face[max=n,shpush] {{
					\face[d=i,shpush] {{
						\face[d=i,shpull] {{
							\face[min=n,shpull] { \algN }
						}}
					}}
				}}
				\ar[r,symbol={=}]
					&[-1.5em]
					\face[max=n+1,shpush] {{
						\face[min=\smash{(n+1)},shpull] {{ \algN }}
					}} .
	\end{tikzcd}\]
\end{lemma}

\begin{proof}
	Clear from naturality of the unit of adjunction.
\end{proof}

\begin{proof}[Proof of \cref{res:descent=comodules}]
	The coassociativity diagram
	\[\begin{tikzcd}
		M
		\ar[r,"\ca"]
			& \algC \tens[\algA] \algM
			\ar[r,yshift=1.5pt,"\id\sotimes\ca"]
			\ar[r,yshift=-1.5pt,"\comult\sotimes\id"']
				& \algC \tens[\algA] \algC \tens[\algA] \algM.
	\end{tikzcd}\]
	is the same as the diagram
	\[\begin{tikzcd}[column sep=scriptsize,row sep=scriptsize]
			&	&
				\face[d=1,shpush] \face[d=0,shpull]
				\face[d=1,shpush] \face[d=1,shpull] \algM
				\ar[r,"\mtheta"]
					&
					\face[d=1,shpush] \face[d=0,shpull]
					\face[d=1,shpush] \face[d=0,shpull] \algM
					\ar[d,equal]
	\\
		M \ar[r]
			& \face[d=1,shpush] \face[d=1,shpull] \algM
			\ar[r,"\mtheta"]
				& \face[d=1,shpush] \face[d=0,shpull] \algM
				\ar[u] 
				\ar[d]
					& \face[d=1,shpush] \face[d=2,shpush] \face[d=0,shpull] \face[d=0,shpull] \algM
	\\
				&	& \face[d=1,shpush] \face[d=1,shpush] \face[d=1,shpull] \face[d=0,shpull] \algM
						\ar[ur,equal]
	\end{tikzcd}\]
	The two lemmas show that this diagram is the same as the diagram
	\[\begin{tikzcd}[column sep=scriptsize,row sep=scriptsize]
		\algM \ar[r]
			& \face[d=1,shpush] \face[d=2,shpush]
			\face[d=2,shpull] \face[d=1,shpull] \algM
			\ar[r,"{\face[d=0,shpull]{\mtheta}\circ\face[d=2,shpull]{\mtheta}}",yshift=1.5pt]
			\ar[r,"{\face[d=1,shpull]{\mtheta}}"',yshift=-1.5pt]
				&[2em] \face[d=1,shpush] \face[d=2,shpush]
				\face[d=0,shpull] \face[d=0,shpull] \algM
	\invcomma
	\end{tikzcd}\]
	so commutativity boils down to the relation~\(
		\face[d=1,shpull]\theta =
		\face[d=0,shpull]\theta\circ \face[d=2,shpull]\theta
	\).
	
	To verify the unit condition, notice that
	condition~\(
		\map{\counit\tens\id[\algA]}[spar]\circ\ca=\id[\algM]
	\),
	note that the left-hand side is the composition
	\[\begin{tikzcd}
		\algM \ar[r]
			& \face[d=1,shpush] { { \face[d=1,shpull] { \algM } } }
			\ar[r, "\mtheta"]
				& \face[d=1,shpush] { { \face[d=0,shpull] { \algM } } }
				\ar[r]
					& \face[d=1,shpush] {{
						\degen[d=0,shpush] {{
							\degen[d=0,shpull] {{
								\face[d=0,shpull] { \algM }
							}}
						}}
					}}
					\invdot
	\end{tikzcd}\]
	The fact that this is the identity on~\( \algM \) follows from the commutative diagram
	\[\begin{tikzcd}
		\algM \ar[r] \ar[rd,equal]
			& \face[d=1,shpush] {{ \face[d=1,shpull] { \algM } }}
			\ar[r,"\mtheta"]
			\ar[d,dashed]
				& \face[d=1,shpush] {{ \face[d=0,shpull] { \algM } }}
				\ar[d]
	\\
			& \face[d=1,shpush] {{
				\degen[d=0,shpush] {{
					\degen[d=0,shpull] {{
						\face[d=1,shpull] { \algM }
					}}
				}}
			}}
			\ar[r,dashed,"{\degen[d=0,shpull]{\mtheta}}"]
				& \face[d=1,shpush] {{
					\degen[d=0,shpush] {{
						\degen[d=0,shpull] {{
							\face[d=0,shpull] { \algM }
						}}
					}}
				}}
				\ar[r,symbol={=}]
					&[-2em] \algM
	\end{tikzcd}\]
	and the assumption that~\( \degen[d=0,shpull] { \mtheta} = \id[\algM] \).
\end{proof}

\endgroup

\begingroup


\section{Descent via homotopy limits}

In this chapter, we give a small resume
of the exposition of homotopy limits
presented in~\textcite{ends}
and refer the reader to that paper for further details.
If \( \catC \)~is a model category
and~\( \catGamma \) a category,
we may consider the category of
functors~\( \catC[diag=\catGamma] = \catfun{\catGamma,\catC} \)
of functors~\( \catGamma\to\catC \),
which we shall also refer to as “diagrams”.
It makes sense to call a
map of diagrams~\( \malpha\colon\mF\to\mG \)
a \textdef{weak equivalence} if~\(
	\malpha[\vgamma]\colon\mF{\vgamma}\to\mG{\vgamma}
\)
is a weak equivalence for all~\( \vgamma \in \catGamma \).
It is natural to refer to such weak equivalences
as \textdef{componentwise} weak equivalences.
However, we immediately run into the problem
that the limit functor~\(
	\invlim
	\colon
	\catC[diag=\catGamma]
	\to
	\catC
\)
does not in general take componentwise weak equivalences
to weak equivalences in~\( \catC \).
Since \( \invlim \)~is a right adjoint,
this leads us into trying to \emph{derive} it.
The right derived functor of~\( \invlim \)
is called the \text{homotopy limit}
and is denoted~\(
	\invholim
	\colon
	\catC[diag=\catGamma]
	\to
	\catC
\).
Dually, the left derived functor of~\( \dirlim \)
is called the \textdef{homotopy colimit}
and is denoted~\(
	\dirholim
	\colon
	\catC[diag=\catGamma]
	\to
	\catC
\).

Quillen's model category machinery tells us how to derive the limit:
We must equip the diagram category~\( \catC[diag=\catGamma] \)
with a model structure with componentwise weak equivalences and in which the
limit functor~\( \invlim\colon\catC[diag=\catGamma]\to\catC \) is a right Quillen functor.
In this case, the derived functor is given by
\(
	\invholim{ \mF } = \invlim { \mfibrep ( \mF ) }
\)
for some fibrant replacement~\( \mfibrep { \mF } \)
in~\( \catC[diag=\catGamma,smash] \).
Indeed, such a model structure on~\( \catC[diag=\catGamma] \) exists e.g.\ if the model category~\( \catC \) is \emph{combinatorial}.
More precisely, we introduce
\begin{itemize}
	\item The \textdef{projective model structure}~\( \catC[diag=\catGamma,proj] \)
	where weak equivalences and fibrations are calculated componentwise.
	\item The \textdef{injective model structure}~\( \catC[diag=\catGamma,inj] \)
	where weak equivalences and cofibrations are calculated componentwise.
\end{itemize}
Denoting by \( \mconst \colon \catC \to \catC[diag=\catGamma] \)
the constant functor embedding,
we clearly see that
\(
	\mconst
	\colon
	\catC
	\rightleftarrows
	\catC[diag=\catGamma,inj,smash]
	\noloc
	\invlim
\)
is a Quillen adjunction since \( \mconst \)~preserves (trivial)
cofibrations.
Dually,~\(
	\dirlim
	\colon
	\catC[diag=\catGamma,proj,smash]
	\rightleftarrows
	\catC
	\noloc
	\mconst
\)
is a Quillen adjunction.
The injective model structure being in general rather complicated, calculating such a replacement of a diagram in practice becomes very involved for all but the simplest shapes of the category~\( \catGamma \). Therefore, traditionally, other tools have been used.


We shall only be interested in homotopy
limits over the simplex category~\( \catdelta \)
with objects finite ordered
sets~\( \ordset{n} = \set{0,1,...,n} \)
and morphisms the order-preserving (i.e.~non-decreasing)
maps.
In this case, one
of the available formulas for the homotopy
limit is the \textdef{fat totalization formula}:

\begin{propositionbreak}[{Proposition
\parencite[Example~6.4]{ends}}]\label{res:fat_tot}
	Suppose the model category~\( \catC \)
	is combinatorial
	and~\( \cosimplicial{X}{*} \colon \catdelta\to\catC \)
	a cosimplicial diagram.
	Then the homotopy limit over
	the simplex category~\( \catdelta \)
	may be calculated by the formula
	\[
		{\textstyle
		\invholim[\catdelta]{\cosimplicial{X}{*}}
		}
		=
		\End[\ordset{n}\in\catdelta[plus]]{
			\fibrep{\cosimplicial{X}{n}}[n]
		}
	\]
	where
	the integral refers to the \textdef{end}
	construction,
	and where~\(\smash{
		\fibrep
		\colon
		\catC
		\to
		\catC[diag={\catdelta[op,plus]},inj]
	}\)
	is a functor
	that
	takes~\( \vx \in \catC \) to
	an injectively fibrant replacement
	of the constant \( \catdelta[op,plus] \)-diagram at~\( \vx \).
\end{propositionbreak}

Suppose that \( \catC \)~is in fact a \emph{simplicial model category},
meaning that it is enriched, powered, and copowered
over simplicial sets.
The powering is given by a Quillen bifunctor
\[
	\catsset[op]\times\catC\longto\catC,
	\qquad
	\tup{K,\vx}
	\longmapsto
	\vx[powering=K]
	,
\]
where \( \catsset \)~denotes simplicial sets
equipped with the Quillen model structure.
Being a Quillen bifunctor means,
among other things, that it takes (trivial) cofibrations
in~\( \catsset \) to (trivial) fibrations.
Since the standard simplex~\( \simp{*} \) is a projectively cofibrant
\( \catdelta[plus,op] \)-diagram of simplicial sets, this implies
that~\( \catdelta[plus,op]\to\catC \),
\( \ordset{n} \mapsto \vx[powering=\simp{n}] \),
is a fibrant replacement functor like the one in the
proposition.


	We recall from \textcite{rezk}
	the existence of a model structure
	on the category~\(\catcat \) of categories
	with
	\begin{itemize}
		\item weak equivalences given by equivalences of categories;
		\item cofibrations given by functors which are injective on objects;
		\item fibrations given by isofibrations,
		i.e.\ functors~\( \mF \colon \catC \to \catD \)
		such that any isomorphism of the
		form~\( \mh\colon \mF{c}\to d \)
		in~\( \catD \)
		there exists a morphism~\( \mg \colon c\to c' \)
		in~\( \catC \)
		such that~\( \mF{\mg} = \mh \).
	\end{itemize}
	This model structure is combinatorial
	and simplicial.
	The powering is given by~\(
		\catC[powering=\simp{n}]
		=
		\catfun{\catiso{n},\catC}
	\),
	where \( \catiso{n} \)~denotes
	the category
	with \( n+1 \)~objects, denoted by~\( 0,1,\ldots,n \),
	and one unique arrow between any two objects
	(in particular, all arrows are isomorphisms, and we have a groupoid).

\newcosimplicial\cosimpcatC{{\catC}}

Thus \cref{res:fat_tot} shows that
\begin{proposition}\label{res:holim_cat}
	The homotopy limit~\(
		\invholim[\catdelta]{\cosimpcatC{*}}
	\)
	of a \( \catdelta \)-diagram
	of categories
	is the category with objects
	pairs~\( \tup{ \algM , \mtheta } \),
	where \( \algM \in \cosimpcatC{0} \)
	and \( \mtheta \)~is
	an isomorphism~\(
		\mtheta
		\colon
		\coface[d=1]{\algM}
		\isoto
		\coface[d=0]{\algM}
	\)
	satisfying the cocycle condition~\(
		\coface[d=0]{\mtheta}
		\circ
		\coface[d=2]{\mtheta}
		=
		\coface[d=1]{\mtheta}
	\).
	A morphism~\(
		\malpha
		\colon
		\tup{\algM,\mtheta}
		\to
		\tup{\algN,\meta}
	\)
	consists of a map~\( \malpha\colon\algM\to\algN \)
	such that~\(
		\meta\circ\coface[d=1]{\malpha}
		=
		\coface[d=0]{\malpha}\circ\mtheta
	\).
\end{proposition}

\begin{proof}
	The object set of the homotopy limit is
	\[
		\End[\ordset{n}\in\catdelta[plus]]{
			\Hom[\catcat]{\catiso{n},\cosimpcatC{n}}
		}
		,
	\]
	which is the set of natural transformations~\(
		\mtheta\colon\catiso{*}\to\cosimpcatC{*}
	\),
	i.e.~maps \( \mtheta[d=n]\colon\catiso{n}\to\cosimpcatC{n} \)
	commuting with all face maps.
	This amounts to an object~\(
		\algM = \mtheta[d=0]{0} \in \cosimpcatC{0}
	\)
	and an isomorphism~\(
		\mtheta
		=
		\mtheta[d=1]
		\colon
		\coface[d=1]{\algM}
		\isoto
		\coface[d=0]{\algM}
	\)
	satisfying the mentioned cocycle condition.
	A morphism~\( \malpha \) consists
	of a collection of natural
	transformations~\( \malpha[d=n] \colon \catiso{n}\to\cosimpcatC{n} \)
	commuting with face maps.
	This boils down to a map~\( \malpha = \malpha[d=0]\colon\algM\to\algN \)
	satisfying the mentioned condition.
\end{proof}

\begin{remark}
	We note that the conditions on~\( \mtheta \) imply that
	\( \codegen[d=0]{\mtheta} = \id \):
	Indeed, applying~\( \codegen[d=0] \codegen[d=1] = \codegen[d=0] \codegen[d=0] \)
	implies~\(
		\codegen[d=0]{\mtheta}
		\circ
		\codegen[d=0]{\mtheta}
		=
		\codegen[d=0]{\mtheta}
	\),
	which together with the isomorphism condition
	implies~\( \codegen[d=0]{\mtheta} = \id \).
	This shows that the homotopy limit in~\( \catcat \)
	is always equal to the \textdef{totalization}
	\[
		\End[\ordset{n}\in\catdelta]{
			\catfun{\catiso{n},\cosimpcatC{n}}
		}
		,
	\]
	where the end is taken over the whole simplex
	category~\( \catdelta \),
	including the degeneracy maps.
	This is a special feature of~\( \catcat \);
	in other simplicial model categories,
	this only holds in some cases, e.g. if the diagram~\( \cosimpcatC{*} \)
	is \emph{Reedy-fibrant}.
\end{remark}

\begin{corollary}\label{res:classical_descent_holim}
	For a groupoid~\( \ssetX{1} \rightrightarrows\ssetX{0} \)
	in schemes with classifying space~\( \ssetX{*} \), we have
	an equivalence of categories~\(
		\catdesc{\ssetX{*}}
		\cong
		\invholim[\catdelta]{\catqcoh{\ssetX{*}}}
	\)
\end{corollary}

\endgroup


%


\begingroup

\chapter{Differential graded categories and \texorpdfstring{\ainfty}{A\_infinity}-categories}\label{sec:differential-graded-categories}

\setupobject\catC{symbol={\categoryformatmath{C}}}

In this chapter, we mainly follow the sources
\textcite{lyu},
\textcite{keller}
and \textcite{lh},
with a few generalizations.
We follow the sign conventions
of~\textcite[section~1.1]{lh}.
In this chapter, we fix a field~\( \ringk \)
and denote by~\( \catvect = \catvect{\ringk} \) the category
of vector spaces over~\( \ringk \).


\section{Graded objects, complexes, and sign conventions}

	Consider the category~\( \catvect[graded] \)
	of \( \Z \)-graded vector spaces, which we shall write as tuples~\(
		\algM
		=
		\algM[*]
		=
		\tup{{\algM[d=n]}}[n\in\Z]
	\)
	with~\( \algM[d=n] \in \catvect \).
	The degree of a homogeneous element~\( \vx \in \algM[*] \)
	is denoted~\( \dgcatdeg{\vx} \).
	If \( \algM[*] \) and~\( \algN[*] \) are graded vector spaces,
	their hom space~\( \Hom[\ringk,*]{\algM,\algN} \) is given by
	\[
		\Hom[\ringk,d=n]{ \algM[*],\algN[*] }
		=
		\Prod[m\in\Z]{ \Hom[\ringk]{ {\algM[d=m]} , {\algN[d=m+n]} } }
		.
	\]
	The category~\( \catvect[graded] \) is monoidal with tensor product given by
	\[
		\alg{\algM\tens[\ringk]\algN}[spar,d=n]
		=
		\dirsum[ {p+q=n} ]{ \algM[d=p] \tens[\ringk] \algN[d=q] }
		.
	\]
	If \( \mf \colon \algM[1] \to \algM[2] \)~is a map of degree~\( r \)
	and~\( \mg \colon \algN[1] \to \algN[2] \)~is a map of degree~\( s \),
	then~\(
		\mf \tens \mg
		\colon
		\algM[1] \tens[\ringk] \algN[1]
		\to
		\algM[2] \tens[\ringk] \algN[2]
	\)
	is a map of degree~\( r + s \) given on the \( \tup{p,q} \)-component, where \( p + q = n \),
	by
	\[
		(-1)^{ps}
		\mf[d=p] \tens \mg[d=q]
		\colon
		\algM[1,d=p] \tens \algN[1,d=q]
		\longto
		\algM[2,d=p+r] \tens \algN[2,d=q+s]
		.
	\]
	
	If \( \algM \)~is a graded vector space, we write~\( \algM[shift=p] \)
	for its shifted vector space given by~\( \algM[shift=p,d=n] = \algM[d=n+p] \).
	We denote by~\( \ms \colon \algM\to\algM[shift=1] \)
	the map of degree~\( -1 \) which is the identity in each component.
	The inverse map is written~\( \momega = \ms[inv] \).
	Because of the above sign conventions, we note that
	\(
		\ms[tens=n]
		\colon
		\algM[tens=n]
		\to
		\algM[shift=1,tens=n]
		=
		\algM[tens=n,return,shift=n]
	\)
	is given by
	\[
		\ms[tens=n]{ \vx[n] \sotimes \cdots \sotimes \vx[1] }
		=
		(-1)^{\Sum{(i-1)\dgcatdeg{\vx[i]}}}
		\ms{\vx[n]}
		\tens
		\cdots
		\tens
		\ms{\vx[1]}
	\]
	and similarly for~\( \momega[tens=n] \).
	As a consequence, we have
	\[
		\momega[tens=n]
		\circ
		\ms[tens=n]
		=
		(-1)^{n(n-1)/2}
		\id[ { \algM[tens=n] } ]
		.
	\]
	If \( \algM \) and~\( \algN \) are graded objects
	and~\( n \ge 1 \), we shall make use of the
	bijections
	\begin{align}
	\Hom[\ringk,d=l+1-n]{{\algM[tens=n]},\algN}
		& \longisoto
		\Hom[\ringk,d=l]{{\algM[shift=1,tens=n]},{\algN[shift=1]}}
		\label{eq:coleibniz}
	\\
	\mf
		& \longmapsto
		\mf[prime]
		=
		(-1)^{l}
		\ms\circ\mf\circ\momega[tens=n]
		\nonumber
	\\
	\shortintertext{and}
	\Hom[\ringk,d=l+1-n]{\algM,{\algN[tens=n]}}
		& \longisoto
		\Hom[\ringk,d=l]{{\algM[shift=-1]},{\algN[shift=-1,tens=n]}}
		\label{eq:leibniz}
	\\
	\mf
		& \longmapsto
		\mf[prime] = (-1)^{l} \momega[tens=n]\circ\mf\circ\ms
		\nonumber
		.
	\end{align}

	We denote by~\( \catcom{\ringk} \)
	the category of complexes of \( \ringk \)-vector spaces.
	We use cohomological notation, so to us, a complex is a graded
	vector space~\( \algM[*] \in \catvect[graded] \)
	equipped with a differential
	map~\( \dif \colon \algM \to \algM[shift=1] \)
	with~\( \dif[squared]=0 \).
	If \( \algM \) and~\( \algN \) are complexes, the graded hom
	space~\( \Hom[*,\ringk]{\algM,\algN} \)
	becomes a complex with the differential
	\(\smash{
		\dif[par]{\mf}
		=
		\dif[\algN] \circ \mf
		-
		(-1)^{\dgcatdeg{\mf}} \mf \circ \dif[\algM]
	}\).
	The tensor product~\( \algM \tens[\ringk] \algN \) of graded objects becomes a complex
	via~\(\smash{
		\dif[\algM \tens[\ringk] \algN]
		=
		\dif[\algM]
		\tens
		\id[\algN]
		+
		\id[\algM]
		\tens
		\dif[\algN]
	}
	\).
	The shifted object~\( \algM[shift=n] \) becomes a complex
	via the differential~\(\smash{
		\dif[ {\algM[shift=n,d=i]} ]
		=
		(-1)^{n}
		\dif[ {\algM[d=n+i]} ]
	}\).
	
	Most of the following considerations make sense both
	in the graded and differential graded~(dg-)~case.
	To treat the two cases simultaneously,
	we let the symbol~\( \cocatstar \) stand
	for any of the symbols~\( \categoryformat{gr} \)
	and~\( \categoryformat{dg} \)
	(and for the terms “graded” and “dg-”).
	We shall use the notation~\( \catxstmod{\ringk} \)
	to refer to either~\( \catvect[graded] \)
	or~\( \catcom{\ringk} \).



\section{Graded quivers and dg-quivers}

	If \( \setobjE \)~is a set, we may regard it as a discrete category.
	Even though~\( \setobjE[op] = \setobjE \), we shall
	use both \( \setobjE \) and~\( \setobjE[op] \) in the
	following.
	The category~\( \catxstmod{\ringk[\setobjE]} \) of \textdef{left \( \ringk[\setobjE] \)-\cocatstar-modules}
	is the category of functors~\( \setobjE[op]\to\catxstmod{\ringk} \),
	and 
	the category~\( \catstmodx{\ringk[\setobjE]} \) of \textdef{right \( \ringk[\setobjE] \)-\cocatstar-modules}
	is the category of functors~\( \setobjE\to\catxstmod{\ringk} \)
	(the two categories are equal, for the time being).
	If \( \setobjF \)~is another set,
	the category~\( \catxstmodx{\ringk[\setobjE]}{\ringk[\setobjF]} \)
	of \textdef{\( \ringk[\setobjE] \)--\( \ringk[\setobjF] \)-\cocatstar-bimodules}
	is the category of functors~\(
		\setobjE[op]\times\setobjF\to\catxstmod{\ringk}
	\).
	We define a tensor product by convolution, i.e.
	\[
		\tens[\ringk[\setobjE]]
		\colon
		\catstmodx{\ringk[\setobjE]}
		\times
		\catxstmod{\ringk[\setobjE]}
		\longto
		\catxstmod{\ringk}
		,
		\qquad
		\tup{\dgcatC,\dgcatD}
		\longmapsto
		\dirsum[s\in\setobjE]{
			\dgcatC{s}\tens[\ringk]\dgcatD{s}
		}.
	\]
	This extends in a natural way to the case when one or both
	of \( \dgcatC \) or~\(\dgcatD \) are \cocatstar-bimodules.
	In particular,
	\(
		\tup{
			\catxstmodx{\ringk[\setobjE]}{\ringk[\setobjE]} ,
			\tens[\ringk[\setobjE]]
		}
	\)~becomes a monoidal category.
	It is also unital, with unit the bimodule~\( \ringk[\setobjE] \)
	defined by~\(
		\ringk[\setobjE,arg={s,s}] = \ringk
	\)
	and~\( \ringk[\setobjE,arg={s,t}] = 0 \) if~\( s\neq t \).
	This category acts on~\( \catxstmod{\ringk[\setobjE]} \)
	from the left and on~\( \catstmodx{\ringk[\setobjE]} \)
	from the right.
	All of these allow a shift operation~\( \dgcatC \mapsto \dgcatC[shift=n] \), \( n\in\Z\), the shift being applied componentwise.

	We shall also refer to the
	monoidal category~\( \catxstmodx{\ringk[\setobjE]}{\ringk[\setobjE]} \)
	as \textdef{\cocatstar-quivers}
	over~\( \ringk \) with object set~\( \setobjE \)
	and write it as~\( \catstquiv{E} \).
	A \textdef{morphism~\( \dgcatC \to \dgcatD\) of degree~\( n \)} is a
	morphism~\( \dgcatC \to \dgcatD[shift=n] \).
	If \( \dgcatC \) and~\( \dgcatD \)~are
	augmented \cocatstar-quivers,
	the tensor product~\(
		\dgcatC \tens[\ringk] \dgcatD
	\)
	has objects~\(
		\catob{\dgcatC}\times\catob{\dgcatD}
	\)
	and morphism spaces given by
	\[
		\algmap{\dgcatC\tens[\ringk]\dgcatD}[spar]{
			\tup{c,c'},
			\tup{d,d'}
		}
		=
		\dgcatC{c,c'}
		\tens[\ringk]
		\dgcatD{d,d'}
		.
	\]

	The \cocatstar-quiver~\( \dgcatC \) is \textdef{augmented}
	if there is a chain of maps of
	\cocatstar-quivers~\(
		\ringk[\setobjE]
		\xto{\meta}
		\dgcatC
		\xto{\epsilon}
		\ringk[\setobjE]
	\)
	whose composition is the identity.
	In that case, its \textdef{reduction}~\( \dgcatC[red,smash] \)
	is given by~\( \dgcatC[red,smash] = \Coker{\meta} \cong \Ker{ \map{\epsilon} } \),
	calculated in each degree.
	This gives us a canonical splitting~\(
		\dgcatC
		=
		\dgcatC[red,smash]
		\oplus
		\ringk[\setobjE]
	\).
	Conversely, the \textdef{augmentation} of a non-augmented \cocatstar-quiver
	is the augmented \cocatstar-quiver~\( \dgcatC[augment,smash] = \dgcatC\oplus\ringk[\setobjE] \).
	If \( \dgcatC \) and~\( \dgcatD \) are augmented \cocatstar-quivers with object set~\( \setobjE \),
	a \textdef{morphism of augmented \cocatstar-quivers}
	is a morphisms~\( \dgcatC \to \dgcatD \)
	that respects \( \meta \) and~\( \map{\epsilon} \).
	This allows us to define a category~\( \catstquiv[aug]{E} \) of augmented \cocatstar-quivers. It is equivalent to~\( \catstquiv{E} \) via reduction, but
	this equivalence is not monoidal.

	The collection~\( \catstquiv \) of \cocatstar-quivers for all choices of the set~\( \setobjE \) also form a category. Namely, if \( \dgcatC \in \catstquiv { E } \) and~\( \dgcatD \in \catstquiv { F } \), a morphism
	\( \mf \colon \dgcatC\to\dgcatD \) consists of a
	map of sets~\( \mf \colon\setobjE\to\setobjF \)
	and a morphism~\(
		\dgcatC\to \mf[qpull] { \dgcatD }
	\)
	in~\( \catstquiv{ E } \). Here
	\( \mf[qpull] { \dgcatD } \)~denotes
	the composition \( \dgcatD \circ \map { \mf \times \mf }[spar] \),
	with \( \dgcatD \) regarded as a functor~\( \setobjF[op] \times \setobjF \to \catxstmod{\ringk} \).
	The category~\( \catstquiv[aug] \) of augmented \cocatstar-quivers for all choices of~\( \setobjE \) is defined analogously.

	If \( \mf , \mg \)~are morphisms \( \dgcatC \to \dgcatD \) of augmented \cocatstar-quivers as above, a \textdef{natural transformation}
	of degree~\( n \)
	consists of a
	morphism~\(
		\malpha
		\colon
		\dgcatC
		\to \map{\mf\times\mg}[spar,qpull]{ \dgcatD[red] }[symbolputright={}{}{}{}{},shift=n]
	\)
	in~\( \catstquiv { \setobjE } \),
	where~\(\smash{
		\map{\mf\times\mg}[spar,qpull]{\dgcatD[red]}
		=
		\dgcatD[red]\circ\map{\mf\times\mg}[spar]
	}\).
	This allows us to define an internal
	hom in the category~\( \catstquiv \),
	denoted~\( \cathom[\catstquiv[aug]] { \dgcatC , \dgcatD } \),
	with objects the morphisms~\( \dgcatC \to \dgcatD \)
	and morphism 
	space
	whose reduction
	consists of natural transformations,
	i.e.
	\begin{align*}
		\cathom[\catstquiv[aug]]{\dgcatC,\dgcatD}[red,d=n]{\mf,\mg}
			&=
			\Hom[\catstquiv{E}]{
				\dgcatC,
				{ \map{\mf\times\mg}[spar,qpull]{\dgcatD[red]}[shift=n] }
			}
	\\
			&=
			\Prod[\vx,\vy\in\setobjE]{
				\Hom[d=n,\ringk,par=\big]{
					\dgcatC[*]{ \vx , \vy } ,
					\dgcatD[red,*]{ \mf{ \vx } , \mg{ \vy } }
				}
			}
	.
	\end{align*}
	In the dg-case, the differential is applied componentwise.
	This turns~\( \catstquiv[aug] \) into a closed
	monoidal category, since we have
	\[
		\Hom[\catstquiv[aug]]{
			\dgcatC\tens[\ringk]\dgcatD,
			\dgcatE
		}
		\cong
		\Hom[\catstquiv[aug]]{
			\dgcatC,
			\cathom[\catstquiv[aug]]{
				\dgcatD,
				\dgcatE
			}
		}
	\]
	\parencite[Lemma~5.1]{keller}.

\section{Graded categories and dg-categories}%
\label{para:dg-categories_as_dg-algebras}

	A \textdef{\cocatstar-category} with object set~\( \setobjE \) is an associative algebra in the monoidal
	category~\(
		\tup {
			\catstquiv { E },
			\tens[\ringk[\setobjE]],
			\ringk[\setobjE]
		}
		=
		\tup{
			\catxstmodx{\ringk[\setobjE]}{\ringk[\setobjE]},
			\tens[\ringk[\setobjE]],
			\ringk[\setobjE]
		}
	\).
	This means exactly that we have a composition
	operation~\( \mult \colon \dgcatA\tens[\ringk[\setobjE]]\dgcatA\to \dgcatA \)
	satisfying associativity.
	Notice that 
	this
	splits into components
	\( \dgcatA{t,u}\tens \dgcatA{s,t}\to \dgcatA{s,u} \)
	and becomes a composition operation in the categorical sense.
	The \cocatstar-category~\( \dgcatA \) is \textdef{unital} if it is unital
	as an algebra in the above sense, i.e. if
	it is also equipped with a unit map~\( \unitmap \colon \ringk[\setobjE]\to\dgcatA \)
	satisfying the usual the unity axiom.
	By a “\cocatstar-category”, we shall mean a unital \cocatstar-category.
	The category of such will be denoted~\( \catstalg{\ringk[\setobjE]} \).
	The \textdef{opposite category}~\( \dgcatA[op] \)
	of \cocatstar-category~\( \dgcatA \) has the same objects and
	morphism spaces as~\( \dgcatA \), but
	\( \mf \circ \mg \)~in~\( \dgcatA[op] \)
	is defined to be~\( (-1)^{\dgcatdeg{\mf}\dgcatdeg{\mg}}\mg\circ\mf \)
	in~\( \dgcatA \).

	We want to turn the category of all \cocatstar-categories,
	for varying sets~\( \setobjE \),
	into a category.
	If \( \mf \colon \setobjE \to \setobjF \)~is a map of sets,
	we obtain
	a functor~\(
		\mf[qpull]
		\colon
		\catxstmodx{\ringk[\setobjF]}{\ringk[\setobjF]}
		\to
		\catxstmodx{\ringk[\setobjE]}{\ringk[\setobjE]}
	\)
	given by restriction.
	This functor is lax monoidal:
	Indeed, we obtain a map~\( \ringk[\setobjE]\to\mf[qpull]{\ringk[\setobjF]} \)
	in~\( \catstalg{\ringk[\setobjE]} \),
	and
	if \( \algM,\algN \in \catxstmodx{\ringk[\setobjF]}{\ringk[\setobjF]} \),
	we obtain maps~\(
		\mf[qpull]{\algM}
		\tens[\ringk[\setobjE]]
		\mf[qpull]{\algN}
		\to
		\mf[qpull,par]{ \algM \tens[\ringk[\setobjF]] \algN }
	\).
	This implies that \( \mf[qpull] \)~induces
	a pullback functor~\(
		\mf[qpull]
		\colon
		\catstalg{\ringk[\setobjF]}
		\to
		\catstalg{\ringk[\setobjE]}
	\).
	Given \cocatstar-categories \( \dgcatA \in \catstalg{\ringk[\setobjE]} \)
	and~\( \dgcatB \in \catstalg{\ringk[\setobjF]} \)
	with different sets of objects,
	we can thus define
	a \textdef{\cocatstar-functor}~\( \mF\colon\dgcatA\to\dgcatB \)
	between \cocatstar-categories to be a
	map of sets~\( \mF \colon \catob{\dgcatA}\to\catob{\dgcatB} \)
	along with a map~\( \mF \colon \dgcatA \to \mF[qpull]{\dgcatB} \)
	in~\( \catstalg{\ringk[\setobjE]} \).
	The \cocatstar-functors form the morphisms
	in the category~\( \catstcat = \catstcat{\ringk} \)
	of all \cocatstar-categories over~\( \ringk \).

	The category of \cocatstar-categories is monoidal:
	The tensor product~\( \dgcatA \tens[\ringk] \dgcatB \)
	has objects~\( \catob{\dgcatA}\times\catob{\dgcatB} \),
	and we write a double~\( \tup{a,b} \)
	as~\( a\tens b \). The morphism space is given by
	\[
		\alg{\dgcatA\tens[\ringk]\dgcatB}[spar,arg={a\tens a',b\tens b'}]
		=
		\dgcatA{a,a'}
		\tens[\ringk]
		\dgcatB{b,b'}
		.
	\]
	The composition is given by
	\[
		(\mf[prime]\tens\mg[prime])
		\circ
		(\mf\tens\mg)
		=
		(-1)^{\dgcatdeg{\mg[prime]}\dgcatdeg{\mf}}
		(\mf[prime]\circ\mf)
		\tens
		(\mg[prime]\circ\mg)
		.
	\]
	It is in fact a closed monoidal category:
	The internal hom~\( \catstfun{\dgcatA,\dgcatB} \)
	has objects the set of \cocatstar-functors~\(\mF\colon\dgcatA\to\dgcatB \).
	The morphism space~\( \catstfun{\dgcatA,\dgcatB}{\mF,\mG} \)
	consists of \textdef{natural \cocatstar-transformations}.
	A natural \cocatstar-transformation~\( \malpha\colon\mF\to\mG \)
	of degree~\( d \)
	consists of a collection~\( \malpha = \tup{\malpha[\vx]}[\vx\in\dgcatA] \)
	of maps~\( \malpha[x] \in \dgcatB[d=d]{\mF{\vx},\mG{\vx}} \)
	such that~\(
		\malpha[y] \circ \mF{\mf}
		=
		(-1)^{dk}
		\mG{\mf} \circ \malpha[x]
	\)
	for all~\( \mf \in \dgcatA[d=k]{\vx,\vy} \).
	Composition is applied componentwise,~\(
		\map{\mbeta\circ\malpha}[spar,x]
		=
		\mbeta[x] \circ \malpha[x]
	\).
	In the dg-case, the differential is also applied componentwise,
	i.e.~\( \dif{\malpha}[spar,\vx] = \dif[par]{\malpha[x]} \).

	A \textdef{left \( \dgcatA \)-\cocatstar-module}
	is a left module~\( \algM \in \catxstmod{\ringk[\setobjE]} \)
	over the algebra~\( \dgcatA \in \catxstmodx{\ringk[\setobjE]}{\ringk[\setobjE]} \) in the categorical sense.
	Thus we require a bifunctor~\(
		\dgcatA
		\tens[\ringk[\setobjE]]
		\algM
		\to
		\algM
	\)
	satisfying the usual associativity and unity conditions,
	and morphisms must respect this structure.
	By the closed monoidal structure, this is the same
	as a \cocatstar-functor~\( \dgcatA\to\catxstmod{\ringk} \),
	which allows us to define
	the \cocatstar-category~\(
		\catxstmod{\dgcatA}
		=
		\catstfun{\dgcatA,\catxstmod{\ringk}}
	\) of left \( \dgcatA \)-modules.
	Similarly, the
	a \textdef{right \( \dgcatA \)-\cocatstar-module}
	is a right module~\( \algN \in \catstmodx{\ringk[\setobjE]} \)
	in the categorical sense,
	and the \cocatstar-category of such is~\( \catstmodx{\dgcatA} = \catstfun{\dgcatA[op],\catxstmod{\ringk}} \).
	If \( \dgcatB \)~is another \cocatstar-category with object set~\( \setobjF \), 
	an \textdef{\(\dgcatA\)--\( \dgcatB \)-\cocatstar-bimodules}
	is a bimodule in the category~\( \catxstmodx{\ringk[\setobjE]}{\ringk[\setobjF]} \),
	and the \cocatstar-category of such
	is~\(
		\catxstmodx{\dgcatA}{\dgcatB}
		=
		\catstfun{\dgcatA\tens[\ringk]\dgcatB[op],\catxstmod{\ringk}}
	\).
	We obtain a paring
	\begin{align*}
		\tens[\dgcatA]
		\colon
		\catstmodx{\dgcatA}
		\tens[\ringk]
		\catxstmod{\dgcatA}
		&\longto
		\catxstmod{\ringk}
	\\
		\tup{\algM,\algN}
		&
		\longmapsto
		\Coeq{
			\algM
			\tens[\ringk[\setobjE]]
			\dgcatA
			\tens[\ringk[\setobjE]]
			\algN
			\rightrightarrows
			\algM
			\tens[\ringk[\setobjE]]
			\algN
		}
	\end{align*}
	which extends to a paring
	\[
		\tens[\dgcatA]
		\colon
		\catxstmodx{\dgcatC}{\dgcatA}
		\tens[\ringk]
		\catxstmodx{\dgcatA}{\dgcatB}
		\longto
		\catxstmodx{\dgcatC}{\dgcatB}
		.
	\]
	In particular, we obtain a monoidal
	category~\(
		\tup{
			\catxstmodx{\dgcatA}{\dgcatA},
			\tens[\dgcatA]
		}
	\).
	The unit is~\( \dgcatA \), regarded as an \( \dgcatA \)--\( \dgcatA \)-bimodule.
	

\newalgmap\dgalgA{A}
\newalgmap\dgalgB{B}


\section{Graded algebras and dg-algebras over categories}

	Suppose as above that \( \dgcatA \)~is a \cocatstar-category.
	A (unital) \textdef{\cocatstar-algebra} over~\( \dgcatA \)
	is a (unital) associative algebra in the monoidal
	category~\(
		\tup{
			\catxstmodx{\dgcatA}{\dgcatA} ,
			\tens[\dgcatA],
			\dgcatA
		}
	\).
	A \cocatstar-algebra is assumed to be unital unless explicitly stated otherwise. The category of \( \dgcatA \)-\cocatstar-algebras
	is denoted~\( \catstalg{\dgcatA} \).
	We note that the category of \cocatstar-categories
	with object~\( \setobjE \)
	is equal to~\( \catstalg{\ringk[\setobjE]} \),
	so the notation is consistent with the one defined above.
	The category of non-unital \cocatstar-algebras
	over~\( \dgcatA \)
	is denoted~\( \catstalg[nu]{\dgcatA} \).

	Above we defined morphisms between \cocatstar-categories
	with different objects.
	We can more generally define morphisms between \cocatstar-algebras
	over different \cocatstar-categories.
	If \( \mf \colon \dgcatA[1] \to \dgcatA[2] \)~is
	a \cocatstar-functor between \cocatstar-categories,
	we obtain a restriction functor~\(
		\mf[qpull]
		\colon
		\catxstmodx{\dgcatA[2]}{\dgcatA[2]}
		\to
		\catxstmodx{\dgcatA[1]}{\dgcatA[1]}
	\).
	We claim that this functor is lax monoidal,
	i.e.\ that we have maps~\( \dgcatA[1]\to\mf[qpull]{\dgcatA[2]} \)
	and~\(
		\mf[qpull]{\algM}
		\tens[\dgcatA[1]]
		\mf[qpull]{\algN}
		\to
		\mf[qpull,par]{\algM\tens[\dgcatA[2]]\algN}
	\)
	in the category~\( \catxstmodx{\dgcatA[1]}{\dgcatA[1]} \)
	for all~\( \algM , \algN \in \catxstmodx{\dgcatA[2]}{\dgcatA[2]} \),
	satisfying the usual conditions.
	The first map comes from the definition of a \cocatstar-functor,
	noting that it is in fact a map
	in~\( \catxstmodx{\dgcatA[1]}{\dgcatA[1]} \).
	The construction of the second map follows from the universal
	property of the tensor product (we write \( \setobjE[1] = \catob{\dgcatA[1]} \) and~\( \setobjE[2] = \catob{\dgcatA[2]} \)):
	\[\begin{tikzcd}[column sep=small]
		\displaystyle
		\mf[qpull]{\algM}
		\tens[limits,\mathclap{\ringk[\setobjE[1]]},return,bin]
		\dgcatA[1]
		\tens[limits,\mathclap{\ringk[\setobjE[1]]},return,bin]
		\mf[qpull]{\algN}
		\ar[r]
			&
			\mf[qpull]{\algM}
			\tens[limits,\mathclap{\ringk[\setobjE[1]]},return,bin]
			\mf[qpull]{\dgcatA[2]}
			\tens[limits,\mathclap{\ringk[\setobjE[1]]},return,bin]
			\mf[qpull]{\algN}
			\ar[r,yshift=1.5pt]
			\ar[r,yshift=-1.5pt]
			\ar[d]
				&
				\mf[qpull]{\algM}
				\tens[limits,\mathclap{\ringk[\setobjE[1]]},return,bin]
				\mf[qpull]{\algN}
				\ar[r]
				\ar[d]
					&
					\mf[qpull]{\algM}
					\tens[limits,\mathclap{\dgcatA[1]},return,bin]
					\mf[qpull]{\algN}
					\ar[d,dashed]
	\\
			&
			\mf[qpull,par]{{
				\algM
				\tens[limits,\mathclap{\ringk[\setobjE[2]]},return,bin]
				\dgcatB
				\tens[limits,\mathclap{\ringk[\setobjE[2]]},return,bin]
				\algN
			}}
			\ar[r,yshift=1.5pt]
			\ar[r,yshift=-1.5pt]
				&
				\mf[qpull,par]{{
					\algM
					\tens[limits,\mathclap{\ringk[\setobjE[2]]},return,bin]
					\algN
				}}
				\ar[r]
					&
					\mf[qpull,par]{{
						\algM
						\tens[limits,\mathclap{\dgcatA[2]},return,bin]
						\algN
					}}
	\invdot
	\end{tikzcd}\]

	Since the functor is lax monoidal,
	it induces a functor~\(
		\mf[qpull]\colon \catstalg{\dgcatA[2]}\to\catstalg{\dgcatA[1]}
	\).
	If \( \dgcatB[1] \in \catstalg{\dgcatA[1]} \)
	and~\( \dgcatB[2] \in \catstalg{\dgcatA[2]} \)
	are \cocatstar-algebras over two different \cocatstar-categories,
	a \textdef{morphism of \cocatstar-algebras}~\(
		\mf
		\colon
		\tup{\dgcatB[1],\dgcatA[1]}
		\to
		\tup{\dgcatB[2],\dgcatA[2]}
	\)
	consists of a \cocatstar-functor~\( \mf\colon\dgcatA[1]\to\dgcatA[2] \)
	along with a morphism of
	algebras~\(
		\dgcatB[1]\to\mf[qpull]{\dgcatB[2]}
	\)
	in~\( \catstalg{\dgcatA[1]} \).
	In particular,
	restricting along the unit~\(
		\mf=\unitmap\colon\ringk[\catob{\dgcatA}]\to\dgcatA
	\),
	we obtain an embedding~\(
		\catstalg{\dgcatA}
		\into
		\catstalg{\ringk[\catob{\dgcatA}]}
	\).
	Thus we may equivalently
	regard the category of \cocatstar-categories
	as consisting of the collection of all \cocatstar-algebras
	over all \cocatstar-categories,
	each~\( \dgcatB \in \catstalg{\dgcatA} \)
	being identified with its image in~\( \catstalg{\ringk[\catob{\dgcatA}]} \).
	This somewhat circular-looking definition will be
	the one we shall later mimic in our other definitions.

	A (unital) \cocatstar-algebra~\( \dgcatB \)
	over a \cocatstar-category~\( \dgcatA \)
	is~\textdef{augmented}
	if there exists a morphism of
	\(\dgcatA\)-\cocatstar-algebras~\( \dgcatB \to \dgcatA \)
	such that
	\(
		\dgcatA
		\xto{\unitmap}
		\dgcatB
		\xto{\aug}
		\dgcatA
	\)~is the identity.
	In this case, the reduction is given by~\( \dgcatB[red,smash] = \Ker{\aug} \),
	and we obtain the splitting~\( \dgcatB = \dgcatB[red,smash]\oplus\dgcatA \) in~\( \catxstmodx{\dgcatA}{\dgcatA} \).
	Conversely, if \( \dgcatB \)~is a non-augmented, non-unital \cocatstar-algebra, its \textdef{augmentation}
	is the augmented \cocatstar-algebra~\(
		\dgcatB[augment]
		=
		\dgcatB
		\oplus
		\dgcatA
	\)
	with augmentation given by the projection~\(
		\dgcatB[augment]
		\to
		\dgcatB
	\).
	The composition map~\(
		\mult[\dgcatB[augment]]
		\colon
		\dgcatB[augment]
		\tens[\dgcatA]
		\dgcatB[augment]
		\to
		\dgcatB[augment]
	\)
	is given by~\(
		\mult[\dgcatB[augment]]
		=
		\aug\tens\id[\dgcatB]
		+
		\mult[\dgcatB]
		+
		\mult[\dgcatA]
		+
		\id[\dgcatB]\tens\aug
	\).
	A \textdef{morphism of augmented \cocatstar-algebras over~\(\dgcatA\)}
	is a morphism of unital \cocatstar-algebras
	which commutes with the augmentation.
	Thus we obtain a
	category~\( \catstalg[aug]{\dgcatA} \)
	of augmented \( \dgcatA \)-\cocatstar-algebras.
	Given
	\( \dgcatB[1] \in \catstalg[aug]{\dgcatA[1]} \)
	and~\( \dgcatB[2] \in \catstalg[aug]{\dgcatA[2]} \),
	a morphism~\(
		\mf
		\colon
		\tup{\dgcatB[1],\dgcatA[1]}
		\to
		\tup{\dgcatB[2],\dgcatA[2]}
	\)
	of augmented \cocatstar-algebras over different \mbox{\cocatstar-categories}
	is a morphism of unital algebras making the square
	\[\begin{tikzcd}[sep=scriptsize]
		\dgcatB[1] \ar[r] \ar[d]
			& \mf[qpull]{\dgcatB[2]} \ar[d]
	\\
		\dgcatA[1] \ar[r]
			& \mf[qpull]{\dgcatA[2]}
	\end{tikzcd}\]
	commutative.
	Equivalently, it
	consists of a \cocatstar-functor~\( \mf \colon \dgcatA[1]\to\dgcatA[2] \)
	together with a morphism~\(
		\mf
		\colon
		\dgcatB[1]
		\to
		\mf[qpull]{ \dgcatB[2,red] }[spar,augment,smash]
	\)
	in~\( \catstalg[aug]{\dgcatA[1]} \).

	In the case~\( \dgcatA = \ringk[\setobjE] \),
	we obtain the definition of an augmented \cocatstar-category.
	If \( \dgcatB \)~is an augmented \cocatstar-category, it is in particular augmented as a \cocatstar-quiver.

\subsection{Tensor algebra}

	It is possible to freely generate a non-unital \( \dgcatA \)-\cocatstar-algebra from an arbitrary \( \dgcatA \)-\cocatstar-bimodule~\( \algV \), namely the \textdef{non-unital tensor category}~\( \tensoralg[red]{\algV}\in\catstalg{ \dgcatA } \) given by
	\[
		\tensoralg[red]{\algV}
		=
		\dirsum[n\ge1] {
			\algV
			\tens[\dgcatA]
			\algV
			\tens[\dgcatA]
			\cdots
			\tens[\dgcatA]
			\algV
		}
		\qquad\text{(\( n \) factors),}
	\]
	equipped with the multiplication
	map~\(
		\mult
		\colon
		\tensoralg[red]{\algV}
		\tens[\dgcatA]
		\tensoralg[red]{\algV}
		\to
		\tensoralg[red]{\algV}
	\)
	given by
	\(
		\mult{{
			(\vx[n]\tens\cdots\tens\vx[i+1])
			\tens
			(\vx[i]\tens\cdots\tens\vx[1])
		}}
		=
		\vx[n]\tens\cdots\tens\vx[i+1]
		\tens
		\vx[i]\tens\cdots\tens\vx[1]
	\).
	A non-unital \cocatstar-algebra is called \textdef{free} if it is (isomorphic to a \cocatstar-algebra) of this form.
	One may also define the \textdef{tensor algebra}~\( \tensoralg{\algV} = \tensoralg[red]{\algV}[augment] \).
	A unital \cocatstar-algebra is called \textdef{free}
	if it is (isomorphic to a \cocatstar-algebra)
	of this form.
%
	The proposition below shows why the term “free” makes sense.
	
	If \( \mf , \mg \colon \tup{\dgcatB[1],\dgcatA[1]} \to \tup{\dgcatB[2],\dgcatA[2]} \)~are morphisms of \cocatstar-algebras,
	we define an \textdef{\(\tup{\mf,\mg}\)-derivation} of degree~\( i \)
	to be a map~\( \map{D}\colon\dgcatB[1]\to\map{\mf\times\mg}[spar,qpull]{\dgcatB[2]}[shift=i] \)
	in~\( \catxstmodx{\dgcatA}{\dgcatA} \)
	such that~\(
		\map{D}\circ\mult
		=
		\map{\mf\times\mg}[spar,qpull]{\mult}\circ(\mf\tens\map{D} + \map{D}\tens\mg)
	\).
	We write~\( \Der[d=i]{\mf,\mg} \) for the set of such.
	A \textdef{derivation} is an \( \tup{\id,\id} \)-derivation.
	A dg-algebra may be considered as a graded algebra
	equipped with a derivation~\( \dif \) satisfying~\( \dif[squared]=0 \).
	
	\begin{propositionbreak}[Proposition
	{\parencite[Lemme~1.1.2.1]{lh}}]
		\begin{propositionlist}
			\item\label{res:free_morphisms}
			If \( \dgcatA \)~is a \cocatstar-category,
			\( \algV \in \catxstmodx{\dgcatA}{\dgcatA} \),
			and~\( \dgcatB \in \catstalg[aug]{\dgcatA} \),
			we have an isomorphism of sets
			\[
				\Hom[\catstalg[aug]{\dgcatA}]{\tensoralg{\algV},\dgcatB}
				\longisoto
				\Hom[\catxstmodx{\dgcatA}{\dgcatA}]{\algV,\dgcatB[red]}
			\]
			given by precomposition with the
			inclusion~\( \algV\into\tensoralg{\algV} \).
			The inverse map takes the map of bimodules~\( \mf \colon \algV \to \dgcatB[red,smash] \)
			to the map
			\[
				\Sum[n\ge0]{ \mult[iterate=n]\circ\mf[tens=n] }
				\colon
				\tensoralg{\algV}\to\dgcatB
			.\]
			Here \( \mult[iterate=n] \)~denotes the \( n \)th~iterate of~\( \mult \).
			\item\label{res:derivations_free}
			If \( \dgcatA[1] , \dgcatA[2] \)~are \cocatstar-categories,
			\( \dgcatB[2] \in \catstalg{\dgcatA[2]} \),
			\( \algV \in \catxstmodx{\dgcatA[1]}{\dgcatA[1]} \),
			and we are given two maps
			of \cocatstar-algebras~\(
				\mf , \mg
				\colon
				\tup{ \tensoralg{\algV},\dgcatA[1]}
				\to
				\tup{ \dgcatB[2],\dgcatA[2]}
			\),
			precomposition with the inclusion~\(
				\algV
				\into
				\tensoralg{\algV}
			\)
			yields an isomorphism of sets
			\[
				\Der[d=i]{\mf,\mg}
				\longisoto
				\Hom[\catxstmodx{\dgcatA}{\dgcatA}]{
					\algV,
					{\map{\mf\times\mg}[spar,qpull]{\dgcatB[2]}[shift=i]}
				}.
			\]
			The map in the opposite direction takes
			a map~\(
				\map{h} 
				\colon
				\algV
				\to
				\map{\mf\times\mg}[spar,qpull]{\dgcatB}[2,shift=i]
			\)
			to the \( \tup{\mf,\mg} \)-derivation~\( \tensoralg{\algV}\to\dgcatB[2,shift=i] \)
			whose \( n \)th 
			component is
			\[
				\mult[iterate=n]
				\circ
				\map{
					\Sum[from={i+1+j=n}]{
						\mf[tens=i]
						\tens
						\map{h}
						\tens
						\mg[tens=j]
					}
				}[spar=\big]
				.
			\]
	\end{propositionlist}
	\end{propositionbreak}

\section{Graded coalgebras and dg-coalgebras}

	If \( \dgcatA \)~is a \cocatstar-category,
	the category of \textdef{\cocatstar-coalgebras} over~\(\dgcatA \)
	is the category of coalgebras~\( \tup{\dgcatC,\comult} \) in the monoidal
	category~\(
		\tup{
			\catxstmodx{\dgcatA}{\dgcatA},
			\tens[\dgcatA],
			\dgcatA
		}
	\).
	In other words, we require a
	comultiplication map~\(
		\comult
		\colon
		\dgcatC
		\to
		\dgcatC \tens[\dgcatA] \dgcatC
	\) satisfying the usual coassociativity axiom.
	It is \textdef{counital}
	if we are also given a counit map~\(
		\counit\colon \dgcatC\to\dgcatA 
	\)
	satisfying the counit axiom.
	We shall assume by default that \cocatstar-coalgebras are counital.
	A morphism of counital \( \dgcatA \)-\cocatstar-coalgebras
	is a morphism that commutes with the unit.
	We call~\( \dgcatC \) \textdef{cocomplete}
	if any element is annihilated by sufficiently many iterations
	of~\( \comult \),
	where by an “element”, we mean any
	morphism in~\( \dgcatC{a,a'} \) for~\( a,a' \in \catob{\dgcatA} \).
	The category of cocomplete, counital \cocatstar-coalgebras
	over~\( \dgcatA \)
	is denoted~\( \catstcoalg{\dgcatA} \).
	The category of cocomplete, non-counital \cocatstar-coalgebras
	over~\( \dgcatA \)
	is denoted~\( \catstcoalg[ncu]{\dgcatA} \).

	If \( \mf \colon \dgcatA[1] \to \dgcatA[2] \)~is
	a \cocatstar-functor, the restriction functor~\(
		\mf[qpull]
		\colon
		\catxstmodx{\dgcatA[2]}{\dgcatA[2]}
		\to
		\catxstmodx{\dgcatA[1]}{\dgcatA[1]}
	\)
	has a left adjoint~\(
		\mf[q!]
		\colon
		\catxstmodx{\dgcatA[1]}{\dgcatA[1]}
		\to
		\catxstmodx{\dgcatA[2]}{\dgcatA[2]}
	\)
	given by
	\[
		\mf[q!]{\algM}
		=
		\dgcatA[2]
		\tens[\dgcatA[1]]
		\algM
		\tens[\dgcatA[1]]
		\dgcatA[2]
		.
	\]
	It is oplax monoidal:
	Indeed, by adjunction, we obtain a
	map~\( \mf[q!]{\dgcatA[1]} \to \dgcatA[2] \)
	in~\( \catxstmodx{\dgcatA[2]}{\dgcatA[2]} \),
	and if~\( \algM , \algN \in \catxstmodx{\dgcatA[1]}{\dgcatA[1]} \),
	the unit of adjunction provides us with a
	map~\(
		\algM
		\tens[\dgcatA[1]]
		\algN
		\to
		\mf[qpull]{ \mf[q!]{ \algM } }
		\tens[\dgcatA[1]]
		\mf[qpull]{ \mf[q!] { \algN } }
		\to
		\mf[qpull,par]{
			\mf[q!]{\algM}
			\tens[\dgcatA[2]]
			\mf[q!]{\algN}
		}
	\),
	which by adjunction is the same as a
	map~\(
		\mf[q!,par]{ \algM \tens[\dgcatA[1]] \algN }
		\to
		\mf[q!]{ \algM } \tens[\dgcatA[2]] \mf[q!] { \algN }
	\).
	This implies that \( \mf[!] \)~descends to a
	functor~\(
		\mf[q!]
		\colon
		\catstcoalg{\dgcatA[1]}
		\to
		\catstcoalg{\dgcatA[2]}
	\).
	If \( \dgcatC[1] \in \catstcoalg{\dgcatA[1]} \)
	and~\( \dgcatC[2] \in \catstcoalg{\dgcatA[2]} \)
	are \cocatstar-algebras over different \cocatstar-categories,
	we can define a
	\textdef{morphism of \cocatstar-algebras}~\(
		\mf
		\colon
		\tup{ \dgcatC[1] , \dgcatA[1] }
		\to
		\tup{ \dgcatC[2] , \dgcatA[2] }
	\)
	to be a \cocatstar-functor~\(
		\mf
		\colon
		\dgcatA[1]
		\to
		\dgcatA[2]
	\)
	together with a map~\(
		\mf[q!]{ \dgcatC[1] }
		\to
		\dgcatC[2]
	\)
	in~\( \catstcoalg{\dgcatA[2]} \).
	By adjunction, when stating such map,
	one may as well state the map~\(
		\mf
		\colon
		\dgcatC[1]
		\to
		\mf[qpull]{\dgcatC[2]}
	\)
	in~\( \catxstmodx{\dgcatA[2]}{\dgcatA[2]} \),
	and we shall usually do so,
	even though this is only a map of \cocatstar-bimodules.
	One may similarly obtain a left
	adjoint~\(
		\map{\mf\times\mg}[spar,q!]
		\colon
		\catxstmodx{\dgcatA[1]}{\dgcatA[1]}
		\to
		\catxstmodx{\dgcatA[2]}{\dgcatA[2]}
	\)
	to the restriction functor~\(
		\map{\mf\times\mg}[spar,qpull]
		\colon
		\catxstmodx{\dgcatA[2]}{\dgcatA[2]}
		\to
		\catxstmodx{\dgcatA[1]}{\dgcatA[1]}
	\)
	when \( \mf , \mg \colon \dgcatA[1] \to \dgcatA[2] \)~are
	\cocatstar-functors.

	A (left) \textdef{\cocatstar-comodule}
	over a \cocatstar-coalgebra~\(
		\dgcatC
		\in
		\catstcoalg{\dgcatA}
	\)
	is a left comodule in the categorical sense
	over~\( \dgcatC \)
	in the category~\( \catxstmod{\dgcatA} \).
	This amounts to an object~\( \algM \in \catxstmod{\dgcatA} \)
	together with a
	coaction map~\(
		\ca
		\colon
		\algM
		\to
		\dgcatC
		\tens[\dgcatA]
		\algM
	\)
	in~\( \catxstmod{\dgcatA} \),
	satisfying the usual coassociativity
	and counity conditions.

	A counital \cocatstar-algebra is \textdef{coaugmented}
	if we are given a morphism of unital
	\( \dgcatA \)-\cocatstar-coalgebras~\(
		\coaug
		\colon
		\dgcatA
		\to
		\dgcatC
	\).
	It follows from the axioms
	of a unital \cocatstar-coalgebras
	that
	\(
		\dgcatA
		\xto{\coaug}
		\dgcatC
		\xto{\counit}
		\dgcatA
	\)~is the identity.
	Here \( \dgcatA \)~is regarded as an \( \dgcatA \)-\cocatstar-coalgebra, the comultiplication being the identity map.
	In that case, we can define the \textdef{reduction} of~\( \dgcatC \)
	as the non-counital \cocatstar-coalgebra~\( \smash{\dgcatC[red] = \Coker{\coaug}} \).
	On the level of \( \dgcatA \)-bimodules, we then have the splitting~\(
		\smash{\dgcatC = \dgcatC[red]\oplus\dgcatA}
	\).
	Conversely, if \( \dgcatC \)~is a non-counital \cocatstar-coalgebra,
	its \textdef{coaugmentation} is the counital, coaugmented
	\cocatstar-algebra~\( \dgcatC[augment] = \dgcatC\oplus\dgcatA \).
	We equip it with the comultiplication map~\(
		\comult[\dgcatC[augment]]
		\colon
		\dgcatC[augment]
		\to
		\dgcatC[augment]\tens[\dgcatA]\dgcatC[augment]
	\)
	given by~\(
		\comult[\dgcatC[augment]]
		=
		\coaug\tens\id[\dgcatC]
		+
		\comult[\dgcatC]
		+
		\comult[\dgcatA]
		+
		\id[\dgcatC]\tens\coaug
	\).
	A \textdef{morphism of coaugmented \cocatstar-coalgebras}
	is a morphism of counital \cocatstar-coalgebras that commutes with the coaugmentation.
	Thus we obtain a category~\( \catstcoalgnoc[coaug]{\dgcatA} \) of coaugmented
	\cocatstar-coalgebras over~\( \dgcatA \).
	If \(
		\dgcatC[1]
		\in
		\catstcoalgnoc[coaug]{\dgcatA[1]}
	\)
	and~\(
		\dgcatC[2]
		\in
		\catstcoalgnoc[coaug]{\dgcatA[2]}
	\),
	a morphism~\(
		\mf
		\colon
		\tup{\dgcatC[1],\dgcatA[1]}
		\to
		\tup{\dgcatC[2],\dgcatA[2]}
	\)
	of coaugmented \cocatstar-coalgebras
	over different \cocatstar-categories
	is a counital morphism making the square
	\[\begin{tikzcd}[sep=scriptsize]
		\mf[q!]{\dgcatC[1]} \ar[r] \ar[d]
			& \dgcatC[2] \ar[d]
	\\
		\mf[q!]{\dgcatA[1]} \ar[r]
			& \dgcatA[2]
	\end{tikzcd}\]
	commutative.
	Equivalently
	is a \cocatstar-functor~\(
		\mf \colon \dgcatA[1] \to \dgcatA[2]
	\)
	together with
	a morphism ~\(
		\mf
		\colon
		\mf[q!]{\dgcatC[1,red,smash]}[spar,augment]
		\to
		\dgcatC[2]
	\)
	in~\( \catstcoalgnoc[coaug]{\dgcatA[2]} \).
	A coaugmented \cocatstar-cocategory
	is \textdef{cocomplete} if its reduction
	is cocomplete in the above sense.
	Thus we obtain a category~\( \catstcoalg[coaug]{\dgcatA} \)
	of coaugmented, cocomplete \cocatstar-coalgebras
	over~\( \dgcatA \).
	
	A \textdef{\cocatstar-category}
	with set of objects~\( \setobjE \) is a \cocatstar-coalgebra
	over~\( \ringk[\setobjE] \).
	We denote by~\( \catstcocat = \catstcocat{\ringk} \)
	the category of \emph{cocomplete} \cocatstar-cocategories
	with arbitrary object sets
	and with morphisms the maps of coalgebras
	in the above sense.

\subsection{Tensor coalgebra}

	It is possible to (co)freely (co)generate a non-counital \cocatstar-coalgebra from
	any~\( \algV \in \catxstmodx{\dgcatA}{\dgcatA} \),
	the \textdef{non-counital tensor coalgebra}~\( \tensorcoalg[red]{\algV}\in\catstcoalg{\dgcatA} \),
	which is defined
	by the same formula as~\( \tensoralg[red]{\algV} \) above, but
	is equipped with the comultiplication
	map~\(
		\comult
		\colon
		\tensorcoalg[red]{ \algV }
		\to
		\tensorcoalg[red]{ \algV }
		\tens[\dgcatA]
		\tensorcoalg[red]{ \algV }
	\)
	given by
	\[\textstyle
		\comult{ \vx[n] \tens \cdots \tens \vx[1] }
		=
		\Sum[i]{
			(\vx[n]\tens\cdots\tens\vx[i+1])
			\tens
			(\vx[i]\tens\cdots\tens\vx[1])
		}
	\]
	for all~\( \vx[n],\ldots,\vx[1]\in\dgcatC \).
	A non-counital \cocatstar-coalgebra is called \textdef{cofree} if it is (isomorphic to a \cocatstar-coalgebra) of this form.
	Note that a cofree non-counital \cocatstar-coalgebra is cocomplete since
	a string of \( n+1 \)~morphisms will be annihilated by \( n \)~iterations of~\( \comult \).
	One may also define the 
	\textdef{coaugmented tensor coalgebra}
	as \( \tensorcoalg{\algV} = \tensorcoalg[red]{\algV}[augment] \).
	A counital \cocatstar-coalgebra
	is \textdef{cofree} if it is (isomorphic to a \cocatstar-coalgebra) of this form.
	
	
	If~\(
		\mf , \mg
		\colon
		\tup{\dgcatC[1],\dgcatA[1]}
		\to
		\tup{\dgcatC[2],\dgcatA[2]}
	\)
	are maps of \cocatstar-coalgebras,
	an~\textdef{\( \tup { \mf , \mg } \)-coderivation} of degree~\( i \) is 
	a map~\(
		\malpha
		\colon
		\map{\mf\times\mg}[spar,q!]{\dgcatC[1]}
		\to
		\dgcatC[2,shift=i]
	\)
	in the category~\(
		\catxstmodx{\dgcatA[2]}{\dgcatA[2]}
	\)
	such
	that~\(
		\comult[\dgcatD]\circ\malpha =
		( \mf\tens\malpha + \malpha\tens\mg ) \circ \map{\mf\times\mg}[spar,q!]{\comult[\dgcatC]}
	\).
	By adjunction, we may equivalently regard
	an \( \tup{\mf,\mg} \)-coderivation
	as a map~\(
		\malpha
		\colon
		\dgcatC[1]
		\to
		\map{\mf\times\mg}[spar,qpull]{\dgcatC[2]}[shift=i]
	\)
	in~\(
		\catxstmodx{\dgcatA[1]}{\dgcatA[1]}
	\),
	and we shall usually do so.
	We shall write~\( \coder[d=i]{ \mf , \mg } \) for the
	set of such.
	A \textdef{coderivation} is an \( \tup{ \id , \id } \)-coderivation.
	A dg-coalgebra may be considered as an algebra equipped with a coderivation~\( \dif \) satisfying~\( \dif[squared] = 0 \).
	
	\begin{propositionbreak}[Proposition {\parencite[Lemma~1.1.2.2]{lh}}]
		\begin{propositionlist}
			\item\label{res:cofun_cofree_morphisms}
			If \( \dgcatA \)~is a \cocatstar-category,
			\( \algV \in \catxstmodx{\dgcatA}{\dgcatA} \),
			and \( \dgcatC\in\catstcoalg[coaug]{\dgcatA} \)
			then postcomposition
			with the projection~\( \tensorcoalg{\algV}\onto \algV \)
			yields an isomorphism of sets
			\[
				\Hom[\catstcoalg[coaug]{\dgcatA}]{\dgcatC,\tensorcoalg{\algV}}
				\longisoto
				\Hom[\catxstmodx{\dgcatA}{\dgcatA}]{\dgcatC[red],\algV}.
			\]
			The inverse map takes a map of bimodules~\( \mf \colon \dgcatC[red] \to \algV \)
			to the map
			\[
				\Sum[n\ge0]{
					\mf[tens=n]
					\circ
					\comult[iterate=n]
				}
				\colon
				\dgcatC
				\to
				\tensorcoalg{\algV}
				.
			\]
			\item\label{res:cofun_cofree_precase}
			If \( \dgcatA[1] , \dgcatA[2] \) are \cocatstar-categories,
			\( \dgcatC \in \catstcoalg[coaug]{\dgcatA[1]} \),
			\( \algV \in \catxstmodx{\dgcatA[2]}{\dgcatA[2]} \),
			and we are given two maps of 
			coaugmented
			\cocatstar-coalgebras~\(
				\mf , \mg
				\colon
				\tup{ \dgcatC[1] , \dgcatA[1] }
				\to
				\tup{ \tensorcoalg{\algV} , \dgcatA[2] }
			\),
			then postcomposition
			with the
			projection~\(
				\tensorcoalg{\algV}
				\onto
				\algV
			\)
			yields an isomorphism of sets
			\[
				\coder[d=i]{ \mf , \mg }
				\cong
				\Hom[\catxstmodx{\dgcatA[1]}{\dgcatA[1]}]{
					\dgcatC ,
					{\map{\mf\times\mg}[spar,qpull]{ \algV }[shift=i]}
				}
			.\]
			The inverse map takes a morphism~\(
				\mh
				\colon
				\dgcatC
				\to
				\map{\mf\times\mg}[qpull,spar]{ \algV }[shift=i]
			\),
			regards it as a map~\(
				\mh
				\colon
				\map{\mf\times\mg}[spar,q!]{\dgcatC}
				\to
				\alg{\algV}[shift=i]
			\),
			and maps it
			to the
			\( \tup{\mf,\mg} \)-coderivation~\(
				\map{\mf\times\mg}[spar,q!]{\dgcatC}
				\to
				\tensorcoalg{\algV}[shift=i]
			\)
			whose \( n \)th~component, \( n\ge1 \), is
			\[
				\Sum[i={i+1+j=n}]{
					\mf[tens=i]\tens\mh\tens\mg[tens=j]
				}[spar=\Big]
				\circ
				\comult[der=n]
				.
			\]
		\end{propositionlist}
	\end{propositionbreak}

\endgroup

\begingroup

\section{\texorpdfstring{\ainfty}{A\_∞}-algebras and \texorpdfstring{\ainfty}{A\_∞}-categories}\label{para:intro_ainfty_cats}

	Let~\( \dgcatA \) be a dg-category.
	We define a (non-unital) \textdef{\ainfty-algebra}~\( \dgcatB \)
	over~\( \dgcatA \)
	to be a cocomplete, coaugmented dg-coalgebra~\( \tup{\dgcatC,\dif} \in \catdgcoalg[coaug]{\dgcatA} \)
	whose underlying graded coaugmented coalgebra is cofree in~\( \catgrcoalg[coaug]{\dgcatA} \)
	and with \( \dif \)~vanishing on~\( \dgcatA \subset \dgcatC \).
	We write it as~\(
		\dgcatC=\tensorcoalg{{\dgcatB[shift=1]}}
	\)
	for some~\( \dgcatB \in \catxgrmodx{\dgcatA}{\dgcatA} \).
	We shall usually focus on~\( \dgcatB \) instead
	of~\( \dgcatC \)
	and therefore refer to~\( \dgcatB \) as an \ainfty-algebra
	over~\( \dgcatA \),
	while \( \dgcatC \)~is called the \textdef{bar construction}
	of~\( \dgcatB \) and is written~\(
		\dgcatC
		=
		\Bar{}{\dgcatB}
		=
		\tup{\tensorcoalg{{\dgcatB[shift=1]}},\dif}
	\).
%
	\Cref{res:cofun_cofree_precase}
	shows that the coderivation~\( \dif \)
	is determined by the projection~\(
		\tensorcoalg{{\dgcatB[shift=1]}}
		\to
		\tensorcoalg{{\dgcatB[shift=1]}}[shift=1]
		\onto
		\dgcatB[shift=2]
	\),
	which is the same as maps~\(
	\smash{
		\mult[i,prime]
		\in
		\Hom[\catxgrmodx{\dgcatA}{\dgcatA},d=1]{
			{\dgcatB[shift=1,tens=i]} ,
			{\dgcatB[shift=1]}
		}
	}\).
	Via the bijection~\eqref{eq:coleibniz}, these~\( \mult[i,prime] \) are the same
	as maps~\(\smash{
		\mult[i]
		\in
		\Hom[\catxgrmodx{\dgcatA}{\dgcatA},d=2-i]{
			{\dgcatB[tens=i]},
			\dgcatB
		}
	}\).
	Then the equation~\( \dif[squared]=0 \)
	is equivalent to having for all~\( m\ge 1 \)
	the equation
	\begin{equation}\label{eq:ainfty_equations}
		\Sum {
			(-1)^{ij+k}
			\mult[l] {
				{\id[tens=i]}
				\tens
				\mult[j]
				\tens
				{\id[tens=k]}
			}
		} = 0
	\end{equation}
	where the sum runs over the integers~\( i , k\ge 0 \)
	and~\( l , j \ge 1 \)
	such that \( l = i + 1 + k \)
	and~\( m = i + j + k \).
	Clearly, a dg-category is an \ainfty-category with~\( \mult[1] \)
	the differential, \( \mult[2] \)~the composition map,
	and all \( \mult[i] = 0 \) for~\( i>2 \).
	A \textdef{unit} for the \ainfty-algebra~\( \dgcatB \)
	consists of a map~\( \unitmap\colon \dgcatA \to \dgcatB \)
	in~\( \catxgrmodx{\dgcatA}{\dgcatA} \)
	such that
	\(
		\mult[2]\map{\id\otimes\unitmap}[spar]
		= \id =
		\mult[2]\map{\unitmap\otimes\id}[spar]
	\)
	and \( \mult[m]\map{{\id[tens=i]}\tens\unitmap\tens{\id[tens=j]}}[spar] = 0 \) for  and all~\( m\neq2 \) and all~\( i , j \ge 0 \) with~\( i + 1 + j = m \).
	By an “\ainfty-algebra”, we shall in general mean a unital \ainfty-algebra.
	We are also going to consider the reduced bar construction~\(
		\Bar[red]{}{\dgcatA} = \Bar{}{\dgcatA}[red,smash]
	\).
	
	Most definitions related to \ainfty-algebras
	carry over from the definitions on dg-coalgebras.
	A \textdef{morphism of \ainfty-algebras}~\(
		\mf
		\colon
		\dgcatB[1]
		\to
		\dgcatB[2]
	\)
	over~\( \dgcatA \)
	is
	a morphism~\(
		\mf
		\colon
		\Bar{}{\dgcatB[1]}
		\to
		\Bar{}{\dgcatB[2]}
	\)
	in~\( \catdgcoalg[coaug]{\dgcatA} \).
	By~\cref{res:cofun_cofree_morphisms},
	this amounts to a map~\(\smash{
		\mf
		\in
		\Hom[\catxgrmodx{\dgcatA}{\dgcatA},d=0]{
			\Bar[red]{}{\dgcatB[1]} ,
			{ \dgcatB[2,shift=1] }
		}
	}\),
	which is the same as a 
	collection of maps~\(\smash{
		\mf[i,prime]
		\in
		\Hom[\catxgrmodx{\dgcatA}{\dgcatA},d=0]{
			{ \dgcatB[1,shift=1,tens=i] } ,
			{ \dgcatB[2,shift=1] }
		}
	}\) for all~\( i\ge 1 \).
	Via the bijection~\eqref{eq:coleibniz},
	this is the same as a collection of maps~
	\[
		\mf[i]
		\in
		\Hom[\catxgrmodx{\dgcatA}{\dgcatA},d=1-i]{
			{ \dgcatB[1,tens=i] } ,
			\dgcatB[2]
		}
	\]
	for all~\( i\ge 1 \).
	These must be subject to some technical conditions that
	we shall not write
	\parencite[see e.g.][Définition~1.2.1.2]{lh},
	which are equivalent to~\( \mf \) commuting
	with differentials.
	These form the maps in the
	category~\( \catainftyalg[nu]{\dgcatA} \)
	of non-unital \ainfty-algebras over~\( \dgcatA \).
	If \( \dgcatB[1] \) and~\( \dgcatB[2] \)
	are unital, the morphism~\( \mf \)
	is called \textdef{unital}
	if~\(
		\mf[1]\circ\unitmap
		=
		\unitmap\circ\mf[1]
		=
		\id
	\)
	and~\(
		\mf[n]{
		{\id[tens=i]}
		\tens
		\unitmap
		\tens
		{\id[tens=j]}
		}
		=
		0
	\)
	for all~\( n>1 \) and~\( i,j\ge0 \) with~\( i + 1 + j = n \).
	These form the morphisms of the category~\( \catainftyalg{\dgcatA} \)
	of (unital) \ainfty-algebras over~\( \dgcatA \).
	
	Suppose that \( \dgcatB[1] \in \catainftyalg{\dgcatA[1]} \)
	and~\( \dgcatB[2] \in \catainftyalg{\dgcatA[2]} \)
	are \ainfty-algebras over different dg-categories.
	A (unital) morphism~\(
		\mf
		\colon
		\tup{\dgcatB[1],\dgcatA[1]}
		\to
		\tup{\dgcatB[2],\dgcatA[2]}
	\)
	is a morphism of coaugmented dg-coalgebras~\(
		\mf
		\colon
		\Bar{}{\dgcatB[1]}
		\to
		\Bar{}{\dgcatB[2]}
	\)
	in the above sense,
	satisfying the unital condition.
	In other words,
	it is a dg-functor~\(
		\mf \colon \dgcatA[1] \to \dgcatA[2]
	\)
	together with a
	morphism~\(
		\mf
		\colon
		\mf[q!]{ \Bar[red]{}{\dgcatB[1]} }[spar,augment]
		\to
		\Bar{}{\dgcatB[2]}
	\)
	in~\( \catdgcoalg[coaug]{\dgcatA[2]} \),
	subject to the unital condition.
	\Cref{res:cofun_cofree_morphisms}
	shows that
	such map is the same as a map~\(\smash{
		\mf[q!]{ \Bar[red]{}{\dgcatB[1]} }
		\to
		\dgcatB[2,shift=1]
	}\)
	in~\( \catxgrmodx{\dgcatA[2]}{\dgcatA[2]} \),
	which by adjunction is the same as a map~\(
		\Bar[red]{}{\dgcatB[1]}
		\to
		\mf[qpull]{{ \dgcatB[2,shift=1] }}
	\)
	in~\( \catxgrmodx{\dgcatA[1]}{\dgcatA[1]} \).
	Another application of \cref{res:cofun_cofree_morphisms} then shows that this is the same as a
	(unital) morphism~\(
		\mf
		\colon
		\Bar{}{\dgcatB[1]}
		\to
		\Bar{}{ \mf[qpull]{{ \dgcatB[2] }} }
	\)
	of \ainfty-algebras
	in~\(
		\catainftyalg{\dgcatA[1]}
	\).
	Thus we may as well define
	a morphism~\(
		\mf
		\colon
		\tup{\dgcatB[1],\dgcatA[1]}
		\to
		\tup{\dgcatB[2],\dgcatA[2]}
	\)
	of \ainfty-algebras
	to be a dg-functor~\( \mf \colon \dgcatA[1] \to \dgcatA[2] \)
	together with a morphism~\(
		\mf
		\colon
		\dgcatB[1]
		\to
		\mf[qpull]{ \dgcatB[2] }
	\)
	in~\( \catainftyalg{\dgcatA[1]} \);
	we notice that it makes sense to regard~\( \mf[qpull]{\dgcatB[2]} \)
	as an \ainfty-algebra
	this way because \( \mf[qpull] \)~is lax monoidal.
	(Remarkably, this shows that cofree dg-cocategories, unlike general dg-cocategories, allow a pullback operation.)

	A unital \ainfty-algebra is \textdef{augmented}
	if it is equipped with a unital
	morphism
	of unital \ainfty-algebras~\(
		\aug
		\colon
		\dgcatB
		\to
		\dgcatA
	\)
	in~\( \catainftyalg{\dgcatA} \).
	It is a consequence of the definition of unital morphisms
	that the composition
	\(
		\dgcatA
		\xto{\unitmap}
		\dgcatB
		\xto{\aug}
		\dgcatA
	\)~is the identity.
	Similarly to the dg-case,
	we may define the \textdef{reduction} of~\( \dgcatA \)
	by~\( \dgcatA[red] = \Ker{\aug} \).
	Conversely, a non-unital \ainfty-algebra~\( \dgcatB \)
	can be made into a unital, augmented \ainfty-algebra
	by~\( \dgcatB[augment] = \dgcatB\oplus\dgcatA \).
	The \ainfty-operations on this will be
	given by~\(
		\mult[1,\dgcatB[augment]]
		=
		\mult[1,\dgcatB]
		+
		\dif[\dgcatA]
	\),
	\(
		\mult[2,\dgcatB[augment]]{b+a,b'+a'}
		=
		\mult[2,\dgcatB]{b,b'}
		+
		ab'
		+
		a'b
		+
		a a'
	\),
	and \( \mult[i,\dgcatB[augment]] = \mult[i,\dgcatB] \)
	for~\( i>2 \).	
	If \( \dgcatB \)~is augmented, one may consider
	the \textdef{augmented bar construction}~\(\smash{
		\Bar[plus] {}{\dgcatB}
		=
		\Bar{}{\dgcatB[red]}
	}\).

	An \textdef{\ainfty-category}
	with object set~\( \setobjE \)
	is a defined to be a (unital) \ainfty-algebra over the dg-category~\( \ringk[\setobjE] \).
	Denote by~\( \catainftycat = \catainftycat{\ringk} \)
	the category of (unital) \ainfty-categories
	with arbitrary object sets and
	morphisms
	the unital
	maps of \ainfty-algebras in the above sense,
	known as
	\textdef{\ainfty-functors}.
	The collection of non-unital \ainfty-categories
	also form a category~\(
		\catainftycat[nu] = \catainftycat[nu]{\ringk}
	\).

\section{\texorpdfstring{\ainfty}{A\_∞}-coalgebras and \texorpdfstring{\ainfty}{A\_∞}-cocategories}

	Let again~\( \dgcatA \) be a dg-category.
	An \textdef{\ainfty-coalgebra} over~\( \dgcatA \)
	is an augmented dg-algebra~\( \tup{\dgcatB,\dif} \in \catdgalg[aug]{\dgcatA} \)
	whose underlying graded augmented algebra is free in~\( \catgralg[aug]{\dgcatA} \)
	and for which \( \dgcatB \xto{\dif} \dgcatB \onto \dgcatA \)
	is zero.
	We therefore write it as~\(
		\dgcatB = \tensoralg{{ \dgcatC[shift=-1] }}
	\)
	for some~\( \dgcatC\in\catxgrmodx{\dgcatA}{\dgcatA} \).
	We shall usually focus on~\( \dgcatC \)
	instead of~\( \dgcatB \) and therefore
	refer to~\( \dgcatC \) as an
	\ainfty-coalgebra over~\( \dgcatA \),
	while \( \dgcatB \)~is called the \textdef{cobar construction}
	of~\( \dgcatC \) and is written~\(
		\dgcatB
		= \Cob{}{\dgcatC}
		= \tup{ \tensoralg{{\dgcatC[shift=-1]}} , \dif }
	\).
	Therefore,
	\cref{res:derivations_free}
	shows that the derivation~\( \dif \)
	on the free algebra~\( \tensoralg{{ \dgcatC[shift=-1] }} \)
	is determined by the inclusion
	\[
		\dgcatC[shift=-1]
		\into
		\tensorcoalg{{ \dgcatC[shift=-1] }}
		\to
		\tensorcoalg{{ \dgcatC[shift=-1] }}[shift=1]
	,
	\]
	which is the same as maps~\(\smash{
		\comult[i,prime]
		\in
		\Hom[\catxgrmodx{\dgcatA}{\dgcatA},d=1]{
			{ \dgcatC[shift=-1] } ,
			{ \dgcatC[shift=-1,tens=i] }
		}
	}\).
	Via the bijection~\eqref{eq:leibniz},
	these~\( \comult[i,prime] \) are the same as maps~\(\smash{
		\comult[i]
		\in
		\Hom[\catxgrmodx{\dgcatA}{\dgcatA},d=2-i]{
			\dgcatC ,
			{ \dgcatC[tens=i] }
		}
	}\).
	Then the equation~\( \dif[squared] = 0 \) is equivalent to having for
	all~\( m\ge 1 \) the equation
%
	\begin{equation}\label{eq:ainfty_cocategory_equation}
		\Sum {
			(-1)^{i+jk}
			(
				\id[tens=i]\tens\comult[j]\tens\id[tens=k]
			)
			\comult[l]
		}
		=0
	\end{equation}
	for all~\( i,j,k,l \) such that \( i + j + k = m \)
	and~\( l = i + 1 + k \).
	Furthermore, for the summation in the cobar construction to be meaningful, we require that the product
	map~\( \Prod[i\ge 1] { \comult[i] }\colon \dgcatC\to \Prod[i\ge 1] {{\dgcatC[tens=i]}} \) must factor
	through the direct sum~\( \dirsum[i\ge 1] {{\dgcatC[tens=i]}} \).
	The \ainfty-algebra~\( \dgcatC \) is \textdef{counital} if
	it is equipped with a
	map~\( \counit \colon \dgcatC\to\dgcatA \)
	in the category~\( \catxgrmodx{\dgcatA}{\dgcatA} \)
	satisfying~\(
		\map { \id \tens \counit }[spar]
		\comult[2]
		=
		\id
		=
		\map { \counit \tens \id }[spar]
		\comult[2]
	\)
	and~\( \map { \id[tens=i]\tens\counit\tens\id[tens=j] }[spar] \comult[m]=0\)
	for all~\( m\neq 2 \) and all~\( i, j\ge 0 \) with~\( i + 1 + j = m \).
	By an “\ainfty-coalgebra”,
	we shall usually mean a counital one.
	We are also going to consider the non-augmented
	algebra~\( \Cob[red]{}{\dgcatC} = \Cob{}{\dgcatC}[red] \).
	
	Most definitions related to \ainfty-coalgebras
	carry over from the definitions on dg-algebras.
	A \textdef{morphism of \ainfty-coalgebras}~\(
		\mf \colon \dgcatC[1] \to \dgcatC[2]
	\)
	over~\( \dgcatA \)
	is a morphism~\(
		\mf
		\colon
		\Cob{}{\dgcatC[1]}
		\to
		\Cob{}{\dgcatC[2]}
	\)
	in~\( \catdgalg[aug]{\dgcatA} \).
	By~\cref{res:free_morphisms},
	this amounts to a map~\(\smash{
		\mf
		\in
		\Hom[d=0,\catxgrmodx{\dgcatA}{\dgcatA}]{
			{ \dgcatC[1,shift=-1] } ,
			\Cob[red]{}{\dgcatC[2]}
		}
	}\),
	which is the same as a collection
	of maps~\(\smash{
		\mf[i,prime]
		\in
		\Hom[d=0,\catxgrmodx{\dgcatA}{\dgcatA}]{
			{ \dgcatC[1,shift=-1] } ,
			{ \dgcatC[2,shift=-1,tens=i] }
		}
	}\).
	Via the bijection~\eqref{eq:leibniz},
	this is the same as a collection
	of maps
	\[
		\mf[i]
		\in
		\Hom[d=1-i,\catxgrmodx{\dgcatA}{\dgcatA}]{
			\dgcatC[1] ,
			{ \dgcatC[2,tens=i] }
		}
		.
	\]
	These must be subject to some technical conditions
	that we shall not write,
	which are equivalent to~\( \mf \)
	commuting with the differential.
	These form the morphisms in the
	category~\( \catainftycoalg[nu]{\dgcatA} \)
	of non-counital \ainfty-coalgebras
	over the dg-category~\( \dgcatA \).
	If \( \dgcatB[1] \) and~\( \dgcatB[2] \)
	are counital,
	a morphism~\( \mf \)
	is called \textdef{counital}
	if~\(
		\mf[1]
		\circ
		\counit
		=
		\counit
		\circ
		\mf[1]
		=
		\id
	\)
	and~\(
		\mf[n]{{
			\id[tens=i]
			\tens
			\counit
			\tens
			\id[tens=j]
		}}
		=
		0
	\)
	for all~\( n>1 \)
	and~\( i,j\ge0 \)
	with~\( i + 1 + j = n \).
	These form the morphisms in the
	category~\( \catainftycoalg{\dgcatA} \)
	of (counital) \ainfty-coalgebras over~\( \dgcatA \).

	Suppose that \( \dgcatC[1] \in \catainftycoalg{\dgcatA[1]} \)
	and~\( \dgcatC[2] \in \catainftycoalg{\dgcatA[2]} \)
	are \ainfty-coalgebras over different dg-categories.
	A (unital) morphism~\(
		\mf
		\colon
		\tup{\dgcatC[1],\dgcatA[1]}
		\to
		\tup{\dgcatC[2],\dgcatA[2]}
	\)
	is a morphism of augmented dg-algebras~\(
		\mf
		\colon
		\Cob{}{\dgcatC[1]}
		\to
		\Cob{}{\dgcatC[2]}
	\)
	in the above sense,
	satisfying the counital condition.
	In other words, it consists
	of a dg-functor~\( \mf \colon \dgcatA[1]\to\dgcatA[2] \)
	together with a morphism~\(
		\mf
		\colon
		\Cob{}{\dgcatC[1]}
		\to
		\mf[qpull]{ \Cob[red]{}{\dgcatC[2]} }[spar,augment]
	\)
	in~\( \catdgalg[aug]{\dgcatA[1]} \),
	subject to the counital condition.
	By~\cref{res:free_morphisms},
	such map is the same as a map~\(
		\dgcatC[1,shift=-1]
		\to
		\mf[qpull]{ \Cob[red]{}{\dgcatC[2]} }[smash]
	\)
	in~\( \catxgrmodx{\dgcatA[1]}{\dgcatA[1]} \),
	which by adjunction is the same as a map~\(
		\mf[q!]{{ \dgcatC[1,shift=-1] }}
		\to
		\Cob[red]{}{\dgcatC[2]}
	\).
	Another application of \cref{res:free_morphisms} shows that
	this is the same as
	a morphism~\(
		\mf
		\colon
		\Cob{}{ \mf[q!]{ \dgcatC[1] } }
		\to
		\Cob{}{ \dgcatC[2] }
	\)
	of \ainfty-coalgebras
	in~\( \catainftycoalg{ \dgcatA[2] } \).
	In other words, we may as well define
	a morphism~\(
		\mf
		\colon
		\tup{\dgcatC[1] , \dgcatA[1]}
		\to
		\tup{\dgcatC[2] , \dgcatA[2]}
	\)
	to be a dg-functor~\(
		\mf
		\colon
		\dgcatA[1]
		\to
		\dgcatA[2]
	\)
	together with a morphism~\(
		\mf
		\colon
		\mf[q!]{ \dgcatC[1] }
		\to
		\dgcatC[2]
	\)
	in~\( \catainftycoalg{\dgcatA[2]} \);
	we notice that it makes sense to
	regard~\( \mf[q!]{ \dgcatC[1] } \)
	as an \ainfty-coalgebra
	because \( \mf[q!] \)~is oplax monoidal.
	(Remarkably, this shows that dg-algebras which are free as graded algebras,
	unlike general dg-algebras,
	allow a \( ! \)-push\-forward operation.)

	A counital \ainfty-coalgebra~\( \dgcatC \)
	is \textdef{coaugmented}
	if it is equipped with a counital
	morphism of counital \ainfty-coalgebras~\(
		\coaug
		\colon
		\dgcatA
		\to
		\dgcatC
	\)
	in~\( \catainftycoalg{\dgcatA} \).
	It is a consequence of the definition of
	counital morphisms that the composition~\(
		\dgcatA
		\xto{\coaug}
		\dgcatC
		\xto{\counit}
		\dgcatA
	\)
	is the identity.
	Similarly to the dg-case,
	we may define the \textdef{reduction}
	of~\( \dgcatC \)
	by~\(
		\dgcatC[red,smash] = \Coker{\coaug}
	\).
	Conversely, a non-counital \ainfty-coalgebra~\( \dgcatC \)
	can be made into a counital,
	coaugmented \ainfty-coalgebra
	by~\(
		\dgcatC[augment]
		=
		\dgcatC
		\oplus
		\dgcatA
	\).
	The \ainfty-coalgebra structure on this is dual
	to the one we defined on augmented \ainfty-algebras.
	If \( \dgcatC \)~is in augmented,
	we may define the \textdef{augmented cobar construction}
	by~\(
		\Cob[plus]{}{ \dgcatC }
		=
		\Cob{}{ \dgcatC[red,smash] }
	\).
	
	An \textdef{\ainfty-cocategory}
	with set of objects~\( \setobjE \)
	is an \ainfty-coalgebra
	over the dg-category~\( \ringk[\setobjE] \).
	Denote by~\( \catainftycocat = \catainftycocat{\ringk} \)
	the category of \ainfty-cocategories
	with arbitrary object sets and
	morphisms the counital maps of \ainfty-coalgebras
	in the above sense.
	The collection of non-counital \ainfty-cocategories
	also form a category~\(
		\catainftycocat[nu] = \catainftycocat[nu]{\ringk}
	\).

	Restricting \( \Bar[plus] {} \) and~\( \Cob[plus] {} \)
	to honest, (co)augmented dg-(co)categories, we obtain an adjunction
	\[
		\begin{tikzcd}[sep=scriptsize]
			\Cob[plus]{}
			\colon
			\catdgcocat[coaug]
			\ar[r,yshift=1.5pt]
			&
			\ar[l,yshift=-1.5pt]
			\catdgcat[aug]
			\noloc
			\Bar[plus]{}
			\invcomma
		\end{tikzcd}
	\]
	see \textcite[Lemma~1.2.2.5]{lh}.
	
	\begin{example}
		Denote by~\( \ordset{n} \) the poset consisting of \( n + 1 \)~objects~\( 0 \), \( 1 \), \ldots, \( n \) and a unique morphisms~\( \tup{i,j}\colon i\to j \) for all~\( i\le j \).
		Let~\( \ordset{n}[linearize=\ringk] \) be its \( \ringk \)-linearization,
		and regard it as an augmented dg-category with zero differential.
		The augmented bar construction
		\[
			\Bar[plus]{}{{\ordset{n}[linearize=\ringk]}}
			=
			\Bar[red]{}{{\ordset{n}[linearize=\ringk]}}
			\oplus
			\ringk[\ordset{n}]
		\]
		has set of objects~\( \ordset{n} \)
		and morphisms
		given by tensor products
		\(
			\tup{i_0,i_1,...,i_k}
			=
			\tup{i_{k-1},i_k}
			\tens
			\tup{i_{k-2},i_{k-1}}
			\tens
			\cdots
			\tens
			\tup{i_0,i_1}
		\).
		The differential is given by
		\[
			\dif{\tup{i_0,i_1,...,i_k}}
			=
			\Sum[from={j=1},to={k-1}]{
				(-1)^{k-j+1}
				\tup{i_0,...,\hat{\imath}_j,...,i_k}
			}
		\]
		
		Applying the cobar construction,
		we obtain the free dg-category
		\[
			\Cob[plus]{}{ \Bar[plus]{}{{
					\ordset{n}[linearize=\ringk]
				}}
			}
			=
			\Cob[red]{}{
				\Bar[red]{}{{
					\ordset{n}[linearize=\ringk,red]
				}}
			}
			\oplus
			\ringk[\ordset{n}]
		.\]
		This again has object
		set~\( \ordset{n} 
		\)
		and with morphisms freely generated
		by morphisms~\(
			\mf[i_0 i_1 \cdots i_k]
			\colon
			i_0
			\to
			i_k
		\)
		of degree~\( 1-k \)
		for all sequences~\( 0 \le i_0 < i_1 < \cdots < i_k \le n \),~\( k>0 \).
		The differential is given by
		\[
			\dif{ \mf[i_0 i_1 \cdots i_k] }
			=
			\Sum[from={j=1},to={k-1}]{
			(-1)^{k-j}
			\bigl(
				\mf[i_0\cdots \hat{\imath}_j \cdots i_k]
				-
				\mf[i_j \cdots i_k]
				\circ
				\mf[i_0\cdots i_j]
			\bigr)
			}
			.
		\]
		In this formula, we use the convention that
		\( \mf[j_0 j_1 \cdots j_l] = 0 \)
		if the index contains repetitions
		and~\( l>1 \),
		and that~\( \mf[j_0 j_0] = \id \).
%
	\end{example}

\subsection{\texorpdfstring{\ainfty}{A\_∞}-comodules}

	Let~\( \dgcatC \) be an \ainfty-coalgebra
	over a dg-category~\( \dgcatA \).
	A \textdef{left \ainfty-comodule} over~\( \dgcatC \)
	is a dg-module
	in~\( \catxdgmod{\dgcatA} \)
	over the dg-algebra~\(
		\Cob{}{\dgcatC}
	\) whose underlying graded module is \emph{free}.
	By free, we mean
	that it has the form~\(
		\Cob{}{\dgcatC}
		\tens[\dgcatA]
		\algM[shift=-1]
	\)
	for some~\( \algM \in \catxgrmod{\dgcatA} \).
	As with coalgebras, we shall usually
	focus on~\( \algM \) and refer to it as an
	\ainfty-comodule,
	while
	the dg-module~\(
		\Cob{}{\dgcatC,\algM}
		:=
		\Cob{}{\dgcatC}
		\tens[\dgcatA]
		\algM[shift=-1]
	\)
	is called the \textdef{cobar construction}
	of~\( \algM \).
	The action map
	\[
		\Cob{}{\dgcatC}
		\tens[\dgcatA]
		\Cob{}{\dgcatC,\algM}
		\to
		\Cob{}{\dgcatC,\algM}
	\]
	amounts, by freeness and~\eqref{eq:leibniz},
	to maps~\(
		\ca[i]
		\colon
		\algM
		\to
		\algC[tens=(1-i)]
		\tens[\dgcatA]
		\algM[shift=2-i]
	\)
	for all~\( i\ge 1\),
	subject to the equation~\eqref{eq:ainfty_cocategory_equation}
	where \( \comult[l] \)~is understood as~\( \ca[l] \),
	while \( \comult[j] \)~must
	be interpreted as~\( \ca[j] \) in the case~\( k=0 \).
	Furthermore, in order for these maps to be definable
	on the cobar construction,
	we must require
	that the product~\(
		\Prod[i\ge1]{\ca[i]}
		\colon
		\algM
		\to
		\Prod[i\ge1]{
			\dgcatC[tens=(i-1)]
			\tens[\dgcatA]
			\algM
		}
	\)
	factors through~\(
		\dirsum[i\ge1]{
			\dgcatC[tens=(i-1)]
			\tens[\dgcatA]
			\algM
		}
	\).
	We call~\( \algM \) a \textdef{formal \ainfty-comodule}
	if we omit this last condition.
	

	The collection of \ainfty-comodules over~\( \dgcatC \)
	form a dg-category,
	namely the full dg-subcategory
	of~\(
		\catxdgmod{\Cob{}{\dgcatC}}
	\)
	consisting of dg-modules that are free as graded modules.
	A map \( \mf \colon \algM \to \algN \)
	of degree~\( \dgcatdeg{\mf} \) is the same as a
	map of \( \Cob{}{\dgcatC}[smash] \)-dg-modules
	\(
		\mf
		\colon
		\Cob{}{\dgcatC,\algM}[smash]
		\to
		\Cob{}{\dgcatC,\algN}[smash,shift=\dgcatdeg{\mf}]
	\).
	This is equivalent via~\eqref{eq:leibniz} to a
	collection of maps
	\[
		\mf[i]
		\colon
		\algM
		\longto
		\dgcatC[tens={(i-1)}]
		\tens[\dgcatA]
		\algN[shift=\dgcatdeg{\mf}+1-i]
	\]
	for~\( i\ge 1 \),
	such that the product~\(
		\Prod[i\ge1] { \mf[i] }
		\colon
		\algM
		\to
		\Prod[i\ge1]{
			\dgcatC[tens={(i-1)}]
			\tens[\dgcatA]
			\algN[shift=\dgcatdeg{\mf}+1-i]
		}
	\)
	factors through~\(
		\dirsum[i\ge1]{
			\dgcatC[tens={(i-1)}]
			\tens[\dgcatA]
			\algN[shift=\dgcatdeg{\mf}+1-i]
		}
	\).
	The differential
	\[
		\dif[par]{\mf}
		= \dif[\Cob{}{\dgcatC,\algN}] \circ \mf
		-
		(-1)^{\dgcatdeg{\mf}}
		\mf \circ \dif[\Cob{}{\dgcatC,\algM}]
	\]
	is given by
	\[
		\dif[par]{\mf}[n]
		=
		\Sum{
			(-1)^{i+jk}
			\map{
				\id[tens=i]
				\tens
				\comult[j]
				\tens
				\id[tens=k]
			}[spar]
			\mf[m]
		}
		-
		\Sum{
			(-1)^{p\dgcatdeg{\mf}}
			\map{
				\id[tens=(p-1)]
				\tens
				\mf[q]
			}[spar]
			\ca[p]
		}
	\]
	where the first sum runs over \( i,k\ge 0\) and \( j,m>1 \)
	with \( i + 1 + k = m \) and~\( i + j + k = n \)
	and the second over the~\( p , q \ge 1 \) with~\( p + q - 1 = n \).
	If \( i = 0 \), \( \comult[j] \)~should be understood as~\( \ca[j] \).
	If \( \mf \colon \algM \to \algN \) and \( \mg \colon \algN \to \algP \)
	are maps of \ainfty-comodules, their composition~\( \mg \circ \mf \)
	is the composition
	\(\smash{
		\Cob{}{\dgcatC,\algM}
		\xto{\mf}
		\Cob{}{\dgcatC,\algN}
		\xto{\mg}
		\Cob{}{\dgcatC,\algP}
	}\).
	One checks via~\eqref{eq:leibniz} that the~\( n \)th component is
	\begin{align*}
		\map{\mg\circ\mf}[spar,n]
		&=
		\Sum[lower={n=l+k-1}]{
			(-1)^{\dgcatdeg{\mg}(l-1)}
			\map{\id[tens=(l-1)]\tens\mg[k]}[spar]
			\circ
			\mf[l]
		}
		.
	\end{align*}
	If \( \dgcatC \)~is counital,
	the \ainfty-comodule~\( \algM \) is called \textdef{counital}
	if \( \ca[2](\counit\tens\id[\algM]) = \id[\algM] \)
	and~\(
		\ca[n]
		(
			\id[tens=i]
			\tens
			\counit
			\tens
			\id[tens=j]
			\tens
			\id[\algM]
		) = 0
	\)
	for all~\( n>2 \)
	and all~\( i,j \ge0 \) with~\( i + j + 2 = n \).
	It is called \textdef{homotopy-counital}
	if we have a homotopy~\( \ca[2](\counit\tens\id[\algM])\simeq\id[\algM]\)
	with respect to the differential~\( \ca[1] \)
	on~\( \algM \).
	We denote the dg-category of such
	by~\( \catxcomod{\dgcatC}[infty,hcu]{\dgcatA} \).

\section{The \texorpdfstring{\ainfty}{A\_∞}-category of \texorpdfstring{\ainfty}{A\_∞}-functors}

There exists an internal hom in the category of \ainfty-categories,
known as the \ainfty-category of \ainfty-functors.
It is in fact just a special case of an internal hom in the category of dg-cocategories.
In this section, we stick mostly to the approach
of \textcite{keller}.
We define the tensor product~\(
	\dgcatC\tens[\ringk]\dgcatD
\) of two \cocatstar-categories
over~\( \ringk \) as the tensor product of the underlying \cocatstar-quivers,
equipped with the natural diagonal \cocatstar-category structure.

	
	\begin{propositionbreak}[{Proposition
	\parencite[Theorem~5.3, Lemma~5.2, Lemma~5.4]{keller}}]\label{res:cofun}
		\begin{propositionlist}
			\item The category~\( \catstcocat[coaug] \) of cocomplete,
			coaugmented \cocatstar-cocategories
			is a closed monoidal category
			with an internal hom which we denote by~\( \catstcofun \).
			The \cocatstar-cocategory~\( \catcofun[star]{ \dgcatC , \dgcatD } \)
			has objects the morphisms of
			coaugmented \cocatstar-cocategories~\( \dgcatC\to\dgcatD \).
			It is a \cocatstar-subquiver
			of~\(\cathom[\catstquiv[aug]]{ \dgcatC , \dgcatD }\),
			the internal hom in augmented \cocatstar-quivers.
			\item\label{res:dgcofun_subset_grcofun}
			If \( \dgcatC \) and~\( \dgcatD \)~are dg-cocategories,
			 \( \catdgcofun{\dgcatC,\dgcatD} \)~is
			the full subcocategory of~\( \catgrcofun{\dgcatC,\dgcatD} \)
			consisting of morphisms of graded coaugmented cocategories commuting with the differential.
			\item\label{res:cofun_cofree} If \( \dgcatD = \tensorcoalg{ \algV } \)~is cofree,
			then
			\[
				\catcofun[star]{ \dgcatC , \tensorcoalg{ \algV } }
				\cong
				\tensorcoalg[par=\big]{{ \cathom[\catstquiv[aug]]{ \dgcatC , \algV }[red] }}
			\]
			is also cofree.
			The isomorphisms on objects is given
			as in~\cref{res:cofun_cofree_morphisms}.
			On morphisms,
			suppose that
			\[\begin{tikzcd}[sep=scriptsize]
				\mf[n]
				\ar[r,"{\malpha[n]}",<-]
				&
				\mf[n-1]
				\ar[r,"{\malpha[n-1]}",<-]
				&
				\cdots
				\ar[r,"{\malpha[1]}",<-]
				&
				\mf[0]
			\end{tikzcd}\]
			is a composable string
			of natural transformations
			in~\( \cathom[\catquiv[star]]{\dgcatC,\algV} \),
			that is,
			\(
				\malpha[i] \in
				\cathom[\catquiv[star]]{\dgcatC,\algV}{ \mf[i-1] , \mf[i] }
				\cong
				\coder{\mf[i-1],\mf[i]}
			\).
%
			Then the leftwards arrow takes the
			tensor~\( \malpha[n]\tens\cdots\tens\malpha[1] \)
			to the natural
			transformation~\(
				\naturaltensor{\malpha[n],...,\malpha[1]}
			\)
			whose \( N \)th
			component~\(
				\naturaltensor{\malpha[n],...,\malpha[1]}[N]
				\colon
				\dgcatC
				\to
				\map{\mf[0]\times\mf[n]}[spar,qpull]{{\algV[tens=N]}}
			\)
			is
			given by the sum of terms
			\[
				\map {
					\mf[n,tens=r_n]
					\tens
					\malpha[n]
					\tens
					\mf[n-1,tens=r_{n-1}]
					\tens
					\cdots
					\tens
					\malpha[1]
					\tens
					\mf[0,tens=r_0]
				}[spar=\big]\circ\comult[der=N]
			\]
			with \( N = n + \Sum { r_i } \).
			Here we regard each~\( \mf[i] \)
			as a natural transformation~\( \mf[i]\to \mf[i] \)
			of degree~\( 1 \)
			via~\(
				\dgcatC
				\onto
				\dgcatC[red,smash]
				\to
				\map{\mf[i]\times\mf[i]}[spar,qpull]{ \algV }
			\).
		\end{propositionlist}
	\end{propositionbreak}

	Suppose now that \( \dgcatA \) and~\( \dgcatB \)
	are non-unital \ainfty-categories, and let us
	show the existence of a non-unital \ainfty-category of
	non-unital \ainfty-functors~\( \dgcatA \to \dgcatB \).
	Since the underlying graded cocategory of~\( \Bar{}{\dgcatB} \) is cofree, we obtain
	from~\cref{res:cofun_cofree} that the graded
	cocategory
	\begin{align*}
		\catgrcofun{ \Bar{}{\dgcatA} , \Bar{}{\dgcatB} }[smash]
		=
		\tensorcoalg{ \cathom[\catgrquiv]{ \Bar{}{\dgcatA} , {\dgcatB[shift=1]} }}[smash]
	\end{align*}
	is also cofree.
	Also, \cref{res:dgcofun_subset_grcofun} shows that,
	as a graded cocategory,
	\(
		\catdgcofun{\Bar{}{\dgcatA},\Bar{}{\dgcatB}}
		\subset
		\catgrcofun{\Bar{}{\dgcatA},\Bar{}{\dgcatB}}
	\)~is the full subcocategory
	consisting of morphisms of graded cocategories commuting with the differential.
	Thus the underlying graded cocategory of~\(
		\catdgcofun{\Bar{}{\dgcatA},\Bar{}{\dgcatB}}
	\)
	is also cofree, being a full subcocategory of a cofree graded cocategory.
	In other words, it is an \ainfty-category.
	Therefore, we may write it as~\( \Bar{}{ \catainftyfun[nonunital]{ \dgcatA , \dgcatB } } \) and call~\( \catainftyfun[nonunital]{ \dgcatA , \dgcatB } \)
	the (non-unital) \textdef{\ainfty-category of non-unital \ainfty-func\-tors}.
	In other words, the underlying graded quiver of~\(
		\catainftyfun[nu]{\dgcatA,\dgcatB}
	\)
	is a full graded subquiver of the graded quiver~\(
		\cathom[\catgrquiv]{
			\Bar{}{\dgcatA} ,
			{\dgcatB[shift=1]}
		}[shift=-1]
	\),
	and hence we obtain that
	the morphism space between \ainfty-functors~\( \mf , \mg \colon \dgcatA \to \dgcatB \)
	is
	\begin{align*}
		\catainftyfun[nu]{ \dgcatA , \dgcatB }{ \mf , \mg }
		&=
		\Hom[\catgrquiv{\catob{\dgcatA}}]{ \Bar{}{\dgcatA} ,
		{\map{\mf\times\mg}[spar,qpull] {{\dgcatB[shift=1]}}} }[shift=-1]
	\\
		&=
		\Hom[\catquiv[gr]{\catob{\dgcatA}}]{
			\Bar{}{\dgcatA}
			,
			\map{\mf\times\mg}[spar,qpull] { \dgcatB }
		}
		.
	\end{align*}
	We shall refer to this as the space of \textdef{non-unital \ainfty-transformations}~\( \mf \to \mg \).
	In \( \catdgcofun { \Bar{}{\dgcatA} , \Bar{}{\dgcatB } } \),
	the morphism space will instead be composable tensors~\(
		\naturaltensor{
			\ms{\malpha[n]},
			...,
			\ms{\malpha[1]}
		}
	\),
	where~\(
		\malpha[i]
		\in
		\catainftyfun[nu]{\dgcatA,\dgcatB}{\mf[i-1],\mf[i]}
	\).
	
	\begin{example}\label{ex:functors_as_transformations}
		Since a non-unital \ainfty-functor~\( \mf \colon \dgcatA \to \dgcatB \)
		is a map of graded quivers~\( \Bar[red]{}{\dgcatA} \to \dgcatB[shift=1] \),
		we may regard it as a
		non-unital \ainfty-trans\-forma\-tion~\( \mf \to \mf \)
		of degree~\( 1 \)
		via the composition
		\(
			\Bar{}{\dgcatA}
			\onto
			\Bar[red]{}{\dgcatA}
			\to
			\dgcatB[shift=1]
		\).
		Note that it is \emph{not} an identity morphism on~\( \mf \),
		as it does not even have degree zero.
	\end{example}
	
	We can describe the differential
	\[
		\dif
		\colon
		\Bar{}{ \catainftyfun[nonunital]{ \dgcatA , \dgcatB } }
		\to
		\Bar{}{ \catainftyfun[nonunital]{ \dgcatA , \dgcatB } }[shift=1]
	\]
	more explicitly.
	A morphism on the \anlhs consists of a sum of tensors
	\[
		\ms[tens=n]{
			\malpha[n]
			\sotimes
			\cdots
			\sotimes
			\malpha[1]
		}
		\in
		\alg{
			\catainftyfun[nu]{\dgcatA,\dgcatB}{\mf[n-1],\mf[n]}
			\tens
			\cdots
			\tens
			\catainftyfun[nu]{\dgcatA,\dgcatB}{\mf[0],\mf[1]}
		}[spar,shift=n]
		.
	\]
	Then
	\begin{align*}
		\dif[par] {{ \ms[tens=n]{ \malpha[n]\tens\cdots\tens\malpha[1] } }}
		&=
		\dif[\Bar{}{\dgcatB}]
		\circ
		\ms[tens=n]{ \malpha[n]\tens\cdots\tens\malpha[1] } 
	\\
		& \qquad\qquad
		{}-
		(-1)^{\dgcatdeg { \malpha[n] } + \cdots + \dgcatdeg { \malpha[1] }-n}
		{ \ms[tens=n]{ \malpha[n]\tens\cdots\tens\malpha[1] } }
		\circ
		\dif[\Bar{}{\dgcatA}]
	.\end{align*}
	Applying the projection~\( \mpr[1]\colon\Bar{}{\dgcatB}\to\dgcatB[shift=1] \),
	the second term vanishes unless~\( n = 1 \)
	(see~\cref{res:cofun_cofree}).
	This allows us to calculate the \( n \)th~component
	\[
		\mult[n,upper=\catainftyfun]
		=
		-
		\momega
		\circ
		\mpr[1]
		\circ
		\dif
		\circ
		\momega[tens=n,spar,inv]
		\colon
		\catainftyfun[nu,tens=n]
		\longto
		\catainftyfun[nu,shift=2-n]
		.
	\]
	Plugging in~\( \momega[tens=n,spar,inv] = (-1)^{n(n-1)/2}\ms[tens=n] \)
	and
	\[
		\ms[tens=n]{
			\malpha[n]
			\sotimes
			\cdots
			\sotimes
			\malpha[1]
		}
		=
		(-1)^{\Sum{(i-1)\dgcatdeg{\malpha[i]}}}
		\ms{\malpha[n]}
		\sotimes
		\cdots
		\sotimes
		\ms{\malpha[1]}
		=
		(-1)^{\Sum{(i-1)\dgcatdeg{\malpha[i]}}}
		\naturaltensor{\ms{\malpha[n]},...,\ms{\malpha[1]}}
	\]
	(see~\cref{res:cofun_cofree}
	for the notation),
	we obtain
	that

\begin{proposition}
		The \ainfty-operations~\(
			\mult[n,upper=\catainftyfun]
		\)
		on~\(
			\catainftyfun[nu]{\dgcatA,\dgcatB}
		\)
		are given by
		\begin{align*}
			\mult[1,upper=\catainftyfun]{\malpha[1]}
				&=
				\textstyle
					\dirsum[k\ge1,limits,operator]{
						\mult[upper=\dgcatB,k]
						\circ
						\momega[tens=k]
					}[spar=\big]
					\circ
					\naturaltensor{\ms{\malpha[1]}}
					-
					(-1)^{\dgcatdeg{\malpha[1]}}
					\malpha[1]
					\circ
					\dif[\Bar{}{\dgcatA}]
		\\
			\mult[n,upper=\catainftyfun]{ \malpha[n],...,\malpha[1] }
				&=
				\textstyle
				(-1)^{n(n-1)/2}
				(-1)^{\Sum{(i-1)\dgcatdeg{\malpha[i]}}}
				\dirsum[k\ge n,limits,operator]{
					\mult[upper=\dgcatB,k]
					\circ
					\momega[tens=k]
				}[spar=\big]
				\circ
				\naturaltensor{
					\ms{\malpha[n]},...,\ms{\malpha[1]}
				}
		\end{align*}
		for~\( n>1 \).
		The notation~\(
			\naturaltensor{ \ms{\malpha[n]},...,\ms{\malpha[1]} }
		\)
		is explained
		in~\cref{res:cofun_cofree}. It must be evaluated
		using the sign conventions with the tensor product of maps of complexes.
\end{proposition}

\begin{corollarybreak}
	\begin{corollarylist}
		\item
		The \ainfty-category \( \catainftyfun[nu]{\dgcatA,\dgcatB} \)~is
		unital if \( \dgcatB \)~is.
		Indeed, the unit
		at~\( \mf \in \catainftyfun[nu]{\dgcatA,\dgcatB} \)
		is the
		composition
		\(
			\Bar{}{\dgcatA}
			\to
			\ringk[\catob{\dgcatA}]
			\xto{
				\mf[qpull]{\meta[\dgcatB]}
			}
			\mf[qpull]{\dgcatB}
		\).
		\item
		If \( \dgcatB \)~is a dg-category, so
		is~\( \catainftyfun[nu]{\dgcatA,\dgcatB} \).
		Indeed, if~\( \mult[n,upper=\dgcatB] = 0 \) for all~\( n>2 \),
		\( \mult[n,upper=\smash{\catainftyfun}] \)~vanishes, too.
		
		Furthermore, if \( \mf \in \catainftyfun[nu]{\dgcatA,\dgcatB} \)~is
		a non-unital \ainfty-functor, we may regard
		it as a non-unital \ainfty-transformation~\( \mf \) of degree~\( 1 \),
		see~\cref{ex:functors_as_transformations}.
		In this case, the equation~\(\smash{
			\dif[\Bar[red]{}{\dgcatB}]
			\circ
			\mf
			-
			\mf
			\circ
			\dif[\Bar[red]{}{\dgcatA}]
			=
			0
		}\)
		is equivalent to
		\[
			\mult[1,upper=\catainftyfun]{\mf}
			=
			\mult[2,upper=\catainftyfun]{
				\mf,
				\mf
			}
		.\]
	\end{corollarylist}
\end{corollarybreak}

	Suppose now that both \( \dgcatA \) and \( \dgcatB \)
	are \emph{unital} \ainfty-categories.
	The \textdef{\ainfty-category of (unital) \ainfty-functors}
	is the non-full subcategory
	\[
		\catainftyfun{ \dgcatA , \dgcatB }
		\subset
		\catainftyfun[nu]{ \dgcatA , \dgcatB }
	\]
	with objects the unital \ainfty-functors.
	The morphism space~\( \catainftyfun { \dgcatA , \dgcatB }{ \mf , \mg } \)
	consists of the morphisms~\(\smash{
		\malpha
		\in
		\Hom[\catquiv[gr] { \catob { \dgcatA } }]{
			\Bar{}{\dgcatA} ,
			{ \map{ \mf \times \mg}[spar,qpull]{{ \dgcatB }} }
		}
	}\)
	satisfying the condition~\(\smash{
		\malpha {{ \id[tens=i] \tens \unitmap \tens \id[tens=j] }}
		=
		0
	}\)
	for all~\( i , j \ge 0 \). Equivalently, we have
	\begin{equation}\label{eq:catainftyfun_transformations}
		\catainftyfun{\dgcatA , \dgcatB }{ \mf , \mg }
		=
		\Hom[\catquiv[gr]{ \catob { \dgcatA } }]{
			\tensorcoalg{{\dgcatA[red,shift=1]}} ,
			{ \map { \mf \times \mg}[spar,qpull] {{ \dgcatB }} }
		}
	\end{equation}
	where \( \dgcatA[red] \)~is the cokernel of the
	unit map~\( \unitmap\colon\ringk[\catob{\dgcatA}]\to\dgcatA \).
	
	\begin{example}\label{ex:A_infty_fun_k_n_to_B}
		Suppose that \( \dgcatB \)~is a unital dg-category,
		and consider the unital dg-category~\( \ordset{n}[linearize=\ringk] \), the \( \ringk \)-linearization of the poset~\( \ordset{n} \), regarded as a category. We
		wish to
		calculate~\(
			\catainftyfun { { \ordset{n}[linearize=\ringk] } , \dgcatB }
		\).
		The object set is
		\begin{align*}
			\MoveEqLeft[0.5]
			\Hom[\catainftycat]{
				{ \ordset{n}[linearize=\ringk] } ,
				\dgcatB 
			}
			=
			\Hom[\catainftycat[nu]]{
				{ \ordset{n}[linearize=\ringk,reduce] } ,
				\dgcatB 
			}
			=
			\Hom[\catcocat[dg,ncu]]{
				\Bar[red]{}{ { \ordset{n}[linearize=\ringk,reduce] } } ,
				\Bar[red]{}{\dgcatB}
			}
		\\
			&=
			\Hom[\catdgcocat[coaug]]{
				\Bar[plus]{}{ { \ordset{n}[linearize=\ringk] } } ,
				\Bar[plus]{}{\dgcatB}
			}
			=
			\Hom[\catdgcat]{
				\Cob[plus]{}{ \Bar[plus]{}{{ \ordset{n}[linearize=\ringk] }} } ,
				\dgcatB
			}
			.
		\end{align*}
		In other words, an object~\( \mf \) consists of
		a collection~\( \vb[0] = \mf{0} , \ldots , \vb[n] = \mf{n} \)
		of objects in~\( \dgcatB \)
		together with a
		collection of 
		morphism~\( \mf[i_0 i_1 \cdots i_k] \colon \vb[i_0] \to \vb[i_k] \)
		of degree~\( 1 - k \)
		for all sequences~\( 0\le i_0 < i_1 < \cdots < i_k \le n\),~\( k>0 \),
		and with differential
		\[
			\dif[par] { \mf[i_0 i_1 \ldots i_k] }
			=
			\Sum[from={j=1},to={k-1}] {{
				(-1)^{k-j}
				\bigl(
					\mf[i_0 \ldots \hat{\imath}_j \ldots i_k]
					-
					\mf[i_j \ldots i_k] \circ \mf[i_0 \ldots i_j]
				\bigr)
			}}
		\]
		with the convention that \( \mf[j_0 j_1 \cdots l_l] = 0 \)
		if the index contains repetitions and~\( l>1 \), and that
		\( \mf[j_0 j_0] = \id \).
		If \( \mf , \mg \colon \Cob[plus]{}{ \Bar[plus]{}{{\ordset{n}[linearize=\ringk]}}} \to \dgcatB \)~are
		two objects,
		\eqref{eq:catainftyfun_transformations}~shows that
		an \ainfty-transformation~\( \malpha \colon \mf \to \mg \)
		of degree~\( \dgcatdeg{\malpha} \)
		is a degree~\( \dgcatdeg{\malpha} \) map of graded quivers
		\[
			\malpha
			\colon
			\tensorcoalg {{
				\ordset{n}[linearize=\ringk,reduce,return,shift=1]
			}}
			\to
			\map{\mf\times\mg}[spar,qpull] { \dgcatB }
		,\]
		which is the data
		of a
		map
		\[
			\malpha[i_0 i_1 \cdots i_k]
			=
			\malpha{ \tup{ i_0,...,i_k } }
			\in
			\Hom[\dgcatB,d=\dgcatdeg{\malpha}-k]{ \mf{ i_0 } , \mg{ i_k } }
		\]
		for each sequence~\( 0\le i_0 < i_1 < \cdots < i_k\le n \),~\(
			k\ge0
		\).
		Also, we obtain that the composition
		map~\(
			\map{ \mbeta\circ\malpha }
			:=
			\mult[2]{{ \mbeta \tens \malpha }}
		\)
		is
		\begin{align*}
			\map { \mbeta\circ\malpha }[spar,i_0 i_1 \cdots i_k]
			&=
			-(-1)^{\dgcatdeg{\mbeta}}
			( \mult[upper=\dgcatB,2] \circ \momega[tens=2] )
			\circ
			\naturaltensor{\ms{\mbeta},\ms{\malpha}}[i_0,\ldots,i_k]
		\\
			&=
			-(-1)^{\dgcatdeg{\mbeta}}
			\map{ { \mult[2,upper=\dgcatB] \circ \momega[tens=2] } }[spar,par=\big]{{
				\Sum[from={j=0},to={k}]{
					(-1)^{ \dgcatdeg{\ms{\malpha}}(k-j) }
					\ms{\mbeta}[i_j,...,i_k]
					\tens
					\ms{\malpha}[i_0,...,i_j]
				}
			}}
		\\
			&=
			-(-1)^{\dgcatdeg{\mbeta}}
			\Sum[from={j=0},to={k}] {
				(-1)^{ \dgcatdeg{\ms{\malpha}} (k-j) }
				(-1)^{ \dgcatdeg{ { \ms{\mbeta}[i_j,...,i_k] } } (-1) }
				\mbeta[i_j \cdots i_k]
				\circ
				\malpha[i_0 \cdots i_j]
			}
		\\
			&=
			-(-1)^{\dgcatdeg{\mbeta}}
			\Sum[from={j=0},to={k}] {
				(-1)^{ (\dgcatdeg{\malpha} - 1) (k-j) }
				(-1)^{ ( \dgcatdeg{ \mbeta } - (k-j) - 1 ) (-1) }
				\mbeta[i_j \cdots i_k]
				\circ
				\malpha[i_0 \cdots i_j]
			}
		\\
			&=
			\Sum[from={j=0},to={k}] {
				(-1)^{(k-j)\dgcatdeg{\malpha}}
				\mbeta[i_j \cdots i_k]
				\circ
				\malpha[i_0 \cdots i_j]
			}
			.
		\end{align*}
		A similar calculation shows that
		the differential~\( \dif{\malpha} = \mult[1]{\malpha} \) is
		\begin{align*}
			\dif { \malpha }[spar,i_0 \cdots i_k]
			&=
			\dif[\dgcatB,par] { \malpha[i_0 \cdots i_k] }
			-(-1)^{ \dgcatdeg{\malpha} }
			\malpha[par=\big]{
				\dif[par,\Bar{}{{\ordset{n}[linearize=\ringk,return,red]}}]{
					\tup{ i_0,...,i_k }
				}
			}
		\\
			&
					\qquad{}+
					( \mult[2,upper=\dgcatB]\circ\momega[tens=k] )
					\bigl(
						\map{ \ms{\mg}\tens\ms{\malpha} }[spar,i_0\cdots i_k]
						+
						\map{ \ms{\malpha}\tens\ms{\mf} }[spar,i_0\cdots i_k]
					\bigr)
		\\
			&=
			\dif[\dgcatB,par] { \malpha[i_0 \cdots i_k] }
			+
			\Sum[from={j=1},to={k-1}]
				\bigl(
					(-1)^{\dgcatdeg{\malpha}}
					(-1)^{k-j}
					\malpha[i_0 \cdots \hat{\imath}_j\cdots i_k]
		\\
			&
					\qquad{}+
					(-1)^{(k-j)\dgcatdeg{\malpha}}
					\mg[i_j\cdots i_k]\circ\malpha[i_0\cdots i_j]
					-
					(-1)^{\dgcatdeg{\malpha}}
					(-1)^{k-j}
					\malpha[i_j\cdots i_k]\circ\mf[i_0\cdots i_j]
				\bigr)
		.
		\end{align*}
		All higher \ainfty-operations vanish, so that we have an honest dg-category.
		We may simplify the last formula by regarding \( \mf \) and~\( \mg \) as morphisms \( \mf \to \mf \) resp.~\( \mg \to \mg \) of degree~\( 1 \). Then the composition operation defined above makes sense for \( \mf \) and~\( \mg \) as well,
		and we obtain
		\begin{align*}
			\dif{\malpha}
			&=
			\dif[\dgcatB]\circ\malpha
			-
			(-1)^{\dgcatdeg{\malpha}}
			\malpha\circ\dif[\Bar{}{{\ordset{n}[linearize=\ringk,red]}}]
			+
			\mg\circ\malpha
			-
			(-1)^{\dgcatdeg{\malpha}}
			\malpha\circ\mf
		\\
			0
			&=
			\dif[\dgcatB]
			\circ
			\mf
			+ \mf \circ \dif[\Bar[red]{}{{\ordset{n}[linearize=\ringk,red]}}]
			+\mf\circ\mf
			.
		\end{align*}
		In this case, as \( \mf \)~factors
		through~\( \Bar{}{{\ordset{n}[linearize=\ringk,red]}} \onto \Bar[red]{}{{\ordset{n}[linearize=\ringk,red]}} \),
		expressions like~\( \mf[i_0] \)
		with a single index must be interpreted as zero.
	\end{example}

\endgroup

\begingroup

\chapter{Homotopy limits in dg-categories}\label{chap:holim_dgcat}

%

	Given a dg-category~\( \dgcatA \), we denote by~\( \dgcatA[z0] \) the dg-category with the same objects as~\( \dgcatA \), but with hom complexes given by
	\(
		\dgcatA[z0,spar]{ \vx , \vy } = \cocy{0}{ \dgcatA{ \vx , \vy } }
	\),
	the set of closed maps~\( \vx \to \vy \) in~\( \dgcatA \). We define~\( \dgcatA[h0] \) analogously.
	A map~\( \mf \in \dgcatA[d=0]{ \vx , \vy } \) in~\( \dgcatA \) is called a \textdef{homotopy equivalence} if its image in~\( \dgcatA[h0] \) is an isomorphism. As in homological algebra, this amounts the existence of a map \( \mg \in \dgcatA[d=0]{ \vy , \vx } \) in the other direction along with correcting morphisms \( \mr[\vx] \in \dgcatA[d=-1]{ \vx , \vx } \) and \( \mr[\vy] \in \dgcatA[d=-1]{ \vy , \vy } \) such that
	\[
		\dif{ \mr[\vx] } = \mg \mf - \id[\vx]
		\qquad\text{and}\qquad
		\dif{ \mr[\vy] } = \mf \mg - \id[\vy].
	\]

	We recall from
	\textcite{tabuada,tabuada_dgcats_vs_simplicial_cats}
	the existence of a model structure on the \( 1 \)-category~\( \catdgcat \) with
	\begin{itemize}
		\item weak equivalences given by \textdef{quasi-equivalences}, i.e.
		functors~\( \mF \colon \dgcatA \to \dgcatB \) such that
		\begin{enumerate}[(i)]
		\item the induced functor~\( \mF[parent=alg,h0] \colon \dgcatA[h0] \to \dgcatB[h0] \) is an equivalence of categories in the classical, non-enriched sense;
		\item for all~\( \vx, \vy \in \dgcatA \),
		\( \mF[\vx, \vy] \colon
		\dgcatA { \vx , \vy } \to \dgcatB { \mF { \vx } , \mF { \vy } } \)
		is a quasi-isomorphism of chain complexes;
		\end{enumerate}
		\item fibrations given by dg-functors~\( \mF \colon \dgcatA \to \dgcatB \) such that
		\begin{enumerate}[(i)]
			\item the induced map
			\(
				\mF \colon
				\dgcatA[*]{ \vx, \vy }
				\to
				\dgcatB[*]{ \mF{ \vx } , \mF{ \vy } }
			\)
			is surjective for all objects~\( \vx, \vy \in \dgcatA \), and
			\item any map \( \mf \colon \mF { \vx } \to \vy \) in~\( \dgcatB \)
			that becomes an isomorphism in~\( \Hdgcat { \dgcatB } \) lifts as \( \mf = \mF { \mg } \) for some morphism \( \mg\colon \vx \to \vx[prime] \) in~\( \dgcatA \)
			that becomes an isomorphism in~\( \Hdgcat { \dgcatA } \).
		\end{enumerate}
	\end{itemize}
	This model structure is combinatorial, see
	\textcite[Proposition~1.3.1.19]{ha}.

We shall use the tools developed in the preceding chapters to develop a notion of homotopy descent of quasi-coherent sheaves on affine dg-schemes.
Much of the inspiration is from \textcite{explicit};
however, using the tools developed earlier
(specifically~\cref{ex:A_infty_fun_k_n_to_B}),
we are able to solve their Conjecture~1 and prove their
results in complete generality.


Recall from \cref{res:fat_tot} that
the homotopy limit of a \( \catdelta \)-diagram in the model category~\( \catdgcat \) is given by
\[\textstyle
	\invholim[\catdelta] { \cosimpdgcatA{n} }
	=
	\invholim[\catdelta[plus]] { \cosimpdgcatA{n} }
	=
	\End[\ordset{n}\in\catdelta[plus]] { \fibrep[deltaplus]{ \cosimpdgcatA{n} }{n} }
\]
where \( \fibrep[deltaplus] \colon \catdgcat\to\catdgcat[diag={\catdelta[plus,op]},inj,smash] \) is a fibrant replacement functor, taking a dg-category~\( \dgcatB \) to an injectively fibrant replacement of the constant \( \catdelta[plus] \)-diagram at~\( \dgcatB \).
We obtain such a functor from Holstein, Poliakova and Arkhipov:

\begin{theorembreak}[{Theorem (\cite[section~3]{holstein} and \cite{holstein_note})}]\label{res:holstein_resolution}
	If \( \dgcatB \)~is a dg-category, then
	\( \simpdgcatB{*} \in \catdgcat[diag={\catdelta[plus,op]}] \)
	given by
	\[
		\simpdgcatB{n}
		= \catainftyfun[o]{ {\ordset{n}[linearize=\ringk]} , \dgcatB }
	\]
	is an injectively fibrant replacement of the constant \( \catdelta[plus] \)-diagram at~\( \dgcatB \). Here,
	the~“\( \mathbf{\circ} \)”
	means that the objects are \ainfty-functors~\( \mf \) such
	that~\( \co[nopar]{0} { \mf } \colon \ordset{n}[lin=\ringk]\to\dgcatB[h0] \)
	sends non-zero maps to isomorphisms.
	Equivalently, in the notation of~\cref{ex:A_infty_fun_k_n_to_B},
	\( \co[nopar]{0}{\mf[ij]} \)~is an isomorphism for all~\( i<j \).
\end{theorembreak}


	In other words,
	if \( \cosimpdgcatA{*} \)~is a cosimplicial system of dg-categories,
	then its homotopy limit is given by
	\[\textstyle
		\invholim[\catdelta] { \cosimpdgcatA{n} }
		=
		\End[\ordset{n}\in\catdelta[plus]] {{
			\catainftyfun[o]{ {\ordset{n}[linearize=\ringk]}  , \cosimpdgcatA{n} }
		}}
		.
	\]
We wish to evaluate this expression more explicitly:

\begin{proposition}\label{res:homotopy_limit_dgcat}
	Suppose~\( \cosimpdgcatA{*} \) is a cosimplicial system of dg-categories. Then the homotopy limit
	\(
		\invholim[\catdelta] { \cosimpdgcatA{*} }
	\)
	is the dg-category with objects~\( \tup { \algM , \mtheta } \)
	where~\(
		\mtheta = \tup { {\mtheta[n]} }[n\ge 1]
	\)
	is a collection of morphisms~%
	\(
		\mtheta[n]
		\in
		\Hom[\cosimpdgcatA{n} , d=1-n]{
			{ \coface[max=n,nopar] { \algM } } ,
			{ \coface[min=n,nopar] { \algM } }
		}
	\),
	satisfying
	\begin{equation}\label{eq:theta_n_equation}
		\dif[par] { {\mtheta[n]} }
		=
		-
		\map{\mtheta\circ\mtheta}[spar,n]
		+
		\Sum[from={i=1},to={n-1}] {{
			(-1)^{n-i}
				\coface[d=i] {{ \mtheta[n-1] }}
		}}
	\end{equation}
	and with \( \mtheta[1] \)~invertible in~\( \co{0} { \cosimpdgcatA{1} } \).
	The composition~\( \mtheta\circ\mtheta \)
	is evaluated via~\eqref{res:morphisms_in_tot_composition}
	below with the conventions~\( \dgcatdeg{\mtheta} = 1 \)
	and~\( \mtheta[0] = 0 \).
	A morphism~\(
		\malpha
		\colon
		\tup{\algM,\mtheta}
		\to
		\tup{\algN,\meta}
	\)
	of degree~\( \dgcatdeg{\malpha} \)
	is a collection~\(
		\malpha = \tup { \malpha[n] }[n\ge0]
	\)
	of morphisms~\(
		\malpha[n]
		\in
		\Hom[\cosimpdgcatA{n},d=\dgcatdeg{\malpha}-n]{
			{ \coface[max=n,nopar] { \algM } } ,
			{ \coface[min=n,nopar] { \algN } }
		}[smash]
	\).
	The composition of two morphisms is
	\begin{equation}\label{res:morphisms_in_tot_composition}
		\map{\mbeta\circ\malpha}[spar,n]
		=
		\Sum[from={i=0},to={n}] {
			(-1)^{\dgcatdeg{\malpha}(n-i)}
			\coface[min=i,nopar] { \mbeta[n-i] }
			\circ
			\coface[max=n-i,nopar] { \malpha[i] }
		}
	.\end{equation}
	The differential on a morphism is
	\begin{equation}\label{res:morphisms_in_tot_differential}
			\dif[par]{\malpha}[n]
			=
			\dif[par]{\malpha[n]}
			+
			\map{\meta\circ\malpha}[spar,n]
			-
			(-1)^{\dgcatdeg{\malpha}}
			\map{\malpha\circ\mtheta}[spar,n]
			+
			(-1)^{\dgcatdeg{\malpha}}
			\Sum[from={j=1},to={n-1}]
					(-1)^{n-j}
					\coface[d=j,nopar] { \malpha[n-1] }
					.
%
	\end{equation}
\end{proposition}

This proves Conjecture~1 in \textcite{explicit}.

\begin{proof}[Proof of \cref{res:homotopy_limit_dgcat}]
	Using \citeauthor{holstein}'s theorem, we obtain the formula
	\begin{align*}
		\MoveEqLeft
		{\textstyle \invholim[\catdelta] { \cosimpdgcatA{n} } }
		=
		\End[\ordset{n}\in\catdelta[plus]] {
			\catainftyfun[o]{
				{\ordset{n}[linearize=\ringk]} ,
				\cosimpdgcatA{n}
			}
		}
	\\
			& = \Eq[par=\Big]{{
				\Prod[\ordset{n}\in\catdelta[plus]] {
					\catainftyfun[o]{
						{\ordset{n}[linearize=\ringk]} ,
						\cosimpdgcatA{n}
					}
				}
				\rightrightarrows
				\Prod[\ordset{n}\into\ordset{m}] {
					\catainftyfun[o]{
						{\ordset{n}[linearize=\ringk]} ,
						\cosimpdgcatA{m}
					}
				}
			}}.
	\end{align*}
	We initially note that the set of objects
	of~\( \catainftyfun[o]{ {\ordset{n}[linearize=\ringk]} , \cosimpdgcatA{n} } \)
	is the subset of the hom set~\( \Hom[\catdgcat]{ \Cob[plus]{}{\Bar[plus]{}{{\ordset{n}[linearize=\ringk]}}} , \cosimpdgcatA{n} } \)
	of functors satisfying the homotopy invertibility condition.
	Thus an object of this equalizer is a
	collection~\( \mF = \tup { {\mF[d=n]} }[\ordset{n}\in\catdelta[plus]] \)
	of dg-functors~\(
		\mF[d=n]
		\colon
		\Cob[plus]{}{\Bar[plus]{}{{\ordset{n}[linearize=\ringk]}}}
		\to
		\cosimpdgcatA{n}
	\)
	satisfying~\(
		\mvarphi[push] \circ \mF[d=n] = \mF[d=m] \circ \mvarphi[push]
	\)
	for all~\( \mvarphi\colon\ordset{n}\into\ordset{m} \)
	in~\( \catdelta[plus] \) (in an enriched sense, one may say that it consists of natural transformations between the \( \catdelta[plus] \)-diagrams~\( \Cob[plus]{}{\Bar[plus]{}{{\ordset{\smallbullet}[linearize=\ringk]}}}\to\cosimpdgcatA{*} \)).
	Now if~\( 0 \le i \le n \), we may consider
	the map \( \mvarphi \colon \ordset{0}\into\ordset{n} \) taking \( 0 \) to~\( i \). Then
	\(
		\mF[d=n]{ i }
		= \mF[d=n]{ \mvarphi[push]{ 0 } }
		= \mvarphi[push]{ {\mF[d=0]{ 0 }} }.
	\)
	Thus all the functors~\( \mF[d=n] \) are determined on the object level by the one object~\( \algM := \mF[d=0]{ 0 } \in \cosimpdgcatA{0} \).
	Similarly, on the morphism level, \( \Cob[plus]{}{\Bar[plus]{}{{\ordset{n}[linearize=\ringk]}}} \)~is freely generated by the morphisms~\( \mf[i_0 i_1 \cdots i_l] \colon i_0 \to i_l \)
	with~\( 0 \le i_0 < i_1 < \cdots < i_l \le n \). But if we
	consider the map~\( \mvarphi \colon \ordset{l} \into \ordset{n} \) given
	by~\( \mvarphi{j} = i_j \), we have
	\( \mf[i_0 i_1 \cdots i_l,smash] = \mvarphi[push]{ \mf[01\cdots l] } \).
	In other words, \( \mF \)~is determined on the object level by what it does to the non-degenerate morphism~\( \mf[01\cdots l] \in \Cob[plus]{}{\Bar[plus]{}{{\ordset{l}[linearize=\ringk]}}} \).
	Furthermore, these non-degenerate morphisms may be chosen freely.
	Summarizing, the objects of this equalizer consist of the data of an element~\( \algM \in \cosimpdgcatA{0} \) and for each~\( n \) a morphism
	\[
		\mtheta[n]
		:=
		\mF[d=n]{ \mf[01\cdots n] }
		\colon
		\face[max=n,shpush]{ \algM }
		\longto
		\face[min=n,shpush]{ \algM }
		.
	\]
	The differential is as stated since this is the image of the differential on~\( \mf[01\cdots n] \).
	
	
	
	On the morphism level, suppose
	\( \mF[d=n] \) and~\( \mG[d=n] \)
	are objects of the equalizer.
	A morphism~\( \mF[*] \to \mG[*] \)
	of degree~\( d \)
	consists of a
	collection~\(
		\malpha[*] = \tup { {\malpha[d=n]} }[\ordset{n}\in\catdelta[plus]]
	\)
	of morphisms \( \malpha[d=n] \colon \mF[d=n] \to \mG[d=n] \)
	in~\( \catainftyfun[o]{ {\ordset{n}[linearize=\ringk]},\cosimpdgcatA{n} } \) that simultaneously lie in the equalizer.
	From \cref{ex:A_infty_fun_k_n_to_B}, we therefore
	get that the space of non-identity transformations~\( \mF[*]\to\mG[*] \) is
	\begin{align*}
		\MoveEqLeft
		\End{
			\Hom[*,par=\big]{
				\tensorcoalg{{
					\ordset{n}[linearize=\ringk,reduce,return,shift=-1]
				}},
				\map{\mF[d=n]\times\mG[d=n]}[spar,qpull] {
					\cosimpdgcatA{n}
				}
			}
		}
		=
		\End{
			\Hom[*,par=\big,
			]{
				\textstyle\dirsum[limits,l\ge0,return,operator] { {\ordset{n}[linearize=\ringk,reduce,return,tens=l,shift=-l]} }
				,
				{
					\map{\mF[d=n]\times\mG[d=n]}[spar,qpull] {
						\cosimpdgcatA{n}
					}
				}
			}
		}
	\\
		&=
		\End{
			\Prod[l\ge0] {
				\Prod[0\le i_0<\cdots<i_l\le n] {
					\Hom[*,\ringk,par=\big
					]{
						\ringk
						\mf[i_{l-1} i_l]
						\sotimes
						\cdots
						\sotimes
						\mf[ i_0 i_1 ]
						,
						{\cosimpdgcatA{n}[arg={
							{\mF[d=n] { i_0 } } ,
							{\mG[d=n] { i_l } }
						}]}
					 }
				}
			}
		}
		.
	\end{align*}
	Thus the transformation~\( \malpha[*] \) is freely determined
	by what it does to the
	non-degenerate elements~\( \mf[n-1,n] \sotimes \cdots \sotimes \mf[0,1] \),
	i.e.\ by the
	elements
	\[
		\malpha[d=n]{{ \mf[n-1,n] \sotimes \cdots \sotimes\mf[0,1] }}
		\in
		\cosimpdgcatA{n}[arg={
			{\mF[d=n] { i_0 } } ,
			{\mG[d=n] { i_n } }
		}]
		.
	\]
	Calling this element~\( \malpha[n] \),
	we obtain the desired description.
\end{proof}

\section{Homotopy descent of dg-schemes}

If \( \ringk \)~is a field, the category of \textdef{affine dg-schemes}
over~\( \ringk \)
is the opposite category~\( \catdgaff = \catdgalg[le0,com,spar,smash,op] \)
to the category of non-positively graded, graded commutative dg-algebras
over~\( \ringk \).
If \( \dgschemeX \)~is a dg-scheme,
the associated dg-algebra is denoted~\( \dgcatA[\dgschemeX] \).
The category of \textdef{quasi-coherent dg-sheaves}
on~\( \tup{ \dgschemeX , \dgcatA[\dgschemeX] } \)
is the category~\( \catxdgmod{\dgcatA[\dgschemeX]} \)
of \( \dgcatA[\dgschemeX] \)-dg-modules.

\begin{theorem}\label{res:A_infty_comodules}
	Suppose that
	\( \dgschemeX[1] \rightrightarrows \dgschemeX[0] \)
	is a groupoid in affine dg-schemes,
	and consider the associated classifying space~\( \dgschemeX[ibullet] \) given by
	\[\textstyle
		\dgschemeX[n]
		=
		\dgschemeX[1]
		\fibre[\dgschemeX[0]]
		\dgschemeX[1]
		\fibre[\dgschemeX[0]]
		\cdots
		\fibre[\dgschemeX[0]]
		\dgschemeX[1]
		.
	\]
	Write~\( \cosimpdgcatA{n} = \dgcatA[\dgschemeX[n],smash] \)
	for the associated cosimplicial system of dg-algebras.
	Let~\( \dgcatA = \cosimpdgcatA{0} \)
	and~\( \dgcatC = \cosimpdgcatA{1} \),
	and note that \( \dgcatC \)~is a
	counital
	coalgebra in~\( \catxdgmodx{\dgcatA}{\dgcatA} \)
	via the map~\(
		\comult
		=
		\face[d=1,comor]
		\colon
		\dgcatC
		\to
		\dgcatC\tens[\dgcatA]\dgcatC
	\).
	Then we have an equivalence of dg-categories
	\[\textstyle
		\invholim[\catdelta]{
				\catdgqcoh{ \dgschemeX[ibullet] }
		}
		\cong
		\catxcomod{\dgcatC}[infty,hcu,formal]{\dgcatA}
		,
	\]
	where the \anrhs denotes the dg-category
	of formal, homotopy-counital \ainfty-comodules
	over~\( \dgcatC \)
	in~\( \catxdgmod{\dgcatA} \).
\end{theorem}

\begin{proof}[Proof of \cref{res:A_infty_comodules}]
	\Cref{res:homotopy_limit_dgcat} gives us the general form
	of the homotopy limits of~\( \cosimpdgcatA{*} \),
	so we use the same notation.
	Thus an object of the homotopy limit
	is a pair~\( \tup{\algM,\mtheta} \),
	where
	\( \mtheta = \tup{\mtheta[n]}[n\ge1] \)~is
	a tuple of maps~\(
		\mtheta[n]
		\colon
		\face[max=n,shpull] { \algM }
		\to
		\face[min=n,shpull] { \algM }
	\) of degree~\( 1-n \),
	and a morphism~\(
		\malpha
		\colon
		\tup{\algM,\mtheta}
		\to
		\tup{\algN,\meta}
	\)
	is a collection~\(
		\malpha = \tup{\malpha[n]}[n\ge0]
	\)
	of maps~\(
		\malpha[n]
		\colon
		\face[max=n,shpull]{\algM}
		\to
		\face[min=n,shpull]{\algN}
	\)
	of degree~\( \dgcatdeg{\malpha} - n \).
	For such a morphism,
	we scale it to define a new morphism~\( \malpha[tilde] \) by
	\(
		\malpha[tilde,n]
		=
		(-1)^{\dgcatdeg{\malpha}(n+1)}
		\malpha[n]
	\).
	Regarding this~\( \mtheta \) as a degree~\( 1 \)
	morphism~\( \tup{\algM,\mtheta}\to\tup{\algM,\mtheta} \),
	we use the same notation as for morphisms and
	write~\(
		\mtheta[tilde,n]
		=
		(-1)^{n+1}
		\mtheta[n]
	\) for~\( n\ge 1 \).
	We notice that \eqref{res:morphisms_in_tot_composition} becomes
	\[
		\map{\mbeta\circ\malpha}[spar,widetilde,n]
		=
		\Sum[from={i=0},to={n}]{
			(-1)^{\dgcatdeg{\mbeta}i}
			\face[min=i,shpull]{{ \mbeta[tilde,n-i] }}
			\circ
			\face[max=(n-i),shpull]{{ \malpha[tilde,i] }}
		}.
	\]
	Thus
	\eqref{eq:theta_n_equation}~becomes
	\begin{align*}
		(-1)^{n+1}
		\dif[ {\face[min=n,shpull]{\algM}} ]
		\circ
		\mtheta[tilde,n]
		-
		\mtheta[tilde,n]
		\circ
		\dif[ {\face[max=n,shpull]{\algM}} ]
		=
		-
		\map{\mtheta\circ\mtheta}[widetilde,spar,n]
		+
		\Sum[from={i=1},to={n-1}]{
			(-1)^{i}
			\face[d=i,shpull]{{ \mtheta[tilde,n-1] }}
		}
		.
	\end{align*}
	Since~\(
		\dif[par]{\malpha}[widetilde,n]
		=
		(-1)^{(\dgcatdeg{\malpha}+1)(n+1)}
		\dif[par]{\malpha}[n]
	\),
	\eqref{res:morphisms_in_tot_differential}~becomes
	\begin{align*}
			\dif[par]{\malpha}[widetilde,n]
			&
			=
			(-1)^{n+1}
			\dif[ {\face[min=n,shpull]{\algN}} ]
			\circ
			\malpha[tilde,n]
			+
			(-1)^{\dgcatdeg{\malpha}}
			\malpha[tilde,n]
			\circ
			\dif[ {\face[max=n,shpull]{\algM}} ]
			+
			\map{\meta\circ\malpha}[widetilde,spar,n]
	\\
			&
			\qquad
			{}
			-
			(-1)^{\dgcatdeg{\malpha}}
			\map{\malpha\circ\mtheta}[widetilde,spar,n]
			+
			\Sum[from={j=1},to={n-1}]
				(-1)^{j-1}
				\face[d=j,shpull] {{ \malpha[tilde,n-1] }}
	\\
			&=
			(-1)^{n+1}
			\dif[ {\face[min=n,shpull]{\algN}} ]
			\circ
			\malpha[tilde,n]
			+
			\Sum[from={i=0},to={n-1}]{
				(-1)^{i}
				\face[min=i,shpull]{{
					\meta[tilde,n-i]
				}}
				\circ
				\face[max=(n-i),shpull]{{
					\malpha[tilde,i]
				}}
			}
	\\
			&\qquad
			{}
			+
			(-1)^{\dgcatdeg{\malpha}}
			\malpha[tilde,n]
			\circ
			\dif[ {\face[max=n,shpull]{\algM}} ]
			-
			\Sum[from={i=1},to=n]{
				(-1)^{\dgcatdeg{\malpha}(i+1)}
				\face[min=i,shpull]{{ \malpha[tilde,n-i] }}
				\circ
				\face[max=(n-i),shpull]{{ \mtheta[tilde,i] }}
			}
	\\
			&\qquad
			{}
			+
			\Sum[from={j=1},to={n-1}]
				(-1)^{j-1}
				\face[d=j,shpull] {{ \malpha[tilde,n-1] }}
	,\end{align*}

	Now for an object~\( \tup{\algM,\mtheta} \),
	we add to the collection~\( \tup{\mtheta[n]}[n\ge1] \)
	the element~\( \mtheta[0] = -\dif[\algM] \).
	We then
	define the comodule
	operations~\( \ca[n] \colon \algM \to \dgcatC[tens=(n-1)]\tens\algM \)
	of degree~\( 2 - n \)
	by
	\[\begin{tikzcd}
		\ca[n]
		\colon
		\algM
		\ar[r]
		&
		\face[max=n-1,shpush] {{ \face[max=(n-1),shpull] { \algM } }}
		\ar[r,"{\mtheta[tilde,n-1]}"]
		&
		\face[max=n-1,shpush] {{ \face[min=(n-1),shpull] { \algM } }}
	\end{tikzcd}\]
	where the first map is the unit of adjunction.
	In particular, \( \ca[1] = -\dif[\algM] \).	
	Similarly, if
	\( \malpha \colon \tup{\algM,\mtheta} \to \tup{\algN,\meta} \)
	is a map of degree~\( \dgcatdeg{\malpha} \),
	we associate to it
	the map of \ainfty-comodules~\(
		\mf \colon \algM \to \algN
	\)
	of degree~\( \dgcatdeg{\mf} = \dgcatdeg{\malpha} \),
	where
	\(
		\mf[n]
		\colon
		\algM
		\to
		\dgcatC[tens=(n-1)]\tens[\dgcatA] \algN
	\)
	is the map of degree~\(
		\dgcatdeg{\malpha} + 1 - n
		=
		\dgcatdeg{\mf} + 1 - n
	\)
	given by
	\[\begin{tikzcd}
		\mf[n]
		\colon
		\algM
		\ar[r]
		&
		\face[max=n-1,shpush] {{ \face[max=(n-1),shpull] { \algM } }}
		\ar[r,"{\malpha[tilde,n-1]}"]
		&
		\face[max=n-1,shpush] {{ \face[min=(n-1),shpull] { \algN } }}
		.
	\end{tikzcd}\]
	It is immediate that the composition map above
	agrees with the composition of maps of \ainfty-comodules.
	Via~\cref{res:add_comultiplication_map,res:base_change_face_maps},
	the other two equations above now become
	\begin{align*}
		\MoveEqLeft
		(-1)^{n+1}
		(
			\dif[\dgcatC\tens\cdots\tens\dgcatC]
			\tens
			\id[\algM]
			-
			\id[\dgcatC\tens\cdots\tens\dgcatC]
			\tens
			\ca[1]
		)
		\ca[n+1]
		+
		\ca[n+1]
		\ca[1]
	\\
		&
		=
		-
		\Sum[from={i=1},to={n-1}]{
			(-1)^{i}
			( \id[ {\dgcatC[tens=i]} ]\tens\ca[n-i+1] )
			\ca[i+1]
		}
	\\
		&
		\qquad
		{}
		+
		\Sum[from={i=1},to={n-1}]{
			(-1)^{i}
			(
				\id[tens=(i-1)]
				\tens
				\comult[\dgcatC]
				\tens
				\id[tens=(n-i)]
			)
			\ca[n]
		}
	\\
	\shortintertext{and}
		\dif{\mf}[spar,n]
		&
		=
		(-1)^{n+1}
		(
			\dif[\dgcatC\tens\cdots\tens\dgcatC]
			\tens
			\id[N]
			-
			\id[\dgcatC\tens\cdots\tens\dgcatC]
			\tens
			\ca[1]
		)
		\mf[n+1]
		+
		\Sum[from={i=0},to={n-1}]{
			(-1)^{i}
			\ca[n-i+1] \mf[i+1]
		}
	\\
		&
		\qquad
		{}
		+
		(-1)^{\dgcatdeg{\malpha}}
		\mf[n+1]\ca[1]
		-
		\Sum[from={i=1},to={n}]{
			(-1)^{\dgcatdeg{\malpha}(i+1)}
			\mf[n-i+1] \ca[i+1]
		}
	\\
		&
		\qquad
		{}
		+
		\Sum[from={j=1},to={n-1}]{
			(-1)^{j-1}
			(
				\id[tens=(j-1)]
				\tens
				\comult[\dgcatC]
				\tens
				\id[tens=(n-j)]
			)
			\ca[n]
		},
	\end{align*}
	which are exactly the \ainfty-comodule equations from earlier.	

	To verify that we get exactly \emph{homotopy-counital} \ainfty-comodules,
	we notice that the
	equation~\eqref{eq:theta_n_equation}
	in the case~\( n = 2 \) yields that
	\[
		\dif[par]{\mtheta[2]}
		=
		\face[d=0,shpull]{\mtheta[1]}
		\circ
		\face[d=2,shpull]{\mtheta[1]}
		-
		\face[d=1,shpull]{\mtheta[1]}
		,
	\]
	so that we have a homotopy~\(
		\face[d=0,shpull]{\mtheta[1]}
		\circ
		\face[d=2,shpull]{\mtheta[1]}
		\simeq
		\face[d=1,shpull]{\mtheta[1]}
	\).
	Then one may adjust the proof of \cref{res:cocycle_iff_degen=0},
	replacing equalities by homotopies,
	shows that \( \mtheta[1] \)~being an isomorphism is equivalent to~\( \degen[d=0,shpull]{\mtheta[1]}\simeq\id \).
	Similarly, one adjusts the
	last part of the proof of \cref{res:descent=comodules}
	to obtain that~\(
		(\counit\tens\id[\algM])\ca[2]
		\simeq
		\id
	\),
	which means that \( \algM \)~is homotopy-counital.
\end{proof}

\endgroup





\setcounter{biburllcpenalty}{7000}
\setcounter{biburlucpenalty}{8000}

\printbibliography

\end{document}